\documentclass{amsart}
\usepackage{amssymb}
\usepackage{amscd}
\usepackage{verbatim}
\usepackage{epsfig}
\begin{document}
\newcommand{\pa}{\partial}
\newcommand{\CI}{C^\infty}
\newcommand{\dCI}{\dot C^\infty}
\newcommand{\supp}{\operatorname{supp}}
\renewcommand{\Box}{\square}
\newcommand{\ep}{\epsilon}
\newcommand{\Ell}{\operatorname{Ell}}
\newcommand{\WF}{\operatorname{WF}}
\newcommand{\WFb}{\operatorname{WF}_{\bl}}
\newcommand{\diag}{\mathrm{diag}}
\newcommand{\sign}{\operatorname{sign}}
\newcommand{\sH}{\mathsf{H}}
\newcommand{\codim}{\operatorname{codim}}
\newcommand{\Id}{\operatorname{Id}}
\newcommand{\cl}{{\mathrm{cl}}}
\newcommand{\piece}{{\mathrm{piece}}}
\newcommand{\bl}{{\mathrm b}}
\newcommand{\Psib}{\Psi_\bl}
\newcommand{\Psibc}{\Psi_{\mathrm{bc}}}
\newcommand{\Psibcc}{\Psi_{\mathrm{bcc}}}
\newcommand{\RR}{\mathbb{R}}
\newcommand{\NN}{\mathbb{N}}
\newcommand{\codimY}{k}
\newcommand{\dimX}{n}
\newcommand{\cS}{\mathcal S}
\newcommand{\cF}{\mathcal F}
\newcommand{\cL}{\mathcal L}
\newcommand{\cH}{\mathcal H}
\newcommand{\cG}{\mathcal G}
\newcommand{\cU}{\mathcal U}
\newcommand{\loc}{{\mathrm{loc}}}
\newcommand{\comp}{{\mathrm{comp}}}
\newcommand{\Tb}{{}^{\bl}T}
\newcommand{\Vf}{\mathcal V}
\newcommand{\Vb}{{\mathcal V}_{\bl}}
\newcommand{\etat}{\tilde\eta}

\setcounter{secnumdepth}{3}
\newtheorem{lemma}{Lemma}[section]
\newtheorem{prop}[lemma]{Proposition}
\newtheorem{thm}[lemma]{Theorem}
\newtheorem{cor}[lemma]{Corollary}
\newtheorem{result}[lemma]{Result}
\newtheorem*{thm*}{Theorem}
\newtheorem*{prop*}{Proposition}
\newtheorem*{cor*}{Corollary}
\newtheorem*{conj*}{Conjecture}
\numberwithin{equation}{section}
\theoremstyle{remark}
\newtheorem{rem}[lemma]{Remark}
\newtheorem*{rem*}{Remark}
\theoremstyle{definition}
\newtheorem{Def}[lemma]{Definition}
\newtheorem*{Def*}{Definition}

\title{Diffraction from conormal singularities}
\author[Maarten de Hoop, Gunther Uhlmann and Andras Vasy]{Maarten de
  Hoop,
Gunther Uhlmann and Andr\'as Vasy}
\date{April 3, 2012}
\address{Department of Mathematics, Purdue University, 
West Lafayette, IN 47907}
\email{mdehoop@math.purdue.edu}
\address{Department of Mathematics, University of Washington, 
Seattle, WA 98195-4350, and Department of Mathematics,
University of California, Irvine, 340 Rowland
Hall, Irvine CA 92697}
\email{gunther@math.washington.edu}
\address{Department of Mathematics, Stanford University, Stanford, CA
94305-2125, U.S.A.}
\email{andras@math.stanford.edu}
\thanks{The authors were partially supported by the National Science Foundation under
grant CMG-1025259 and are grateful for the
stimulating environment at the MSRI in Berkeley
where this work was initiated in Autumn 2010, and where the second author
was supported by a Senior Clay Award and a Chancellor's Professorship.}

\begin{abstract}
In this paper we show that for metrics with conormal singularities that correspond to
class $C^{1,\alpha}$, $\alpha>0$, the reflected wave is more regular than
the incident wave in a Sobolev sense. This is helpful
in the analysis of the multiple scattering series since higher order
terms can be effectively `peeled off'.
\end{abstract}

\maketitle

\section{Introduction}
In this paper we show that for metrics with conormal singularities that correspond to
class $C^{1,\alpha}$, $\alpha>0$, the reflected wave is more regular than
the incident wave in a Sobolev sense for a range of background Sobolev
spaces. That is, informally, for suitable $s\in\RR$ and $\ep_0>0$, depending on
the order of the conormal singularity (thus on $\alpha$),
if a solution of the wave equation is microlocally
in the Sobolev space $H^{s-\ep_0}_{\loc}$ prior to hitting the
conormal singularity of the metric in a normal fashion, then the reflected wave front
is in $H^s_{\loc}$, while the transmitted front is just in the a
priori space $H^{s-\ep_0}_{\loc}$. (This assumes that along the backward
continuation of the reflected ray, one has $H^s_{\loc}$ regularity,
i.e.\ there is no incident $H^s_{\loc}$ singularity for which
transmission means propagation
along our reflected ray.)
Such a result is helpful
in the analysis of the multiple scattering series, i.e.\ for waves
iteratively reflecting from conormal singularities, since higher order
terms, i.e.\ those involving more reflections,
can be effectively `peeled off' since they have higher regularity.

Here the main interest is in
$\alpha<1$, for in the $C^{1,1}$ setting one has at least a partial
understanding of wave propagation {\em without} a geometric structure
to the singularities of the metric,
such as conormality (though of course one does need {\em some}
geometric structure to obtain a theorem analogous to ours), as then
the Hamilton vector field is Lipschitz, and automatically has unique
integral curves; see Smith's paper \cite{Smith:Parametrix} where a
parametrix was constructed, and also the work of Geba and Tataru
\cite{Geba-Tataru:Phase}. We also recall that, in
a different direction, for even lower
regularity coefficients, Tataru has shown Strichartz estimates
\cite{Tataru:Strichartz}; these are not microlocal in the sense of
distinguishing reflected vs.\ transmitted waves as above.

In order to state the theorem precisely we need more notation.
First suppose $X$ is a $\dim X=n$-dimensional $\CI$ manifold, and $Y$
is a smooth embedded submanifold of 
codimension
$$
\codim Y=\codimY.
$$
With H\"ormander's normalization \cite{FIO1}, the class of Lagrangian
distributions associated to the conormal bundle $N^*Y$ of $Y$ (also
called distributions conormal to $Y$), denoted by
$I^\sigma(N^*Y)$, arises from symbols in $S^{\sigma+(\dim
 X-2\codimY)/4}$ when parameterized via a partial inverse Fourier
transform in the normal variables. That is, if one has local coordinates
$(x,y)$, such that $Y$ is given by $x=0$, then $u\in I^\sigma(N^*Y)$
can be written, modulo $\CI(\RR^n)$, as
$$
(2\pi)^{-\codimY}\int e^{ix\cdot \xi} a(y,\xi)\,d\xi,\qquad a\in S^{\sigma+(\dimX-2\codimY)/4}.
$$
For us it is sometimes convenient to have the orders relative to
delta distributions associated to $Y$, which arise as the partial
inverse Fourier transforms of symbols of order $0$, as in
\cite{Greenleaf-Uhlmann:Recovering}, thus we let
$$
I^{[-s_0]}(Y)=I^{-s_0-(\dim X-2\codimY)/4}(N^*Y),
$$
so elements of $I^{[-s_0]}(Y)$ are $s_0$ orders more regular than such
a delta distribution.
For any $\CI$ vector bundle over $X$ one can then talk about conormal
sections (e.g.\ via local trivialization of the bundle); in
particular, one can talk conormal metrics.

Thus, if $X$ is a $\CI$ manifold, $Y$ an embedded submanifold, and $g$
a symmetric 2-cotensor which is in $I^{[-s_0]}(Y)$ with
$s_0>\codimY=\codim Y$ (here we drop the bundle from the notation of
conormal spaces), then $g$ is continuous. We say that $g$ is
Lorentzian if for each $p\in X$, $g$
defines a symmetric bilinear form on $T_p X$ of signature $(1,\dimX-1)$,
$\dimX=\dim X$. (One would say $g$ is Riemannian if the signature is
$(n,0)$. Another possible normalization of Lorentzian signature is
$(n-1,1)$.) We say that $Y$ is time-like if the pull-back of $g$ to
$Y$ (which is a $\CI$ 2-cotensor) is Lorentzian, or equivalently if
the dual metric $G$ restricted to $N^*Y$ is negative definite. 

A typical example, with $Y$ time-like, is if $X=X_0\times\RR_t$, where $X_0$ is the
`spatial' manifold, $Y=Y_0\times\RR$, $g=dt^2-g_0$, $g_0$ is (the
pull-back of) a
Riemannian metric on $X_0$ which is conormal to $Y_0$, in the class
$I^{[-s_0]}(Y_0)$,
where $s_0>\codim_X Y=\codim_{X_0}Y_0$. In this case, one may choose
local coordinates $(x,y')$ on $X_0$ such that $Y_0$ is given by $x=0$;
then with $y=(y',t)$, $(x,y)$ are local coordinates on $X$ in which
$Y$ is given by $x=0$. Thus, the time variable $t$ is one of the $y$
variables in this setting.

Before proceeding, recall that there is a propagation of singularities
result in the manifolds with corners setting
\cite{Vasy:Propagation-Wave}, which requires only minimal changes to adapt
to the present setting.
This states that for solutions of the wave equation lying
in $H^{1,r}_{\bl}(X)$ for some $r\in\RR$, $\WF^{1,m}_{\bl}$ propagates
along generalized broken bicharacteristics. Thus, for a ray normally
incident at $Y$, if all of the incoming rays that are incident at the same
point in $Y$ and that have the same tangential
momentum carry $H^{m+1}$ regularity, then the outgoing rays from this
point in $Y$ with this tangential momentum will carry
the same regularity. In other words, in principle (and indeed, when
one has boundaries, or transmission problems with jump singularities
of the metric, this is typically the case) $H^{m+1}$ singularities can jump from
a ray to another ray incident at the same point with the same
tangential momentum (let us call these {\em related rays}), i.e.\ one has a whole cone (as the magnitude of
the normal momentum is conserved for the rays) of reflected rays
carrying the $H^{m+1}$ singularity. Here we recall that for $r\geq 0$, $H^{1,r}_{\bl}(X)$
is the subspace of $H^1(X)$ consisting of elements possessing $r$ b
(i.e.\ tangential to $Y$) derivatives in $H^1(X)$; for $r<0$ these are
distributions obtained from $H^1(X)$ by taking finite linear
combinations of up to $-r$ derivatives of elements of $H^1(X)$. In
particular, one can have arbitrarily large singularities;
one can always represent these by taking tangential derivatives, in
particular time derivatives. Via standard functional analytic duality
arguments,
these estimates (which also hold for the inhomogeneous equation) also give {\em
  solvability}, provided there is a global time function $t$.
Phrased in terms of these spaces, and for convenience for the
inhomogeneous equation with vanishing
initial data, for $f\in
H^{-1,r+1}_{\bl}(X)$ supported in $t>t_0$ there exists a unique $u\in H^{1,r}_{\bl}(X)$ solving the
equation $\Box_g u=f$ such that $\supp u\subset\{t>t_0\}$.

The object of this paper is to improve on this propagation result by showing
that, when $s_0>\codimY+1$ (thus $I^{[-s_0]}(Y)\subset C^{1+\alpha}$ for
$\alpha<s_0-\codimY-1$) in fact this jump to the related rays does not happen in an
appropriate range of Sobolev spaces. As above,
let $(x,y)$ denote local
coordinates on $X$, $Y$ given by $x=0$, and let $(\xi,\eta)$ 
denote dual variables.
Let $\Sigma\subset T^*X$ denote the
characteristic set of the wave operator $\Box=\Box_g$; this is the
zero-set of the dual metric $G$ in $T^*X$.

\begin{thm}\label{thm:microloc-background}
Suppose $\codim Y=\codimY=1$, $\codimY+1+2\ep_0< s_0$ and
$0<\ep_0\leq s<s_0-\ep_0-1-\codimY/2$.
Suppose that $u\in L^2_{\loc}$, $\Box u=0$,
$$
q_0=(0,y_0,\xi_0,\eta_0)\in\Sigma,\ \xi_0\neq 0,
$$
and the backward bicharacteristics from {\em related} points
$(0,y_0,\xi,\eta_0)\in\Sigma$ are disjoint from $\WF^{s-\ep_0}(u)$,
and
the backward bicharacteristic from the point
$q_0$ is disjoint from $\WF^s(u)$. Then
the forward bicharacteristic from
$(0,y_0,\xi_0,\eta_0)$ is disjoint from $\WF^s(u)$.
\end{thm}

\begin{rem}
The theorem is expected to be valid for all values of $\codimY$, and
the limitation on $\codimY$ in the statement is so that it fits
conveniently into the existing (b-microlocal) framework for proving
the basic propagation of singularities (law of reflection)
without too many technical changes. This is discussed in
Section~\ref{sec:reflection}, and is to some extent `orthogonal' to
the actual main ideas of the paper; it is only used to microlocalize
the `background regularity', $H^{s-\ep_0}$. If one does not want to microlocalize
the background regularity, i.e.\ assumes $u$ is in $H^{s-\ep_0}$
at least locally, we prove the result for all codimensions, see Theorem~\ref{thm:global-background}.
\end{rem}

Thus, the limiting Sobolev regularity $s$ that one can obtain, if $s_0$ is
slightly greater than $1+\codimY$, i.e.\ $2$ in the case of a
hypersurface, which is the minimum allowed by the first constraint, is
just above $\codimY/2$. On the other hand, if $s_0>1+\codimY$ then for any
$0\leq s<s_0-1-\codimY/2$, one can choose $\ep_0>0$ sufficiently small
so that all the inequalities are satisfied, so the theorem always
provides interesting information on wave propagation for a range of
values of $s$, providing at least some improvement over the basic
propagation of singularities result (which would not allow better regularity than that
on backward rays from $(0,y_0,\xi,\eta_0)\in\Sigma$, i.e.\ $H^{s-\ep_0}$).

\begin{cor}
Under assumptions as in the theorem, the terms of the multiple
scattering series have higher regularity, in the sense of Sobolev wave
front sets, with each iteration, until the limiting regularity,
$H^{s_0-1-\codimY/2}$, is reached.
\end{cor}

In view of the propagation of singularities along {\em generalized
  broken bicharacteristics}, i.e.\ that singularities can spread at
most to related rays, Theorem~\ref{thm:microloc-background} is in fact
equivalent to the weaker version where one assumes
$H_{\loc}^{s-\ep_0}$ regularity {\em not just on related
rays}. Thus, as we show in Section~\ref{sec:reflection}, it
suffices to prove the following theorem, which is what we prove in Section~\ref{sec:prop-sing}:

\begin{thm}\label{thm:global-background}
Suppose that $\ep_0>0$,
$\codimY+1+2\ep_0< s_0$ and $-\codimY/2<s<s_0-\ep_0-1-\codimY/2$. Then
for $u\in H_{\loc}^{s-\ep_0}$, $\Box u\in H_{\loc}^{s-1}$, $\WF^s(u)$ is a union of maximally extended
bicharacteristics in $\Sigma$.
\end{thm}

Note that if  $s_0>1+\codimY$, then first taking $-\codimY/2<
s<s_0-1-\codimY/2$, and then $\ep_0>0$ sufficiently small, all
the inequalities in the Theorem are satisfied.

This theorem is proved by a positive commutator, or microlocal energy,
estimate. They key issue is that as the wave operator does not have
$\CI$ coefficients, the commutator of a pseudodifferential
microlocalzier with it is {\em not} a pseudodifferential operator;
instead it is a sum of paired Lagrangian distributions associated to
various Lagrangian submanifolds of $T^*(X\times X)$. Thus, the main
technical task is to analyze these Lagrangian pairs, including their
Sobolev boundness properties.

The plan of the paper is the following. In Section~\ref{sec:structure}
we recall the structure of positive commutator estimates, in
particular the robust version due to Melrose and Sj\"ostrand
\cite{Melrose-Sjostrand:I, Melrose-Sjostrand:II}, used in their proof
of propagation of singularities at glancing rays on manifolds with a
smooth boundary. In Section~\ref{sec:bichar} we describe the structure
of the bicharacteristics, in particular their uniqueness
properties. In Section~\ref{sec:reflection} we recall the already
mentioned b-Sobolev spaces and the `standard' propagation of
singularities theorem based on these, also discussing how these can be
used to reduce Theorem~\ref{thm:microloc-background} to Theorem~\ref{thm:global-background}. Section~\ref{sec:paired-Lag} is the technical heart of the
paper in which we analyze paired Lagrangian distributions relevant to
the positive commutator estimates in our
setting. Section~\ref{sec:elliptic} gives microlocal elliptic
regularity in this setting, and is used as a warm-up towards the
positive commutator estimate. Section~\ref{sec:prop-est} gives the proof of the key
analytic estimate towards the proof of the propagation of
singularities, which is completed in Section~\ref{sec:prop-sing} in the form of Theorem~\ref{thm:global-background}.

\section{The structure of positive commutator estimates}\label{sec:structure}
In order to motivate our proof, we recall
the structure of the standard positive commutator estimate, in the
formulation of H\"ormander \cite{Hormander:Existence}, Melrose and
Sj\"ostrand \cite{Melrose-Sjostrand:I, Melrose-Sjostrand:II}, giving
propagation of singularities for the wave operator $\Box$ on a $\CI$
Lorentzian manifold $(X,g)$ (and indeed more generally for pseudodifferential operators
of real principal type).

We state at the outset that since all results are local, {\em one may
always arrange that the Schwartz kernels of various operators we
consider have proper support, or even compact support, and we do not
comment on support issues from this point on}. Similarly, {\em all Sobolev
spaces in which distributions are assumed to lie are local}, and we do
not always show this in the notation explicitly.

One arranges that for an appropriate operator $A\in\Psi^{2s-1}(X)$
that
$$
{i}[\Box,A]=B^*B+E+F,\ B\in\Psi^s(X),\ E\in\Psi^{2s}(X),\ F\in\Psi^{2s-2\ep_0}(X),
$$
with $\ep_0>0$ (typically $\ep_0=1/2$), where the solution is a priori
known to lie in $H^s$ on $\WF'(E)$ (this is where we propagate the
estimate from), and lie in $H^{s-\ep_0}$ on $\WF'(F)$ (which is
typically equal to $\WF'(A)$). Then one gets for $u$ with $\Box u=0$
(or even $\Box u=f$),
\begin{equation}\label{eq:basic-pairing}
\langle {i} Au,\Box u\rangle-\langle {i} \Box u,A^*u\rangle=
\langle {i}[\Box,A]u,u\rangle=\|Bu\|^2+\langle Eu,u\rangle+\langle 
Fu,u\rangle,
\end{equation} 
provided that $u$ is sufficiently nice for the pairings and the
adjoint (integration by parts) to make sense; then one can estimate
$Bu$ in $L^2$, and thus $u$ on the elliptic set of $B$ in $H^s$ in
terms of
$u$ on $\WF'(E)$ in $H^s$, $u$ on $\WF'(F)$ in $H^{s-\ep_0}$ and $u$
itself in any Sobolev space $H^{-N}$ globally (the latter is to deal
with smoothing errors). A standard regularization
argument gives that $u\in H^s$ actually on the elliptic set $\Ell(B)$
of $B$ even without
stronger a priori assumptions.

The desired commutator then is arranged by choosing some symbol $a$ in $S^{2s-1}$,
such that, with $p$ denoting the dual metric function, which is the
principal symbol of $\Box$,
\begin{equation}\label{eq:pos-comm-symbol}
H_p a=-b^2+e,\ \text{modulo}\ S^{2s-2\ep_0},
\end{equation}
and letting $a,b,e$ be the principal symbols of $A$, $B$ and $E$
respectively.
We recall how to do this in a robust manner, following the
presentation of \cite[Section~7]{Vasy:Geometric-optics}, though with
the more convenient notation of constants of
\cite{Vasy:Propagation-Wave} and \cite{Vasy:Maxwell}. Fix $\rho$ to be
a positive elliptic symbol of order $1$ locally in the region where we are
considering, e.g.\ $\rho=\langle\xi\rangle$ in canonical coordinates
$(x,\xi)$ based on local coordinates $x$ on the base space $X$. Let
$$
\sH_p=\rho^{-m+1}H_p,
$$
so $\sH_p$ is homogeneous of degree zero. Homogeneous
degree zero functions can be regarded as functions on $S^*X$, and
correspondingly $\sH_p$ can be considered a vector field on $S^*X$.
One can actually arrange
local coordinates $(q_1,q_2,\ldots,q_{2n-1})$ on $S^*X$ such that
$\sH_p=\frac{\pa}{\pa q_1}$ -- this is not necessary, but is a useful guide.
First let $\etat\in\CI(S^*X)$ be a function with
\begin{equation}\label{eq:Ham-vf-flow-relation}
\etat(\bar q)=0,\ \sH_p\etat(\bar q)>0.
\end{equation}
Thus, $\etat$ measures propagation along bicharacteristics; e.g.\
$\etat=q_1$ works, but so do many other choices.
We will use a function $\omega$ to localize near putative bicharacteristics.
This statement is deliberately vague; at first
we only assume that $\omega\in\CI(S^*X)$ is the sum of the squares of
$\CI$ functions $\sigma_j$, $j=1,\ldots,2n-2$,
with non-zero differentials at $\bar q$ such that $d\etat$ and
$d\sigma_j$, $j=1,\ldots,2n-2$, span $T_{\bar q}S^*X$, and such that
\begin{equation}\label{eq:Ham-vf-localizer-relation}
\sH_p\sigma_j(\bar q)=0.
\end{equation}
Such a function $\omega$ is non-negative
and it vanishes
quadratically at $\bar q$, i.e.\ $\omega(\bar q)=0$ and
$d\omega(\bar q)=0$. Moreover, $\omega^{1/2}+|\etat|$ is equivalent to
the distance from $\bar q$ with respect to
any distance function given by a Riemannian metric on $S^*X$.
An example is $\omega=q_2^2+\ldots+q_{2n-1}^2$
with the notation from before, but again there are many other possible
choices; with this choice $\sH_p\omega=0$.
We now consider a family symbols, parameterized by constants 
$\delta\in(0,1)$, $\ep\in(0,1]$, $\beta\in(0,1]$, of the form
\begin{equation}\label{eq:a-form-def}
a=\chi_0\left(\digamma^{-1}\Big(2\beta-\frac{\phi}{\delta}\Big)\right)\chi_1\left(\frac{\etat+\delta}{\ep\delta}+1\right),
\end{equation}
where
\begin{equation*}
\phi=\etat+\frac{1}{\ep^2\delta}\omega,
\end{equation*}
$\chi_0(t)=0$ if $t\leq 0$, $\chi_0(t)=e^{-1/t}$ if $t>0$,
$\chi_1\in\CI(\RR)$, $\chi_1\geq 0$, $\sqrt{\chi_1}\in\CI(\RR)$, $\supp\chi_1\subset[0,+\infty)$,
$\supp\chi_1'\subset[0,1]$, and $\digamma>0$ will be taken large. Here
$\digamma$ is used to deal with technical issues such as weights and
regularization, so at first reading
one may consider it fixed. We also need
weights such as $\rho^{2s-1}$ where $s\in\RR$ is as above; in a product type
Lorentzian setting these can be arranged Hamilton commute with $p$ by
taking $\rho=|\tau|$ and thus can
be ignored, otherwise taking $\digamma$ large will deal with them in
any case. Thus, the actual principal symbol of $A$ is
\begin{equation}\label{eq:commutant-symbol}
\sigma_{2s-1}(A)=\rho^{2s-1}\chi_0\left(\digamma^{-1}\Big(2\beta-\frac{\phi}{\delta}\Big)\right)\chi_1\left(\frac{\etat+\delta}{\ep\delta}+1\right),
\end{equation}

We analyze the properties of $a$ step by step. First, note
that $\phi(\bar q)=0$, $\sH_p\phi(\bar q)=\sH_p\etat(\bar q)>0$,
and $\chi_1(\frac{\etat+\delta}{\ep\delta}+1)$ is identically $1$
near $\bar q$, so $\sH_p a(\bar q)<0$. Thus, $\sH_p a$ has
the correct sign, and is in particular non-zero, at $\bar q$.

Next,
\begin{equation*}
q\in\supp a\Rightarrow\phi(q)\leq2\beta\delta
\ \text{and}\ \etat(q)\geq-\delta-\ep\delta.
\end{equation*}
Since $\ep\leq 1$, we deduce that in fact $\etat=\etat(q)\geq -2\delta$.
But $\omega\geq 0$, so $\phi=\phi(q)\leq 2\beta\delta$ implies that
$\etat=\phi-\ep^{-2}\delta^{-1}\omega\leq \phi\leq 2\beta\delta\leq 2\delta$. Hence, $\omega=\omega(q)
=\ep^2\delta(\phi-\etat)\leq 4\ep^2\delta^2$. Since $\omega$ vanishes quadratically
at $\bar q$, it is useful to rewrite the estimate as $\omega^{1/2}\leq
2\ep\delta$. Combining these, we have seen that on $\supp a$,
\begin{equation}\label{eq:supp-a-est}
-\delta-\ep\delta\leq \etat\leq 2\beta\delta
\ \text{and}\ \omega^{1/2}\leq 2\ep\delta.
\end{equation}
Moreover, on $\supp a\cap\supp\chi_1'$, 
\begin{equation*}
-\delta-\ep\delta\leq \etat\leq -\delta
\ \text{and}\ \omega^{1/2}\leq 2\ep\delta.
\end{equation*}
Note that given any neighborhood $U$
of $\bar q$, we can thus make $a$ supported in $U$ by
choosing $\delta$ sufficiently small (and keeping $\ep,\beta\leq 1$).
Note that $\supp a$ is a parabola shaped region, which is very explicit
in case $\etat=q_1$ and $\omega=q_2^2+\ldots+q_{2n-1}^2$.
Note that as $\ep\to 0$, but $\delta$ fixed,
the parabola becomes very sharply localized at $\omega=0$; taking
$\beta$ small makes $a$ localized very close to the segment $\etat\in
[-\delta,0]$.

So we have shown that $a$ is supported near $\bar q$. We define
\begin{equation}\label{eq:e-form}
e=\chi_0\left(\digamma^{-1}\Big(2\beta-\frac{\phi}{\delta}\Big)\right)\sH_p\left(\chi_1\Big(\frac{\etat+\delta}{\ep\delta}+1\Big)\right),
\end{equation}
so the crucial question in our quest for \eqref{eq:pos-comm-symbol}
is whether $\sH_p\phi\geq 0$ on $\supp a$. Note that
choosing $\delta_0\in(0,1)$ sufficiently small,
one has for $\delta\in (0,\delta_0]$, $\ep\in(0,1]$, $\beta\in(0,1]$,
$\sH_p\etat\geq c_0>0$ where $|\etat|\leq 2\delta_0$,
$\omega^{1/2}\leq 2\delta_0$. So $\sH_p\phi\geq \frac{c_0}{2}>0$ on
$\supp a$ if $\delta<\delta_0$, $\ep,\beta\leq 1$,
provided that $|\sH_p\omega|\leq \frac{c_0}{2}\ep^2\delta$ there, which is
automatically the case if one arranges
\begin{equation}\label{eq:precise-localizer}
\sH_p q_j=0\ \text{for}\ j\geq 2,\ \text{and}\ \sigma_j=q_{j+1},
\end{equation}
i.e.\ any $\ep>0$ works. Note that if
$H_p\phi\geq \frac{c_0}{2}$
on $\supp a$ then one can let
\begin{equation}\label{eq:b-form}
b=\digamma^{-1/2}\delta^{-1/2}\sqrt{\sH_p\phi}\sqrt{\chi_0'\left(\digamma^{-1}\Big(2\beta-\frac{\phi}{\delta}\Big)\right)}\sqrt{\chi_1\left(\frac{\etat+\delta}{\ep\delta}+1\right)};
\end{equation}
thus \eqref{eq:pos-comm-symbol} holds with $s=1/2$ and $\ep_0=1/2$.

However, we do not need such a strong relationship to $\sH_p$, which
cannot be arranged (with smooth $\sigma_j$) if one makes $p$ have
conormal singularities at a submanifold. Suppose instead
that we merely get $\omega$ `right' at $\bar q$, in the sense that
\begin{equation}\label{eq:rougher-localizer}
\omega=\sum \sigma_j^2,\ \sH_p\sigma_j(\bar q)=0.
\end{equation}
Then,
$\sH_p\sigma_j$ being a $\CI$, thus locally Lipschitz, function,
\begin{equation}\label{eq:H_p-sigma-est}
|\sH_p\sigma_j|\leq C_0(\omega^{1/2}+|\etat|),
\end{equation}
so
$|\sH_p\omega|\leq C \omega^{1/2}(\omega^{1/2}+|\etat|)$. Using
\eqref{eq:supp-a-est}, we deduce that $|\sH_p\omega|\leq \frac{c_0}{2}\ep^2\delta$
provided that $\frac{c_0}{2}\ep^2\delta\geq C''(\ep\delta)\delta$, i.e.\ that
$\ep\geq C'\delta$ for some constant $C'$ independent of $\ep$,
$\delta$ (and of $\beta$).
Now the size of the parabola at $\etat=-\delta$ is roughly
$\omega^{1/2}\sim\delta^2$, i.e.\ we have localized along a single
direction, namely the direction of $\sH_p$ at $\bar q$.

By a relatively
simple argument, also due to Melrose and Sj\"ostrand
\cite{Melrose-Sjostrand:I, Melrose-Sjostrand:II} in the case of smooth
boundaries,
one can piece together such estimates (i.e.\ where
the direction is correct `to first order') and deduce the
propagation of singularities. We explain this in more detail in the
last section of the paper.

This argument would go through if one manages to arrange this with $F$
having just the property that $F:H^{s-\ep_0}\to H^{-s+\ep_0}$, i.e.\
the ps.d.o.\ behavior of $F$ does not matter as long as one has $H^{s-\ep_0}$
background regularity -- indeed, one only needs the $H^{s-\ep_0}$
regularity on the wave front set of $F$.

We finally indicate how one deals with regularizers and weights.
Let $\rho$ is a positive elliptic symbol of order $1$ as above.
It is convenient to write
$$
\check a=\rho^{s-1/2}\sqrt a\in S^{s-1/2}
$$
with $a$ as in \eqref{eq:a-form-def},
and
let $\check A\in\Psi^{s-1/2}$ have principal symbol $\check a$,
$\WF'(\check A)$ contained in the conic support of $\check a$, and be
formally self-adjoint (e.g.\ take $\check A_0$ to be a quantization of
$\check a$ in local coordinates, and then take the self-adjoint part,
$\check A=(\check A+\check A^*)/2$), and
let $A=\check A^2$. We
also let $\Lambda_r$, $r\in[0,1]$, be such that the family is uniformly
bounded in $\Psi^0(X)$, $\Lambda_r\in\Psi^{-1}$ for $r>0$, and
$\Lambda_r\to\Id$ in $\Psi^\ep$ for $\ep>0$, and $\Lambda_r$ formally
self-adjoint.
For instance, one can
take $\Lambda_r$ to be a (symmetrized) quantization of $\phi_r=(1+r\rho)^{-1}$.
Let
$$
A_r=\Lambda_r A\Lambda_r,\ a_r=\phi_r^2\rho^{2s-1} a.
$$
Then the principal symbol of
${i}[\Box,A_r]$, as a family with values in $\Psi^{2s}$, is
$$
\phi_r^2\rho^{2s}\sH_p a+a\phi_r^2\rho^{2s}((2s-1)-r\phi_r\rho) (\rho^{-1}\sH_p\rho).
$$
Now, $|r\phi_r\rho|\leq 1$ while $\rho^{-1}\sH_p\rho$ is bounded,
being a symbol of order $0$, so the second term is bounded in absolute
value by $Ca\phi_r^2\rho^{2s}$. Now, given $M>0$, for sufficiently large
$\digamma$, not only is $\sH_p a$ of the form $-b^2+e$, but
$$
\phi_r^2\rho^{2s}\big(\sH_p a+((2s-1)-r\phi_r\rho)
(\rho^{-1}\sH_p\rho) a\big)=-b_r^2-M^2
\rho a_r+e_r,
$$
with $e_r=\phi_r^2 \rho^{2s}e$, $e$ as before.
This is due to $\chi_0(t)=t^2\chi_0'(t)$ for
$t\in\RR$, so
\begin{equation}\begin{aligned}\label{eq:absorb-a-into-b}
&\digamma^{-1}\delta^{-1}(\sH_p\phi)\chi_0'\left(\digamma^{-1}\Big(2\beta-\frac{\phi}{\delta}\Big)\right)\\
&\qquad\qquad-\Big(\big((2s-1)-r\phi_r\rho\big) (\rho^{-1}\sH_p\rho)+M^2\Big)\chi_0\left(\digamma^{-1}\Big(2\beta-\frac{\phi}{\delta}\Big)\right)\\
&=\digamma^{-1}\delta^{-1}\left((\sH_p\phi)-\Big(\big((2s-1)-r\phi_r\rho\big)
  (\rho^{-1}\sH_p\rho)+M^2\Big)\digamma^{-1}\delta\Big(2\beta-\frac{\phi}{\delta}\Big)^2\right)\\
&\qquad\qquad\qquad\qquad\qquad\times\chi_0'\left(\digamma^{-1}\Big(2\beta-\frac{\phi}{\delta}\Big)\right),
\end{aligned}\end{equation}
and $|2\beta-\frac{\phi}{\delta}|\leq 4$ on $\supp a$, so for sufficiently
large $\digamma$ (independent of $\delta,\ep,\beta\in(0,1]$ as long as
$\ep\geq C'\delta$, $C'$ as above),
the factor in the large parentheses on the right hand side is positive,
with a positive lower bound,
and thus its square root $c_r$ satisfies that $c_r\in S^0$ uniformly,
$c_r\in S^{-1}$ for $r>0$, and $c_r$ is elliptic where $\chi_0$ and
$\chi_1$ are both positive. Now with $E_r=\Lambda_r E\Lambda_r$, $E$
as before  with wave front set in the conic support of $a$,
and taking $B_r$ a family, uniformly
bounded in $\Psi^{s}$, with (uniform, or
family) wave front set in the conic support of $a$ and with principal symbol
\begin{equation}\label{eq:reg-b-def}
b_r=\phi_r \rho^s c_r \sqrt{\chi_0'\left(\digamma^{-1}\Big(2\beta-\frac{\phi}{\delta}\Big)\right)}\sqrt{\chi_1\left(\frac{\etat+\delta}{\ep\delta}+1\right)},
\end{equation}
we have
$$
{i}[\Box,A_r]=-B_r^*B_r-M^2(\check A_r)^*Q^*Q\check A_r+E_r+F_r,\ \check
A_r=\check A\Lambda_r,
$$
with $Q\in\Psi^{1/2}$ with symbol $\rho$ (thus elliptic),
with $F_r$ uniformly bounded in $\Psi^{2s-1}$, and with uniform
wave front set in the conic support of $a$. Now for $r>0$ applying this
expression to $u$ and pairing with $u$, as in
\eqref{eq:basic-pairing}, makes sense provided $\WF^{s-1/2}(u)$ is
disjoint from the conic support of $a$, and we obtain
\begin{equation}\label{eq:basic-pairing-mod}
\|B_r u\|^2+M^2\|Q\check A_r u\|^2\leq 2|\langle  A_ru,\Box u\rangle|
+|\langle E_ru,u\rangle|+|\langle F_ru,u\rangle|.
\end{equation}
Further, with $G$ a parametrix for $Q$ with $GQ=\Id+R$, $R\in\Psi^{-\infty}$,
\begin{equation}\begin{aligned}\label{eq:break-up-inhomog}
2|\langle A_ru,\Box u\rangle|&\leq 2|\langle Q\check A_r u,G\check A_r\Box
u\rangle|+2|\langle R\check A_r u,\check A_r\Box
u\rangle|\\
&\leq \|Q\check A_r u\|^2+\|G\check A_r\Box
u\|^2 +2|\langle R\check A_r u,\check A_r\Box u\rangle|,
\end{aligned}\end{equation}
and the first term on the left hand side now can be absorbed into
$M^2\|Q\check A_r u\|^2$ (if we chose $M\geq 1$).
Letting $r\to 0$
we get a uniform bound for $\|B_r u\|$, and thus by the weak
compactness of the unit ball in $L^2$ plus that $B_r u\to B_0 u$ in
distributions, we conclude that $B_0 u\in L^2$, completing the proof
that the elliptic set of $B_0$, i.e.\ where $\chi_0$ and $\chi_1$ are
positive, is disjoint from $\WF^s(u)$.

One completes the proof of the
propagation estimate by an inductive argument in $s$, raising the
order $s$ by
$1/2$ in each step. During this process one needs to shrink the
support of $a$ so that, denoting the replacement of $a$ given in the next step of the
iteration by $a'$, at every point of $\supp a'$ either $b$ is elliptic (the $b$
corresponding to the original $a$), or one has a priori
regularity there (which is the case on $\supp e$). This can be done by
reducing $\beta$ which shrinks the support as
desired. We refer to \cite[Section~24.5]{Hor}, in particular to last paragraph of the proof of
Proposition~24.5.1, for further details.

\section{Bicharacteristics}\label{sec:bichar}

Since $g$ is not $\CI$, we need to discuss the behavior of bicharacteristics, i.e.\ integral
curves of $\sH_p$, in some detail. When $g\in I^{[-s_0]}(Y)$ and $\codim
Y+1+\alpha<s_0<\codim Y+2$ (with $0<\alpha<1$),
which is the main case of interest for us,
then $g$ is $C^{1,\alpha}$, and thus
$\sH_p$ is a $C^{0,\alpha}$. Thus, the standard ODE theory ensures the
existence of bicharacteristics, but does not ensure their uniqueness
(as H\"older-$\alpha$, $\alpha<1$, is insufficient for this; Lipschitz would
suffice). Nevertheless, for normally incident rays at a codimension
one hypersurface $Y$ one has local
uniqueness. In this setting, locally, $\sH_p$ is transversal to $T^*_Y
X$, and using local coordinates $(x,y)$ such that $Y=\{x=0\}$ and dual coordinates
$(\xi,\eta)$, $\sH_p$
is continuous in $x$ and $\CI$ in $(y,\xi,\eta)$, so the following lemma gives this conclusion:

\begin{lemma}
If $I\subset\RR_{x_n}$ is an open interval containing $0$,
$O\subset\RR^{n-1}_{x'}$ open containing $0$, $V=\sum_{j=1}^n V_j(x)\pa_j$ is a continuous real vector field on
$O\times I$ with
$V_j\in C(I;C^{0,1}(O))$ and with $V_n(0)\neq 0$ then there exists $\Omega\subset
O\times I$ open containing $0$ and $\delta>0$ such that the given $x^{(0)}\in\Omega$,
there is a unique $C^1$ integral curve $x:(-\delta,\delta)\to O\times I$
with $x(0)=x^{(0)}$.
\end{lemma}

\begin{proof}
Since the other sign works similarly, we may assume that $V_n(0)>0$,
and also at the cost of shrinking $I$ and $O$ then $V_n>c>0$ on
$O\times I$.

Being an integral curve means that $\frac{dx_j}{dt}(t)=V_j(x(t))$. We
consider an other system of ODE, namely writing $Z(s)=(z'(s),s)$, with
$(-\delta',\delta')\subset I$, $s_0\in (-\delta',\delta')$,
$z'\in C^1((-\delta',\delta');O)$, $z_n\in
C^1((-\delta',\delta');\RR)$, $z=(z',z_n)$.
\begin{equation}\label{eq:rewritten-ODE}
\frac{dz}{ds}(s)=F(z'(s),s),\ z(s_0)=z^{(0)}\in O'\times I',
\end{equation}
with
\begin{equation}\begin{aligned}\label{eq:rewritten-ODE-coeffs}
&F_j(y)=\frac{V_j(y)}{V_n(y)},\ j=1,\ldots,n-1,\\
&F_n(y)=\frac{1}{V_n(y)},\\
\end{aligned}\end{equation}
so $F\in C((-\delta',\delta')_s;C^{0,1}(O))$. The key point here is that
$F(z'(s),s)$ on the right hand side of \eqref{eq:rewritten-ODE} is
independent of $z_n(s)$, i.e.\ \eqref{eq:rewritten-ODE} is of the type
$\frac{dz}{ds}(s)=\Phi(z(s),s)$, with $\Phi$ continuous in the last
variable and Lipschitz in the first. Thus, the standard ODE existence
and uniqueness theorem applies, giving the local existence and
uniqueness of solutions to \eqref{eq:rewritten-ODE}, provided
$O'\times I'$ is a sufficiently small neighborhood of $0$.

Now if $x=x(t)$ is a $C^1$ integral curve of $V$, and we let
$T$ be the inverse function of $x_n=x_n(t)$ near $0$, which exists and
is $C^1$ by
the inverse function theorem as $\frac{dx_n}{dt}(t)=V_n(x(t))\geq
c>0$, with $T'(s)=\frac{1}{V_n(x(T(s)))}$,
then $z=(z',z_n)$ with $z'=x'\circ T$, $z_n=T$,
satisfies \eqref{eq:rewritten-ODE} with $s_0=x_n(0)=(x^{(0)})_n$,
$z^{(0)}=(x'(0),0)=((x^{(0)})',0)$.
Indeed, $z$ is $C^1$ as $x$ and $T$ are such, and
\begin{equation*}\begin{aligned}
&\frac{dz_j}{ds}=(\frac{dx_j}{dt}\circ T)T'=\frac{V_j\circ x\circ
  T}{V_n\circ x\circ T},\ j=1,\ldots,n-1,\\
&\frac{dz_n}{ds}=\frac{1}{V_n\circ x\circ T},
\end{aligned}\end{equation*}
which, as $x_n\circ T(s)=s$, is a rewriting of
\eqref{eq:rewritten-ODE}. One can also proceed backwards, starting
with a solution of \eqref{eq:rewritten-ODE}, by letting $x_n$ be the
inverse function of $z_n$, and then letting $x_j=z_j\circ x_n$ for
$j=1,\ldots,n-1$.

Thus, if one has two solutions $x(t)$ and $\tilde x(t)$ of
$\frac{dx_j}{dt}(t)=V_j(x(t))$ with $x(0)=x^{(0)}$, then defining $T$,
resp.\ $\tilde T$, as the inverse functions of $x_n$, resp.\ $\tilde
x_n$, we have solutions $z$, resp.\ $\tilde z$ of
\eqref{eq:rewritten-ODE} with initial conditions $((x^{(0)})',0)$ and
time $(x^{(0)})_n$. Thus, by the uniqueness part of the ODE
theorem, $z=\tilde z$. The $n$th components give then $T=\tilde T$,
hence $x_n=\tilde x_n$,
and thus the other components yield $x_j=\tilde x_j$, completing the proof.
\end{proof}

As mentioned, an immediate consequence is, if one lets $\cG$ be the
glancing set, i.e.\ where $\sH_p$ is tangent to $T^*_YX$:

\begin{cor}
Suppose $Y$ has codimension $1$. Then
the integral curves of $\sH_p$ in $\Sigma\setminus \cG$ through a given point are unique.
\end{cor}

\begin{proof}
Suppose there are two solutions $x(t)$ and $\tilde x(t)$ with the same
initial condition $x^{(0)}$ at time $0$. Assuming that $x(t)\neq\tilde
x(t)$ for some $t>0$,
let $t_0$ be the infimum of positive times such that $x(t)\neq\tilde
x(t)$, so any neighborhood $I$ of $t_0$ contains $t\in I$ such that
$x(t)\neq\tilde x(t)$ but, as $x$ and $\tilde x$ are continuous
$x(t_0)=\tilde x(t_0)$. (The last assertion is clear if $t_0=0$; if
$t_0>0$ it follows as $x(t)=\tilde x(t)$ for $t\in [0,t_0)$ by
definition of $t_0$.) Then the local uniqueness result stated above
yields a contradiction. Since negative times are dealt with similarly, this
completes the proof.
\end{proof}

\section{Law of reflection: standard propagation of singularities}\label{sec:reflection}

We now recall from \cite{Vasy:Propagation-Wave} the basic law of
reflection. In \cite{Vasy:Propagation-Wave} this is shown in the
setting of manifolds with corners with Dirichlet or Neumann boundary
conditions. However, the same arguments go through in our setting,
where we consider the quadratic form domain
$H^1_{\loc}(X)$. Generalized broken bicharacteristics (GBB) are defined in
this setting to allow reflected rays as follows.

For simplicity consider $Y$ of codimension $1$ (this is all that is
needed for Theorem~\ref{thm:microloc-background}, and
Theorem~\ref{thm:global-background} does not need this at all). Since the results are
local, we may assume that $Y$ separates $X$ into two manifolds $X_\pm$
with boundary $Y$. Each of $X_\pm$ comes equipped with the so-called
b-cotangent bundle, $\Tb^* X_\pm$. This is the dual bundle of the
b-tangent bundle, whose smooth sections are $\CI$ vector fields on
$X_\pm$ tangent to $Y$, denoted by $\Vb(X_\pm)$.
Over $\CI(X_\pm)$, these are locally spanned by
$x\pa_x$ and $\pa_{y_j}$, $j=1,\ldots,n-1$, and correspondingly, a
local basis for smooth sections of $\Tb^*X_\pm$ is $\frac{dx}{x}$ and
$dy_j$, $j=1,\ldots,n-1$. One may thus write smooth sections of $\Tb^*X_\pm$
as
\begin{equation}\label{eq:b-coords}
\sigma(x,y)\,\frac{dx}{x}+\sum_j \eta_j(x,y)\,dy_j;
\end{equation}
so $(x,y,\sigma,\eta)$ are local coordinates on $\Tb^*X_\pm$.
As $\Vb(X_\pm)\subset
\Vf(X_\pm)$, there is a dual map $\pi_\pm:T^*X_\pm\to \Tb^*X_\pm$; the kernel
at $p\in Y$
is given by $N^*_pY$, and the range can be naturally
identified with $T^*_pY=T^*_p X_\pm/N^*_pY$. Concretely, if one uses
canonical dual coordinates $(x,y,\xi,\eta)$ on $T^*X$, writing
one-forms as
$$
\xi(x,y)\,dx+\sum_j\eta_j(x,y)\,dy_j,
$$
then
$$
\pi_\pm(x,y,\xi,\eta)=(x,y,x\xi,\eta),
$$
corresponding to the identification $\xi\,dx=(x\xi)\,\frac{dx}{x}$.
The same constructions can be performed directly on $X$, working with
$\CI$ vector fields tangent to $Y$, which we denote by $\Vb(X;Y)$. The
so obtained cotangent bundle $\Tb^* X$, which is a $\CI$ vector
bundle,
when restricted to $X_\pm$,
gives $\Tb^*X_\pm$, and again comes with a natural map
$\pi:T^*X\to\Tb^*X$.

In particular, one can now consider the characteristic set
$\Sigma\subset T^*X$ of $\Box$, and its image
$\dot\Sigma\subset\Tb^*X$ under $\pi$; this is called the compressed
characteristic set. A GBB $\tilde\gamma$ is defined to be a continuous map from an interval to
$\dot\Sigma$ satisfying a Hamilton vector field condition, namely that
for all $f\in\CI(\Tb^*X)$ real valued,
$$
\limsup_{s\to s_0} \frac{f(\tilde\gamma(s))-f(\tilde\gamma(s_0))}{s-s_0}\leq \sup
\{(\sH_p\pi^*f)(q):\ q\in\Sigma,\ \pi(q)=\tilde\gamma(s_0)\}.
$$
Thus, $C^1$ integral curves of $\sH_p$ in $\Sigma\subset T^*X$
are certainly generalized broken
bicharacteristics (i.e.\ their image under $\pi$ is),
but more generally, any two integral curve segments
of $\sH_p$, say $\gamma_+$ defined on $[0,s_0)$ and $\gamma_-$ on
$(-s'_0,0]$, can be combined
into a single GBB provided $\pi(\gamma_+(0))=\pi(\gamma_-(0))$.

For a Lorentzian metric $g$, $T^*Y$ can be regarded as a subset of $T^*X$,
identified as the orthocomplement of the spacelike $N^*Y$. In fact,
one may arrange that the dual metric $G$ is
$$
G=A(x,y)\pa_x^2+\sum_j 2C_j(x,y)\pa_x\pa_{y_j}+\sum_{ij}B_{ij}(x,y)\pa_{y_i}\pa_{y_j},
$$
with
$$
C_j(0,y)=0,\ A(0,y)<0,\ B(0,y)\ \text{Lorentzian on}\ T^*_yY,
$$
see \cite[Section~2]{Vasy:Maxwell}. We write
$$
B(0,y)\eta\cdot\eta=\sum_{ij}B_{ij}(0,y)\eta_i\eta_j
$$
for the dual metric function of $B$. Then $T^*Y$ is identified with
points with $x=0$ and $\xi=0$.
We recall from \cite{Vasy:Propagation-Wave} and \cite{Vasy:Maxwell}
that $\dot\Sigma=\cH\cup\cG$ is the union of the hyperbolic
and the glancing sets at $\Tb^*_YX$ with
$$
\cH\cap\Tb^*_YX=\pi(\Sigma\setminus T^*Y),\ \cG\cap\Tb^*_YX=\pi(\Sigma\cap T^*Y).
$$
Concretely, in coordinates on a chart $\cU$, using the b-coordinates $(x,y,\sigma,\eta)$,
\begin{equation*}\begin{aligned}
&\cH\cap\Tb^*_{\cU\cap Y}X=\{(0,y_0,0,\eta_0)\in\Tb^*_{\cU\cap Y}X:\ B(0,y_0)\eta_0\cdot\eta_0>0\},\\
&\cG\cap\Tb^*_{\cU\cap Y}X=\{(0,y_0,0,\eta_0)\in\Tb^*_{\cU\cap Y}X:\ B(0,y_0)\eta_0\cdot\eta_0=0\}
\end{aligned}\end{equation*}
If
$q_0=(0,y_0,\xi_0,\eta_0)\in\Sigma$ is not a glancing point, then locally all GBB $\tilde\gamma$
with $\tilde\gamma(0)=q_0$ are of
the form discussed above, i.e.\ the concatenation of two integral
curves of $\sH_p$. Indeed, such GBB stay outside $\Tb^*_Y X$ for a punctured time
interval, i.e.\ there is $\ep>0$ such that $\tilde\gamma(s)\notin\Tb^*_Y X$
for $s\in(-\ep,\ep)\setminus\{0\}$, so $\gamma_+=\tilde\gamma|_{(0,\ep)}$,
$\gamma_-=\tilde\gamma|_{(-\ep,0)}$ are integral curves of $\sH_p$;
see \cite[Lemma~2.1]{Vasy:Maxwell}. In view
of the kernel of the map $T^*X\to\Tb^*X$ at $Y$, this means exactly
that GBBs allow the standard law of reflection, i.e.\ the incident and
reflected rays differ by a covector in $N^*Y$.

In order to state the propagation of singularities theorem, we need a
notion of wave front set in $\Tb^*X\setminus o$. This is a simple
extension of $\WFb^{1,m}(u)$ introduced in
\cite{Vasy:Propagation-Wave} for manifolds with corners to a manifold
with a codimension one hypersurface $Y$ replacing the boundary, as
above.
This wave front set in turn is based on the so-called b-pseudodifferential operators.
In the setting of manifolds
with boundaries, or indeed, corners, such as $X_\pm$, these are just the totally
characteristic, or b-, pseudodifferential operators introduced by
Melrose \cite{Melrose:Transformation}, see also \cite{Melrose:Atiyah}, and discussed
by Melrose and Piazza \cite[Section~2]{Melrose-Piazza:Analytic}. We
also refer to \cite{Vasy:Propagation-Wave} for a concise description
of the background. In our setting, to work on $X$ with these
operators, we recall that the Schwartz kernels of $\Psib(X_+),\Psibc(X_+)$ are
tempered distributions on $X_+\times X_+$ which are conormal on the
blow-up $[X_+\times X_+;\pa X_+\times \pa X_+]$ to the front face and
the lifted diagonal, in the sense of being either the partial Fourier
transforms of symbols in the case of $\Psibc(X_+)$, or those of
classical (one-step polyhomogeneous) symbols in the case of
$\Psib(X_+)$, which extend smoothly across the front face (to which
the diagonal is transversal, and thus this makes sense), and
vanishing to infinite order on the side faces,
i.e.\ the lifts of $X_+\times\pa X_+$ and $\pa X_+\times X_+$.
Concretely, fixing $\phi\in\CI_c(\RR)$, identically $1$ near $0$,
supported in $(-1/2,1/2)$ and a coordinate chart $(x,y)$, a large subset of elements of
$\Psibc^m(X_+)$ and
$\Psib^m(X_+)$ (and indeed, all modulo smoothing
operators, i.e.\ elements of
$\Psib^{-\infty}(X_+)=\Psibc^{-\infty}(X_+)$) have the form
\begin{equation*}\begin{aligned}
&(A_+ v)(x,y)\\
&=(2\pi)^{-n}\int  e^{i\big(\sigma\frac{x-x'}{x'}+\sum_j
 \eta_j (y_j-y'_j)\big)}\,\phi\Big(\frac{x-x'}{x'}\Big)\,a_+(x,y,\sigma,\eta)\,v(x',y')\,\frac{dx'\,dy'}{x'},
\end{aligned}\end{equation*}
where
$$
a_+\in S^m([0,\infty)_x\times \RR^{n-1}_y;\RR^n_{\sigma,\eta}),\
\text{resp.}
\ a_+\in S^m_{\cl}([0,\infty)_x\times \RR^{n-1}_y;\RR^n_{\sigma,\eta})
$$
if $A_+\in\Psibc^m(X_+)$,
resp.\ $A_+\in\Psib^m(X_+)$. (Here the symbol notation denotes
symbolic behavior in the variables after the semicolon.) Note that $\phi$ is identically $1$ near
the diagonal lifted to $[X_+^2;(\pa X_+)^2]$, i.e.\ it does not affect
the diagonal singularity at all; its role is to localize away from the
side faces. Here the image of $a_+$ in $S^m/S^{m-1}$, or if $a_+$ is classical, the
homogeneous degree $m$ summand in its asymptotic expansion, is the
principal symbol $\sigma_{\bl,m}(A_+)$ of $A_+$; this is naturally a
function (or equivalence class of functions) on $\Tb^* X_+\setminus o$
(with $o$ the zero section) regarding
$(\sigma,\eta)$ as fiber coordinates on this bundle as in \eqref{eq:b-coords}.

We then
{\em define} $\Psib(X,Y)$ to consist of operators $A$ acting on
$\CI_{\piece}(X)$, continuous piecewise $\CI$ functions, i.e.\
continuous functions $v$ on $X$ with $v|_{X_\pm}$ being $\CI$, via
Schwartz kernels on $X^2$ supported in $(X_+)^2\times(X_-)^2$,
conormal on $[X^2;Y^2]$ such that the normal operators are the
same. Such an operator can be identified with a pair of operators
$(A_+,A_-)$ given by the restriction to $\dCI(X_+)$, $\dCI(X_-)$,
which are then in $\Psib(X_+)$, resp.\ $\Psib(X_-)$. Thus, modulo
$\Psib^{-\infty}(X,Y)$, with $\phi$ as above, these operators are of the form
\begin{equation*}\begin{aligned}
&(A v)(x,y)\\
&=(2\pi)^{-n}\int  e^{i\big(\sigma\frac{x-x'}{x'}+\sum_j
 \eta_j (y_j-y'_j)\big)}\,\phi\Big(\frac{x-x'}{x'}\Big)\,a(x,y,\sigma,\eta)\,v(x',y')\,\frac{dx'\,dy'}{x'},
\end{aligned}\end{equation*}
where
$$
a\in S^m(\RR_x\times \RR^{n-1}_y;\RR^n_{\sigma,\eta}),\
\text{resp.}
\ a\in S^m_{\cl}(\RR_x\times \RR^{n-1}_y;\RR^n_{\sigma,\eta})
$$
if $A\in\Psibc(X,Y)$,
resp.\ $A\in\Psib(X,Y)$. Note that the support condition on $\phi$
implies that $\frac{1}{2}\leq \frac{x}{x'}\leq \frac{3}{2}$ on
$\supp\phi$, so in particular $x$ and $x'$ have the same sign, which
means that $A$ preserves the class of distributions supported in
$X_+$, as well as those in $X_-$.

The key property of $\Psibc^0(X,Y)$ is given in the following lemma:

\begin{lemma} (cf.\ \cite[Lemma~3.2]{Vasy:Propagation-Wave})
Any $A\in\Psibc^0(X,Y)$ of compactly support
is bounded on $H^1(X)$,
with norm bounded by a seminorm in $\Psibc^0(X,Y)$. By duality, the
analogous statement holds on $H^{-1}(X)$ as well.
\end{lemma}

\begin{proof}
If $u\in\CI_\comp(X)$ (which is a dense subspace of $H^1(X)$), then
the compactly supported $Au$ restricts to a $\CI$ function on both $X_+$
and $X_-$, namely $A_\pm u|_{X_\pm}$, whose restriction to the
boundary is the indicial operator $\hat N(A_\pm)(0)$ applied to
$u|_Y$, and thus these two $\CI$ functions coincide at $Y$. As
first derivatives of such a continuous piecewise $\CI$ function are
given by the (no longer necessarily continuous, but still locally
bounded) $\CI$ functions given by differentiating the restrictions to
each half-space separately, and as $\|A_\pm u|_{X_\pm}\|_{H^1}\leq
C\|u|_{X_\pm}\|_{H^1}$ by \cite[Lemma~3.2]{Vasy:Propagation-Wave},
with $C$ bounded by a continuous seminorm on $\Psibc^0(X_\pm)$, the
claim follows.
\end{proof}

We in fact need to generalize the coefficients of $\Psibc(X,Y)$ to
allow conormal singularities if $g_{ij}$ are not simply piecewise
smooth, i.e.\ have $\CI$ restrictions to $X_\pm$.
The key point is that one can allow more general
conormal behavior at the front faces, i.e.\ allow $a$ to satisfy
symbolic bounds in $x$:
$$
\left|\big((xD_x)^\ell D_y^\alpha D_{(\sigma,\eta)}^\beta
  a\big)(x,y,\sigma,\eta)\right|
\leq C_{\ell\alpha\beta} \langle (\sigma,\eta)\rangle^{m-|\beta|};
$$
denote by $\Psibcc(X,Y)$ the resulting space.
With such coefficients, in general, $A\in\Psibcc^0(X,Y)$ no longer
preserves $H^1$, though if one requires $A=A_0+A_1$ with
$A_0\in\Psibc^0(X,Y)+x\Psibcc^0(X,Y)$, the $H^1$ bounds remain
valid. However, $L^2$ bounds are valid in general, and $\Psibcc(X,Y)$
is closed under composition with
\begin{equation*}\begin{aligned}
&A\in\Psibcc^m(X,Y),\ B\in\Psibcc^{m'}(X,Y)\\
&\qquad\Longrightarrow
AB\in\Psibcc^{m+m'}(X,Y),
\ [A,B]\in\Psibcc^{m+m'-1}(X,Y),
\end{aligned}\end{equation*}
with principal symbols given by
$$
\sigma_{\bl,m+m'}(AB)=\sigma_{\bl,m}(A)\sigma_{\bl,m'}(B),
\ \sigma_{\bl,m+m'}([A,B])=\frac{1}{{i}}\{\sigma_{\bl,m}(A),\sigma_{\bl,m'}(B)\},
$$
with $\{.,.\}$ being the Hamilton bracket lifted to $\Tb^*X$.
Note that if $f\in I^{[-s]}(Y)$ then the operator of multiplication by
$f$ is in $\Psibc^0(X,Y)$ provided $s>1$.

The propagation of singularities theorem is then the following:

\begin{thm}\label{thm:prop-sing}
Suppose $r,m\in\RR$, $u\in H^{1,r}_{\bl,\loc}(X)$ and
$\Box u\in H^{-1,m+1}_{\bl,\loc}(X)$. Then $\WFb^{1,m}(u)$ is a union of maximally extended GBB.
\end{thm}

This theorem is proved by using b-ps.d.o's, $A\in\Psibc(X)$ (so no
conormal coefficients allowed), as microlocalizers, gaining
regularity relative to $H^1_{\loc}(X)$. One works with the quadratic
form as was done in \cite{Vasy:Propagation-Wave} for the Neumann
boundary condition and in \cite{Vasy:Maxwell} for differential
forms. This requires commuting $A$ past $D_i$, which works exactly as
in these papers, as well as commuting $A$ through $g_{ij}\in
I^{[-s_0]}$. However, the commutator $[A,g_{ij}]\in \Psibcc(X,Y)$ need
not be further commuted through the derivatives $D_i$ in view of the
arguments of \cite[Proposition~3.10]{Vasy:Maxwell} and its uses in
Propositions~5.1 and Propositions~6.1 there, thus
the proof of Theorem~\ref{thm:prop-sing} can be completed as there.

\begin{rem}
Note that in particular Theorem~\ref{thm:prop-sing} holds for
transmission problems; indeed, these do not even require the
introduction of $\Psibcc(X,Y)$, i.e.\ are in this sense technically a
bit easier than our, more regular, problem!
\end{rem}

Thus, if $q_0=(0,y_0,\xi_0,\eta_0)\in\WFb^{1,m}(u)$, then there is a
GBB $\tilde\gamma$ with $\tilde\gamma(0)=q_0$ which is in $\WFb^{1,m}(u)$. If
$q_0$ is not glancing, this states that for small $\ep>0$, one
of the backward integral curve segments of $\sH_p$, defined over
$(-\ep,0]$, is in $\WFb^{1,m}(u)$. Since $\WFb^{1,m}(u)$ is just
$\WF^{m+1}(u)$ outside $Y$, we thus have that if
$q_0\in\WFb^{1,m}(u)$, then there is a backward integral curve segment
from $q_0$ which is in $\WF^{m+1}(u)$ over $(-\ep,0)$.

As a corollary we can now prove that
Theorem~\ref{thm:global-background} implies
Theorem~\ref{thm:microloc-background}:

\begin{proof}[Proof of Theorem~\ref{thm:microloc-background} given Theorem~\ref{thm:global-background}.]
By assumption, for some $\delta>0$, $u$ is in $H^{s-\ep_0}$
along the backward
bicharacteristics from $q_0$, i.e.\ $\WF^{s-\ep_0}(u)\cap
\tilde\gamma|_{(-\delta,0)}=\emptyset$ for all $\tilde\gamma$ with
$\tilde\gamma(0)=q_0$; note that for $\delta>0$ sufficiently small, these
are disjoint from $\Tb^*_Y X$. The wave front set being closed, there
is a neighborhood $U$ of these bicharacteristic segments disjoint from
$\WF^{s-\ep_0}(u)$.
Let $t$ be a global time function, which thus has
a derivative with a definite sign along $\sH_p$ depending on the
component of the characteristic set. Since the other case of similar,
we assume that $t$ is increasing along $\sH_p$ in the component of
$q_0$. Now let $t_0=t(q_0)$, and let
$$
T_2=\sup\{t(\tilde\gamma(-3\delta/4)):\ \tilde\gamma\ \text{a GBB},\
\tilde\gamma(0)=q_0\}<t_0,
$$
and let $T_1\in (T_2,t_0)$. Let
$$
K=\{\tilde\gamma(s):\ t(\tilde\gamma(s))\in[T_2,T_1],\ \tilde\gamma\ \text{a GBB},\
\tilde\gamma(0)=q_0\},
$$
which is thus compact, and if $\tilde\gamma(s)\in K$ then $s\in(-\delta,0)$,
so $\tilde\gamma(s)\notin \WF^{s-\ep_0}(u)$, so
$K\cap\WF^{s-\ep_0}(u)=\emptyset$ and $U$ is a neighborhood of $K$.
Let $\chi_0\in\CI(\RR)$ be such that $\chi_0\equiv 1$ near
$(-\infty,T_2]$, and $\chi_0\equiv 0$ near $[T_1,\infty)$, and let
$\chi=\chi_0\circ t$. Let
$\Box_+^{-1}$ denote the forward solution operator for $\Box$, i.e.\
given $f$ supported in $t>t_1$, $v=\Box_+^{-1}f$ is the unique
solution of $\Box v=f$ with $t>t_1$ on $\supp v$. Then
$$
u=\chi u-\Box_+^{-1}[\Box,\chi]u,
$$
since both sides solve $\Box w=0$ and the difference is supported in
$t\geq T_2$. Similarly, with $\Box_-^{-1}$ the backward solution operator,
$$
u=(1-\chi)u-\Box_-^{-1}[\Box,1-\chi]u=(1-\chi)u+\Box_-^{-1}[\Box,\chi]u,
$$
so
$$
u=(\Box_-^{-1}-\Box_+^{-1})[\Box,\chi]u.
$$
Moreover, for any $f$, $v=(\Box_-^{-1}-\Box_+^{-1})f$ solves $\Box
v=f$, and as $\WFb^{1,m}(\Box_+^{-1})(f)$ is contained in points from
which some backward GBB enters $\WFb^{-1,m-1}(f)$, and analogously
$\WFb^{1,m}(\Box_-^{-1})(f)$ is contained in points from
which some forward GBB enters $\WFb^{-1,m-1}(f)$, $\WFb^{1,m}(v)$ is
contained in GBB through $\WFb^{-1,m-1}(f)$.

So now let $Q\in\Psi^0(X)$ be such that $\WF'(Q)\subset U$ and
$\WF'(\Id-Q)\cap K=\emptyset$, and let
$$
u_0=(\Box_-^{-1}-\Box_+^{-1})Q[\Box,\chi]u,\ u_1=(\Box_-^{-1}-\Box_+^{-1})(\Id-Q)[\Box,\chi]u.
$$
We treat $u_0$ and $u_1$ separately.

We start with $u_1$.
We note that
backward bicharacteristics from $q_0$ cannot enter $\WF'(\Id-Q)\cap
T^*_{\supp d\chi}X$,
for if $\tilde\gamma$ is such a backward bicharacteristic from $q_0$ and
$\tilde\gamma(s)\in T^*_{\supp d\chi}X$, then $t(\tilde\gamma(s))\in [T_2,T_1]$,
so $\tilde\gamma(s)\in K$, which is
disjoint from $\WF'(\Id-Q)$. Correspondingly
$$
q_0\notin\WFb^{1,\infty}(u_1),
$$
and $\WFb^{1,\infty}(u_1)$ is disjoint from forward bicharacteristic
segments from $q_0$, in particular, for sufficiently small $s>0$,
for which $\tilde\gamma(s)\notin \Tb^*_YX$, $\tilde\gamma(s)\notin\WF(u_1)$.

Now we turn to $u_0$. As
$\WF^{s-\ep_0-1}([\Box,\chi]u)\subset\WF^{s-\ep_0}(u)\cap
T^*_{\supp d\chi}X$ is disjoint from $U$, we deduce that $Q[\Box,\chi]u\in
H^{s-\ep_0-1}$, and thus
$$
u_0=(\Box_-^{-1}-\Box_+^{-1})Q[\Box,\chi]u\in
H_{\bl,\loc}^{1,s-\ep_0-1}(X).
$$
In particular, $u_0\in L^2_{\loc}$ as $s-\ep_0\geq 0$. By
Corollary~\ref{cor:forward-solution-reg}, $u_0\in H^{s-\ep_0}_{\loc}$.
Moreover, with $\gamma_0$ denoting the integral curve of $\sH_p$
through $(0,y_0,\xi_0,\eta_0)$, $\gamma_0|_{(-\delta,0)}$ is disjoint from
$\WF^s(u_0)$ since the analogous statement is true for $u$. Thus,
Theorem~\ref{thm:global-background} is applicable to $u_0$, giving
that all of $\gamma_0$ is disjoint from $\WF^s(u_0)$. Combining with
the result on $u_1$, Theorem~\ref{thm:microloc-background} is proved.
\end{proof}

\section{Paired Lagrangian distributions}\label{sec:paired-Lag}

The class of distributions that plays the starring role below is that
of paired Lagrangian distributions associated to two cleanly
intersecting Lagrangians with the intersection having codimension $k$;
these were introduced by Guillemin and Uhlmann
\cite{Guillemin-Uhlmann:Oscillatory} following the codimension $1$
work of Melrose and Uhlmann \cite{Melrose-Uhlmann:Intersection}.
In the model case where these Lagrangians are
$\tilde\Lambda_0=T^*_0\RR^n$ and $\tilde\Lambda_1=N^*\{x''=0\}$ in
$T^*\RR^n$ where the coordinates on $\RR^n$ are
$x=(x',x'')\in\RR^{k}\times\RR^{n-k}$,
these (compactly supported) elements of $I^{p,l}(\tilde\Lambda_0,\tilde\Lambda_1)$
are defined in \cite{Guillemin-Uhlmann:Oscillatory}, modulo $\CI_c(\RR^n)$, by oscillatory
integrals of the form
\begin{equation}\label{eq:Gui-Uhl-paired}
\int e^{i[(x'-s)\zeta'+x''\zeta''+s\sigma]}a(x,s,\zeta,\sigma)\,ds\,d\zeta\,d\sigma,
\end{equation}
$a$ being a {\em product type} symbol $a\in
S^{M,M'}(\RR^{n+k}_{x,s},\RR^n_\zeta,\RR^k_\sigma)$ with
$M=p-n/4+k/2$, $M'=l-k/2$ and with compact support in $x,s$, and in general
via reduction to this model Lagrangian pair via a Fourier integral
operator. Here $a\in
S^{M,M'}(\RR^{n+k}_{x,s},\RR^n_\zeta,\RR^k_\sigma)$ means that
$$
|(D^\alpha_{x,s} D^\beta_\zeta D^\gamma_\sigma
a)(x,s,\zeta,\sigma)|\leq C_{\alpha\beta\gamma} \langle\zeta\rangle^{M-|\beta|}\langle\sigma\rangle^{M'-|\gamma|}.
$$
Such a distribution is, microlocally away from
$\tilde\Lambda_0\cap\tilde\Lambda_1$,
in $I^p(\tilde\Lambda_1\setminus\tilde\Lambda_0)$ and in
$I^{p+l}(\tilde\Lambda_0\setminus\tilde\Lambda_1)$. It is important to
realize that these distributions are {\em not} a simple extension of
these two
classes of Lagrangian distributions, and in particular it is {\em not} the case that
$I^{p+l}(\tilde\Lambda_0)\subset
I^{p,l}(\tilde\Lambda_0,\tilde\Lambda_1)$ for all $p,l$, though this
inclusion of course holds away from
$\tilde\Lambda_0\cap\tilde\Lambda_1$. In fact, what is true is
$$
I^{p}(\tilde\Lambda_0)\subset I^{p-k/2,k/2}(\tilde\Lambda_0,\tilde\Lambda_1);
$$
we show this below in Lemma~\ref{lemma:Lambda_0-Lag}. On the other
hand, $I^p(\tilde \Lambda_1)\subset
I^{p,l}(\tilde\Lambda_0,\tilde\Lambda_1)$, so there is a fundamental
asymmetry between the two Lagragians.

Indeed, this model can be simplified as follows.  A distribution $u$
is in $I^{p,l}(\tilde\Lambda_0,\tilde\Lambda_1)$, modulo
$\CI_c(\RR^n)$, if it can be written as
$$
\int e^{i[x'\zeta'+x''\zeta'']}b(x,\zeta)\,d\zeta,
$$
i.e.\ is essentially the inverse Fourier transform of $b$, with $b$
satisfying the following estimates with $M=p-n/4+k/2$, $M'=l-k/2$ as before:
First, in the region $|\zeta'|\leq C'|\zeta''|$, $|\zeta''|\geq
1$, the conditions on
$b$ amount to
$$
|(Qb)(x,\zeta)|\leq C\langle\zeta''\rangle^{M}\langle\zeta'\rangle^{M'}
$$
whenever $Q$ is a finite product of differential operators of the form
$$
D_{\zeta'_m},\ \zeta'_j D_{\zeta'_m},\ \zeta''_j D_{\zeta''_m},
$$
i.e.\ standard product-type regularity, when localized to this
region. (Note that by localizing to the region where $\zeta''_q$, for
instance, dominates the other $\zeta''_j$, one may simply replace
$\zeta''_j$ by $\zeta''_q$, as may be convenient on occasion.)
On the other hand, in the region where $|\zeta''|\leq
C''|\zeta'|$, $|\zeta'|\geq 1$, which maps to $\tilde\Lambda_0$ away from the intersection of
$\tilde\Lambda_0$ and $\tilde\Lambda_1$ and is not of too much interest, one has
standard symbolic regularity, i.e.
$$
|(Qb)(x,\zeta)|\leq C\langle\zeta'\rangle^{M+M'}
$$
whenever $Q$ is a finite product of differential operators of the form
$$
\zeta'_j D_{\zeta'_m},\ \zeta'_j D_{\zeta''_m}.
$$
Alternatively, altogether, without any localization, one has bounds
\begin{equation}\label{eq:iterated-est-for-b}
|(Qb)(x,\zeta)|\leq C\langle\zeta\rangle^{M}\langle\zeta'\rangle^{M'}
\end{equation}
whenever $Q$ is a finite product of differential operators of the form
\begin{equation}\label{eq:diff-factors-for-b-est}
D_{\zeta'_m},\ \zeta'_j D_{\zeta'_m},\ D_{\zeta''_m},\ \zeta''_j
D_{\zeta''_m},\ \zeta'_j D_{\zeta''_m}.
\end{equation}
One direction of this equivalence
claim is easily shown by starting from \eqref{eq:Gui-Uhl-paired}
by taking
$$
b(x,\zeta)=\int
e^{is(\sigma-\zeta')}a(x,s,\zeta,\sigma)\,ds\,d\sigma=\int(\cF' a)(x,\zeta'-\sigma,\zeta,\sigma)\,d\sigma,
$$
where $\cF'$ is Fourier transform in the second slot (so $\cF' a$ is
Schwartz in this variable!) and directly checking the stability
estimates. For the converse, if $b$ is supported in
$|\zeta'|<C'|\zeta''|$, as one may assume, one can take
$$
a(x,s,\zeta,\sigma)=(2\pi)^{-k} b(x,\sigma,\zeta'')\chi(\langle\zeta'\rangle/\langle\zeta''\rangle)\chi_0(s),
$$
where $\chi\in\CI_c(\RR)$ is identically $1$ on $[0,2C']$, while
$\chi_0\in\CI_c(\RR^k)$ is such that if $b$ is supported in $|x|<R$
then $\chi_0(s)$ is identically $1$ on $|s|<2R$. Here the localizer
$\chi$ makes $a$ into a symbol of the desired product type in
$(\zeta,\sigma)$, while $\chi_0$ localizes the support in $s$. With
this definition of $a$,
$$
\int(\cF' a)(x,\zeta'-\sigma,\zeta,\sigma)\,d\sigma=\chi(\langle\zeta'\rangle/\langle\zeta''\rangle)\int
b(x,\sigma,\zeta'')
(2\pi)^{-k}\hat\chi_0(\zeta'-\sigma)\,d\sigma;
$$
by the support conditions on $\chi$ and $b$ and as $\hat\chi_0$ is
Schwartz, dropping the factor $\chi$ only causes a Schwartz
error to obtain $\tilde b(x,\zeta)=\int
b(x,\sigma,\zeta'')
(2\pi)^{-k}\hat\chi_0(\zeta'-\sigma)\,d\sigma$. Now,
$$
\int e^{ix'\cdot\zeta'}\tilde b(x,\zeta')\,d\zeta'
=\chi_0(x)\int e^{ix'\cdot\zeta'} b(x,\zeta')\,d\zeta'=\int e^{ix'\cdot\zeta'} b(x,\zeta')\,d\zeta',
$$ 
so the distributions defined by $a$ and $b$ differ by an element of
$\CI(\RR^n)$ as claimed.

We remark that, although we do not use this point of view here, the
regularity statement
\eqref{eq:iterated-est-for-b}-\eqref{eq:diff-factors-for-b-est} for
$b$ amount to the statement that
$b$ is a conormal
function on the blow up of $\RR^n\times\overline{\RR^n}$, with the
second factor radially compactified, at
$\RR^n\times\pa\overline{\RR^{n-k}_{\zeta''}}$, i.e.\ at infinity in
$\zeta$ where $\zeta'=0$, with order $M$ on the front face, and order
$M+M'$ on the lift of $\RR^n\times\pa\overline{\RR^n}$, where
$M=p-n/4+k/2$, $M'=l-k/2$ as before.

Indeed, a further argument shows that first, modulo $\CI(\RR^n)$, the
$x''$ dependence of $b$ can be eliminated via expanding $b$ in Taylor
series
around $x''=0$ and noting that $(x'')^\alpha$ becomes
$(-1)^{|\alpha|}D_{\zeta''}^\alpha$ after an integration by parts, so
in view of the symbolic estimates in $\zeta''$ corresponds to reduced
$p$, with an asymptotic summation argument completing the
argument. Next, modulo $I^p(\tilde\Lambda_1)$, the $x'$ dependence of $b$
can be eliminated by a similar argument, expanding in Taylor series in
$x'$, which via integration by parts gives
$(-1)^{\alpha}D_{\zeta'}^\alpha$, thus reducing $l$, which via an
asymptotic summation argument completes the claim. Hence, it may be
assumed that, modulo a term in $I^p(\tilde\Lambda_1)$, a paired Lagrangian
distribution is the inverse Fourier transform of a conormal
function on the blow up of $\overline{\RR^n}$ at
$\pa\overline{\RR^{n-k}_{\zeta''}}$, i.e.\ at infinity in
$\zeta$ where $\zeta'=0$, with order $M$ on the front face, and order
$M+M'$ on the lift of $\pa\overline{\RR^n}$, where
$M=p-n/4+k/2$, $M'=l-k/2$ as before.

One immediate consequence is:

\begin{lemma}\label{lemma:filter}
If $p_1\leq p_2$ and $p_1+l_1\leq p_2+l_2$ then
$I^{p_1,l_1}(\Lambda_0,\Lambda_1)\subset I^{p_2,l_2}(\Lambda_0,\Lambda_1)$.
\end{lemma}

\begin{proof}
It suffices to consider the model pair, $(\tilde\Lambda_0,\tilde\Lambda_1)$.
Since the class of differential operators under which one has
stability in the two cases is the same, one just has to remark that
for $p_1'\leq p_2'$, $p_1'+l_1'\leq p_2'+l_2'$,
$$
\langle\zeta\rangle^{p_1'}\langle\zeta'\rangle^{l_1'}\leq
\langle\zeta\rangle^{p_1'}\langle\zeta'\rangle^{l_2'}\langle\zeta'\rangle^{p_2'-p_1'}
\leq
\langle\zeta\rangle^{p_1'}\langle\zeta'\rangle^{l_2'}\langle\zeta\rangle^{p_2'-p_1'}
= \langle\zeta\rangle^{p_2'}\langle\zeta'\rangle^{l_2'}.
$$
\end{proof}

Another immediate consequence is:

\begin{lemma}\label{lemma:Lambda_0-Lag}
$$
I^{p}(\Lambda_0)\subset I^{p-k/2,k/2}(\Lambda_0,\Lambda_1).
$$
\end{lemma}

\begin{proof}
Again, it suffices to consider the model pair,
$(\tilde\Lambda_0,\tilde\Lambda_1)$. An element of $I^p(\tilde\Lambda_0)$
can be written, modulo $\CI(\RR^n)$, as the inverse Fourier transform
of a symbol in $S^{p-\frac{n}{4}}(\RR^n)$. But
$S^{p-\frac{n}{4}}(\RR^n)$ is conormal on $\overline{\RR^n}$, of order
$p-\frac{n}{4}$, hence on its blow up at
$\pa\overline{\RR^{n-k}_{\zeta''}}$, with order $M=M+M'=p-n/4$ {\em
  both} on the front face, and on the lift of $\pa\overline{\RR^n}$. In
terms of $I^{\tilde p,\tilde l}(\tilde\Lambda_0,\tilde\Lambda_1)$ this
corresponds to orders $\tilde p=p-k/2$, $\tilde l=k/2$, proving the lemma.
\end{proof}

Note from the proof that one {\em cannot} lower $\tilde p=p-k/2$ even by
increasing $\tilde l=k/2$. In fact, on the one hand, for an element of
$S^{\tilde p',\tilde l'}$ the growth rate at the front face is
determined by $\tilde p'$ alone, and on the other hand for $u\in
I^p(\tilde\Lambda_0)$, the growth rate at this place is determined by
$p$ in general (i.e.\ there is no extra decay at the front face
compared to other directions).

One can now easily describe the principal symbol on $\Lambda_1$ in
general (without homogeneity discussions as in
\cite{Guillemin-Uhlmann:Oscillatory}). For this purpose it is useful
to work with half-densities to avoid having to tensor with bundles
that vary with the particular problem we want to study (such as
half-density bundles from the base space $X$, or a factor of the base
space on product spaces $X=X_L\times X_R$). Since the half-density bundles are trivial,
from now on, without further comments, we trivialize them on the base manifold, as well as its
factors, so as to {\em regard distributions (e.g.\ elements of
$I^{p,l}(\Lambda_0,\Lambda_1)$) as distributional half-densities, and
distributions with values in densities on the right factor $X_R$
(which are the Schwartz kernels of operators acting on functions) also
as distributional half-densities}.

\begin{lemma}\label{lemma:basic-symbol}
Suppose $u\in I^{p,l}(\Lambda_0,\Lambda_1)$ given by an inverse Fourier
transform $\cF^{-1} b$, $b$ conormal on the blow up of
$[\overline{\RR^n_\zeta};
\pa\overline{\RR^{n-k}_{\zeta''}}]$ supported in
$\langle\zeta'\rangle\leq C\langle\zeta''\rangle$, with order $M$ on the front face, and order
$M+M'$ on the lift of $\pa\overline{\RR^n}$, where
$M=p-n/4+k/2$, $M'=l-k/2$ as before.
Let $a=(\cF')^{-1} b$, where $\cF'$ is partial Fourier transform in
the primed variables. Then
\begin{equation}\label{eq:Lambda_1-symb-as-IFT}
a\in
S^{p-n/4+k/2}(\RR^{n-k}_{\zeta''};I^{M'+\frac{k}{4}}(\RR^k_{x'};N^*\{0\})),
\end{equation}
and the equivalence class of $a$ modulo
$S^{p-n/4+k/2-1}(\RR^k_{x'}\setminus 0;\RR^{n-k}_{\zeta''})$ satisfies
\begin{equation}\label{eq:Lambda_1-princ-symbol}
[(2\pi)^{\frac{(n-2k)}{4}}a|_{x'\neq 0}\,|dx'|^{1/2} |d\zeta''|^{1/2}]=\sigma_{\Lambda_1\setminus\Lambda_0,p}(u),
\end{equation}
with the right hand side being the standard principal symbol of a (microlocal)
element of $I^p(\Lambda_1)$. The equivalence class of
$$
(2\pi)^{\frac{(n-2k)}{4}}a\,|dx'|^{1/2}\,|d\zeta''|^{1/2}
\ \text{modulo}
\
S^{p-1-n/4+k/2}(\RR^{n-k}_{\zeta''};I^{M'+1+\frac{k}{4}}(\RR^k_{x'};N^*\{0\}))
$$
is the principal symbol of $u$ on $\Lambda_1$, which is well-defined.

Furthermore,
$$
a\in
S^{p-1-n/4+k/2}(\RR^{n-k}_{\zeta''};I^{M'+1+\frac{k}{4}}(\RR^k_{x'};N^*\{0\}))
\Longrightarrow u\in I^{p-1,l+1}(\Lambda_0,\Lambda_1),
$$
while if
$$
\tilde a\in
S^{p-n/4+k/2}(\RR^{n-k}_{\zeta''};I^{M'+\frac{k}{4}}(\RR^k_{x'};N^*\{0\}))
$$
then there is $u\in I^{p,l}(\Lambda_0,\Lambda_1)$ such that the
principal symbol of $u$ on $\Lambda_1$ is $\tilde a$.
\end{lemma}

\begin{proof}
Note that elements of $S^{M,M'}$ with the stated support condition are
exactly the functions on $\RR^n$ with a bound
$|b|\leq C\langle\zeta''\rangle^M\langle\zeta'\rangle^{M'}$ which is
stable upon iteratively applying finite products of
$\zeta'_jD_{\zeta'_m}, D_{\zeta'_m},
\zeta''_jD_{\zeta''_m},D_{\zeta''_m}$ to $b$, so it consists exactly
of elements of $S^M(\RR^{n-k}_{\zeta''};S^{M'}(\RR^k))$ with the
stated support.
Since the partial inverse Fourier transform in the primed variables
maps $S^{M'}(\RR^k_{x'})$ contiuously to
$I^{M'+\frac{k}{4}}(\RR^k_{x'};N^*\{0\})$,
\eqref{eq:Lambda_1-symb-as-IFT} follows immediately. As the standard
parameterization of a conormal distribution in $I^p(\Lambda_1)$ is
$$
(2\pi)^{-(n+2(n-k))/4}\int e^{ix''\cdot\zeta''}\tilde
a(x',x'',\zeta'')\,d\zeta'',
$$
with $\tilde a\in S^{p+(n-2(n-k))/4}(\RR^n_x;\RR^{n-k}_{\zeta''})$
with principal symbol given by the equivalence class of the restriction of $\tilde a$ to
$x'=0$, while
$$
u=(2\pi)^{-n+k}\int e^{ix''\cdot\zeta''}(\cF')^{-1}b (x',\zeta'')\,d\zeta'',
$$
with $(\cF')^{-1}b (x',\zeta'')$ in
$S^M(\RR^{n-k}_{\zeta''};\CI(\RR^k\setminus 0))$,
\eqref{eq:Lambda_1-princ-symbol} follows.

Since conversely we have that the partial Fourier transform in the
primed variables maps
$S^{M}(\RR^{n-k}_{\zeta''};I^{M'+\frac{k}{4}}(\RR^k_{x'};N^*\{0\}))$
to $S^M(\RR^{n-k}_{\zeta''};S^{M'}(\RR^k))$, if $u=\cF^{-1}b$, and $b\in S^{M,M'}$ satisfies
$$
(\cF')^{-1}b\in
S^{M-1}(\RR^{n-k}_{\zeta''};I^{M'+1+\frac{k}{4}}(\RR^k_{x'};N^*\{0\})),
$$
then $b\in S^{M-1,M'+1}$ and thus $u\in I^{p-1,l+1}$. Further, if
$$
\tilde a\in
S^{p-n/4+k/2}(\RR^{n-k}_{\zeta''};I^{M'+\frac{k}{4}}(\RR^k_{x'};N^*\{0\}))
$$
then defining $b$ to be  $(2\pi)^{-\frac{(n-2k)}{4}}(\cF'\tilde
a)\chi$, where $\chi$ is a symbol on $\RR^n$, with support in
$\langle\zeta'\rangle<2\langle\zeta''\rangle$, identically $1$ on
$\langle\zeta'\rangle<\langle\zeta''\rangle$, then
$b-(2\pi)^{-\frac{(n-2k)}{4}}(\cF'\tilde a)\in S^{M-N,M'+N}$ for every
$N\geq 0$, and thus
$$
(2\pi)^{\frac{(n-2k)}{4}} (\cF')^{-1}b-\tilde a\in
S^{M-1}(\RR^{n-k}_{\zeta''};I^{M'+1+\frac{k}{4}}(\RR^k_{x'};N^*\{0\}))
$$
as claimed.
\end{proof}

This description of paired Lagrangians is rather convenient for describing what happens when
$\Lambda_0$ and $\Lambda_1$ are interchanged.

\begin{prop}\label{prop:reverse-order}
For $l<-k/2$ and $N\in\NN$ such that $l+N<-k/2$ one has
$$
I^{p,l}(\Lambda_0,\Lambda_1)\subset I^p(\Lambda_1)+I^{p-N-\frac{k}{2},N+\frac{k}{2}}(\Lambda_1,\Lambda_0).
$$
On the other hand, for $l>-k/2$,
$$
I^{p,l}(\Lambda_0,\Lambda_1)\subset I^{p+l,\frac{k}{2}}(\Lambda_1,\Lambda_0).
$$
In both cases the inclusion maps are continuous, i.e.\ in the first
case, when restricted to distributions with support in a fixed
compact set, for any $M$ there is $M'$ and $C>0$
such that for $u\in I^{p,l}(\Lambda_0,\Lambda_1)$ there are
$I^p(\Lambda_1)$ and  $u_2\in
I^{p-N-\frac{k}{2},N+\frac{k}{2}}(\Lambda_1,\Lambda_0)$ with
\begin{equation}\label{eq:reverse-order-cont}
\|u_1\|_{I^p(\Lambda_1);M}+\|u_2\|_{I^{p-N-\frac{k}{2},N+\frac{k}{2}}(\Lambda_1,\Lambda_0);M}\leq C\|u\|_{I^{p,l}(\Lambda_0,\Lambda_1);M'},
\end{equation}
where $\|.\|_{I^p(\Lambda_1);M}$, etc., denotes the $M$th seminorm giving the
topology on $I^p(\Lambda_1)$, etc.
\end{prop}

Note that when $l>-k/2$, $I^p(\Lambda_1)\subset
I^{p,l}(\Lambda_0,\Lambda_1)$
is included in $I^{p+l,\frac{k}{2}}(\Lambda_1,\Lambda_0)$ by
Lemma~\ref{lemma:Lambda_0-Lag} and Lemma~\ref{lemma:filter}, while the
same conclusion does not hold when $l<-k/2$ necessitating the addition of
$I^p(\Lambda_1)$ explicitly to the right hand side.

\begin{proof}
As usual, it suffices to consider the model Lagrangians.
It is straightforward to write down an explicit homogeneous
symplectomorphism, and quantize it as a Fourier integral operator,
microlocally near $\tilde\Lambda_0\cap\tilde\Lambda_1$. Explicitly,
where $C|\xi''_q|>\langle\xi\rangle$, as one may always arrange
microlocally near a point in the intersection by suitably picking the index $q$, letting
$e_q$ be the corresponding coordinate unit vector,
one can take the symplectomorphism
$$
(x',x'',\xi',\xi'')\mapsto (-\frac{\xi'}{\xi''_q},x''+\frac{x'\cdot
  \xi'}{\xi''_q}\,e_q,\xi''_k x',\xi''),
$$
and quantize it as
\begin{equation*}\begin{aligned}
Fu(y)&=\int e^{i(y''-x''+(x'\cdot y')e_q)\cdot\xi''}|\xi''_q|^{k/2}\,u(x)\,dx\,d\xi''\\
&=\int e^{iy''\cdot\xi''}|\xi''_q|^{k/2}(\cF u)(-\xi''_q
y', \xi'')\,d\xi'',
\end{aligned}\end{equation*}
where the symbol $|\xi''_q|^{k/2}$ is chosen to make $F$ elliptic of
order $0$. Thus, for $u\in I^{p,l}(\tilde\Lambda_0,\tilde\Lambda_1)$,
assuming as we may that $u$ is the inverse Fourier transform of an
element $b$ of $S^{p',l'}$ with $p'=p-n/4+k/2$, $l'=l-k/2$, and with
support in $|\xi|\leq C|\xi''_q|$, $|\xi''_q|\geq 1$,
\begin{equation}\begin{aligned}
Fu(y)&=\int e^{iy''\cdot\xi''}|\xi''_q|^{k/2}b(-\xi''_q
y', \xi'')\,d\xi''
=\int e^{i(y'\cdot\xi'+y''\cdot\xi'')}|\xi''_q|^{k/2}(\cF' \tilde b)(\xi',\xi'')\,d\xi,\\
&\tilde b(\zeta',\zeta'')=b(\zeta''_q\zeta',\zeta''),
\end{aligned}\end{equation}
where
$$
\cF' \tilde b(\xi',\xi'')=(2\pi)^{-k}\int e^{i\xi'\cdot\zeta'} \tilde b(\zeta',\xi'')\,d\zeta'
$$
is the partial inverse Fourier transform of $\tilde b$.
Thus, $Fu$ is (up to a constant factor) the inverse Fourier transform
of
$$
a(\xi',\xi'')=|\xi''_q|^{k/2}(\cF' \tilde
b)(\xi',\xi'')=|\xi''_q|^{-k/2}(\cF'
b)((\xi''_q)^{-1}\xi',\xi'')=(\tilde\cF' b)(\xi',\xi''),
$$
with $\tilde\cF'$ defined by the last equation,
and in order to prove the proposition, we only need to show that (with
$p'=p-n/4+k/2$, $l'=l-k/2$)
\begin{equation}\begin{aligned}\label{eq:pFT-of-symbols}
&l'>-k\Rightarrow \tilde\cF' S^{p',l'}\subset S^{p'+l'+k/2,0}\\
&l'<-k\Rightarrow \tilde\cF' S^{p',l'}\subset S^{p'-k/2}(\RR^n)+ S^{p'-N-k/2,N},
\end{aligned}\end{equation}
with continuous inclusions.
We first prove the first implication as well as the second in the
special case $N=0$, when the first term on the right hand side can be
absorbed in the second. Since it is straightforward to check that the
differential operators under which we require iterative regularity
transform properly, the main issue is to obtain $\sup$ bounds.
But
\begin{equation}\begin{aligned}
&|(\tilde\cF' b)(\xi',\xi'')|\leq |\xi''_q|^{-k/2}\int
|b(\zeta',\xi'')|\,d\zeta'\lesssim |\xi''_q|^{-k/2}\int
\langle \zeta'\rangle^{l'}|\xi''_q|^{p'}\,d\zeta'\\
&\leq |\xi''_q|^{p'-k/2}\Big(\int_{|\zeta'|\leq 1}\,d\zeta'+\int_{1\leq
  |\zeta'|\leq C|\xi''_q|} |\zeta'|^{l'}\,d\zeta'\Big)\lesssim
|\xi''_q|^{p'-k/2}(1+|\xi''_q|^{l'+k}),
\end{aligned}\end{equation}
so the conclusion immediately follows. (We remark that if $l'=-k$, a
logarithmic term in $|\xi''_q|$ would appear on the right hand side,
so in terms of spaces with polynomial weights, we would have to lose
$\ep>0$ to end up in $S^{p'-k/2+\ep,0}$, which is the result one
obtains if one simply replaces $l'$ by $l'+\ep$ and applies the
statement in that case, hence not stating the case $l=-k/2$ separately.)
Now, for general $N\geq 1$, we expand $(\tilde\cF' b)(\xi',\xi'')$ in Taylor
series around $\xi'=0$ to order $N-1$,
\begin{equation}\begin{aligned}\label{eq:symbol-Taylor-exp}
(\tilde\cF' b)(\xi',\xi'')=&\sum_{|\alpha|\leq N-1}\frac{1}{\alpha !}
(\xi')^{\alpha}\pa_{\xi'}^\alpha(\tilde\cF'
b)(0,\xi'')\\
&\qquad+\sum_{|\alpha|=N}\frac{N}{\alpha!}\int_0^1 (1-t)^{N-1}(\xi')^{\alpha}\pa_{\xi'}^\alpha(\tilde\cF'
b)(t\xi',\xi'')\,dt,
\end{aligned}\end{equation}
and check that the two terms are respectively in $S^{p'-k/2}(\RR^n)$
and $S^{p'-N-k/2,N}$ when microlocalized to $|\xi'|\leq \tilde
C|\xi''_q|$. Here the key point is that
$$
\xi'^\alpha D_{\xi'}^\alpha(\cF'
b)((\xi''_q)^{-1}\xi',\xi'')=(-1)^{|\alpha|}\xi'^\alpha
|\xi''_q|^{-|\alpha|}(\cF' ((M')^\alpha b)) ((\xi''_q)^{-1}\xi',\xi''),
$$
where $M'_j$ is multiplication by the $j$th primed coordinate
function, with $(M')^\alpha$ then defined by the standard multiindex
notation, so $(M')^\alpha b\in S^{p',l'+|\alpha|}$, and thus, in view
of \eqref{eq:pFT-of-symbols} with the already proved case, $N=0$, the
$\alpha$th term in \eqref{eq:symbol-Taylor-exp} is in
$S^{p'+l'+k/2,|\alpha|}$ if $l'+|\alpha|>-k$, and in
$S^{p'-|\alpha|-k/2,|\alpha|}$ if $l'+|\alpha|<-k$, with the
additional information (in view of the evaluation at $\xi'=0$) that if $|\alpha|<N$ then in fact
the $\alpha$th term is in $S^{p'+l'+|\alpha|+k/2}(\RR^n)$ if
$l'+|\alpha|>-k$, and in
$S^{p'-k/2}(\RR^n)$ if $l'+|\alpha|<-k$. This proves
\eqref{eq:pFT-of-symbols} when $N\geq 0$ is an integer with
$l'+N<-k$, and thus the proposition.
\end{proof}

\begin{cor}\label{cor:reverse-contain}
For $l<-k/2$ and for $\ep>0$,
$$
I^{p,l}(\Lambda_0,\Lambda_1)\subset I^p(\Lambda_1)+I^{p+l+\ep,-l-\ep}(\Lambda_1,\Lambda_0),
$$
with continuous inclusions.
\end{cor}

Note that the second summand has order $p$ on $\Lambda_1$, i.e.\ it
did not increase when reversing the order of the two Lagrangians,
while it has order $p+l+\ep$ on $\Lambda_0$, so it only increased by
$\ep$ as compared to the left hand side. This is an affordable loss
when $\Lambda_0$ is thought of as carrying a `small singularity',
while any loss on $\Lambda_1$ is unaffordable.

Also note that in view
of Lemma~\ref{lemma:filter}, the corollary indeed becomes stronger if
one decreases $\ep$.

\begin{proof}
If $l=-k/2-\ep-N$ for some $N\in\NN$ and $\ep>0$, then this is just
Proposition~\ref{prop:reverse-order}. Below we assume that seminorms
are actually norms, as we may (by including the weighted $\sup$ norm without
derivatives on the symbol in all of them), and that they get stronger
with increasing index $M$. Let
$I_{M'}^{p-k/2-\ep}(\Lambda_0,\Lambda_1)$ denote the completion of
$I^{p-k/2-\ep}(\Lambda_0,\Lambda_1)$ with respect to the $M'$th norm,
so the result is a Banach space;
thus, for $M''\geq M'$, the identity map on $I^{p-k/2-\ep}(\Lambda_0,\Lambda_1)$ extends
to a continuous map
$$
I_{M''}^{p-k/2-\ep}(\Lambda_0,\Lambda_1)\to I_{M'}^{p-k/2-\ep}(\Lambda_0,\Lambda_1),
$$
and the completeness of $I^{p-k/2-\ep}(\Lambda_0,\Lambda_1)$ as a
Fr\'echet space means that
$$
I^{p-k/2-\ep}(\Lambda_0,\Lambda_1)=\cap_{M'}I_{M'}^{p-k/2-\ep}(\Lambda_0,\Lambda_1).
$$
Indeed, if $u\in \cap_{M}I_{M}^{p-k/2-\ep}(\Lambda_0,\Lambda_1)$
then for each $M$ there is a Cauchy sequence in
$I^{p-k/2-\ep}(\Lambda_0,\Lambda_1)$ converging to $u$; we may assume
that this Cauchy sequence is of the form $\{u_{M,j}\}_{j=1}^\infty$
with $\|u_{M,j}-u\|_M\leq 2^{-j}$. Then the diagonal sequence
$u_j=u_{j,j}$ is Cauchy with respect to all norms $M$, and it
converges to $u$ in all of these, so by the completeness of
$I^{p-k/2-\ep}(\Lambda_0,\Lambda_1)$, $u\in I^{p-k/2-\ep}(\Lambda_0,\Lambda_1)$.
We use similar notation for completions of other spaces with respect
to various norms below.

The complex interpolation spaces for
$I_{M'}^{p,-k/2-\ep}(\Lambda_0,\Lambda_1)$ and
$I_{M'}^{p,-k/2-\ep-N}(\Lambda_0,\Lambda_1)$ are
$I_{M'}^{p,-k/2-\ep-N\theta}(\Lambda_0,\Lambda_1)$, $\theta\in[0,1]$,
since in the interpolation only the weight corresponding to
$\Lambda_0$ is changed, and the seminorms are weighted $L^\infty$
bounds, i.e.\ the interpolation is actually for a family of multiplication operators.
Similarly,
the complex interpolation spaces between
$I_M^{p-k/2,k/2}(\Lambda_1,\Lambda_0)$, and $I_M^{p-N-k/2,N+k/2}(\Lambda_1,\Lambda_0)$,
are $I_M^{p-N\theta-k/2,N\theta+k/2}(\Lambda_1,\Lambda_0)$,
$\theta\in[0,1]$;
now both weights are interpolated, but this still is interpolation
for a family of multiplication operators.
In view of the continuity of the
inclusion map
$$
I^{p,-k/2-\ep-N}(\Lambda_0,\Lambda_1)\hookrightarrow
I^p(\Lambda_1)+I^{p-N-k/2,N+k/2}(\Lambda_1,\Lambda_0)
$$
for $N\in\NN$, for all $M$ there is $M'$ such that
the inclusion map extends to a map from the $M'$th
completion of the left hand side to the $M$th completion of the right
hand side. Thus, complex interpolation is applicable, and yields that
$$
I_{M'}^{p,-k/2-\ep-N\theta}(\Lambda_0,\Lambda_1)\hookrightarrow
I_M^p(\Lambda_1)+I_M^{p-N\theta-k/2,N\theta+k/2}(\Lambda_1,\Lambda_0),\ \theta\in[0,1].
$$
In particular, as
$$
I^{p,-k/2-\ep-N\theta}(\Lambda_0,\Lambda_1)\subset
I_{M'}^{p,-k/2-\ep-N\theta}(\Lambda_0,\Lambda_1),
$$
the inclusion map extends to
$$
I^{p,-k/2-\ep-N\theta}(\Lambda_0,\Lambda_1)\hookrightarrow
I_M^p(\Lambda_1)+I_M^{p-N\theta-k/2,N\theta+k/2}(\Lambda_1,\Lambda_0),\ \theta\in[0,1],
$$
for all $M$, with the spaces on the right becoming stronger with
$M$. Since the intersections of these spaces is
$I^p(\Lambda_1)+I^{p-N\theta-k/2,N\theta+k/2}(\Lambda_1,\Lambda_0)$,
we deduce that
$$
I^{p,-k/2-\ep-N\theta}(\Lambda_0,\Lambda_1)\hookrightarrow
I^p(\Lambda_1)+I^{p-N\theta-k/2,N\theta+k/2}(\Lambda_1,\Lambda_0),\ \theta\in[0,1].
$$
As $N\in\NN$ is arbitrary,
$$
I^{p,-k/2-\ep-m}(\Lambda_0,\Lambda_1)\subset I^p(\Lambda_1)+I^{p-m-k/2,m+k/2}(\Lambda_1,\Lambda_0).
$$
when $m\geq 0$ real, which is just a rewriting of the statement of the corollary.
\end{proof}

We also recall the composition rule of Antoniano and Uhlmann
\cite{Antoniano-Uhlmann:Functional} for flow-outs, with $\Lambda_1$
the flow-out of $\Lambda_0=N^*\diag$, as referred to in
\cite[Proposition~1.39]{Greenleaf-Uhlmann:Estimates}, namely (with $k$
the codimension of the intersection)
\begin{equation}\label{eq:Ant-Uhl-compose}
I^{p,l}(\Lambda_0,\Lambda_1)\circ I^{p',l'}(\Lambda_0,\Lambda_1)\subset I^{p+p'+k/2,l+l'-k/2}(\Lambda_0,\Lambda_1).
\end{equation}
We recall the set-up of flow-outs here, phrased in the general
codimension case as in 
Greenleaf and Uhlmann \cite{Greenleaf-Uhlmann:Estimates}. Thus,
one of the Lagrangians is the conormal bundle $N^*\diag$ of the
diagonal, and the other is a flow-out $\Lambda=\Lambda_\Gamma$
corresponding to a conic, codimension $k$, involutive (i.e.\ coisotropic)
$\Gamma\subset T^*\RR^n$. Such a $\Gamma$ is defined by the vanishing
of $k$ functions $p_i$ which Poisson commute on $\Gamma$; $\Lambda_\Gamma$
is then the set of points $((x,\xi),(y,-\eta))\in T^*\RR^{2n}$ such
that $(y,\eta)=\exp(\sum t_j \sH_{p_j})(x,\xi)$ for some
$t\in\RR^k$. We give a concrete example: if $\Gamma=T^*_Y X$ with $Y$ defined
by $x'=0$, $x'\in\RR^k$, then one can take $x'_1,\ldots,x'_k$ as the
Poisson commuting functions, and then $\Lambda_\Gamma$ consists of
points $((x,\xi),(y,\eta))$ such that $x=y\in Y$ (i.e.\ $x'=0=y'$,
$x''=y''$) and $\xi+\eta\in N^*_x Y$ (i.e.\ $\xi''=-\eta''$), i.e.
$$
\Lambda_\Gamma=N^*\{x'=0=y',\ x''=y''\}.
$$
Another
example, considered in \cite{Greenleaf-Uhlmann:Estimates},
is with $\tilde\Gamma$ given by $\xi'=0$, so
\begin{equation}\label{eq:Gr-Uhl-model-2}
\tilde\Lambda
=\Lambda_{\tilde\Gamma}=\{((x,\xi),(y,\eta)):\ \xi'=0=\eta',\
\xi''=-\eta'',\ x''=y''\}.
\end{equation}

For purposes of
considering elements of $I^{p,l}(\Lambda_0,\Lambda_1)$ as operators on
functions or distributions on $\RR^n$, it is important whether the
Lagrangians intersect $T^*\RR^n\times o_{\RR^n}$ or $o_{\RR^n}\times
T^*\RR^n$, with $o_{\RR^n}$ denoting the zero section of $T^*\RR^n$.
Our first example, with $\Gamma=T^*_YX$, $\Lambda_0=N^*\diag$ and
$\Lambda_1=\Lambda_\Gamma$ contains covectors of both types, namely
points like
$$
\{x'=0=y',\ x''=y'',\ \xi''=-\eta''=0,\ \xi'=0,\
\eta'\neq 0\}
$$
and
$$
\{x'=0=y',\ x''=y'',\ \xi''=-\eta''=0,\ \eta'=0,\ \xi'\neq 0\};
$$
these are in the
intersections of $N^*\{x'=0=y',\ x''=y''\}$ with $N^*\{y'=0\}$ resp.\
$N^*\{x'=0\}$. The behavior at these intersections is best considered
in terms of another Lagrangian pair, discussed below after
\eqref{eq:one-sided-mapping}, and for now we assume that the wave
front set of the elements of $I^{p,l}(\Lambda_0,\Lambda_1)$ we
consider is disjoint from $T^*\RR^n\times o_{\RR^n}$ and $o_{\RR^n}\times
T^*\RR^n$.
We write
\begin{equation}\begin{aligned}\label{eq:tilde-I-def}
&\tilde I^{*,*}(\Lambda_0,\Lambda_1)\\
&\qquad=\{K\in
I^{*,*}(\Lambda_0,\Lambda_1):\ \WF(K)
\cap (T^*\RR^n\times o_{\RR^n})=\emptyset,\\
&\qquad\qquad\qquad\qquad\qquad \WF(K)\cap (o_{\RR^n}\times
T^*\RR^n)=\emptyset\}.
\end{aligned}\end{equation}

If one reverses the order of the Lagrangians, i.e.\ $\Lambda_0$ is the
flow-out of $\Lambda_1=N^*\diag$, then for $l,l'<-k/2$, one has
$$
I^{p,l}(\Lambda_0,\Lambda_1)\subset
I^p(\Lambda_1)+I^{p+l+\ep,-l-\ep}(\Lambda_1,\Lambda_0),
$$
with a similar decomposition for
$I^{p',l'}(\Lambda_0,\Lambda_1)$. Now,  by \eqref{eq:Ant-Uhl-compose}
in the last case (and in fact all the other, simpler, statements can be reduced to this
using Lemma~\ref{lemma:Lambda_0-Lag}),
\begin{equation*}\begin{aligned}
&I^p(\Lambda_1)\circ I^{p'}(\Lambda_1)\subset I^{p+p'}(\Lambda_1),\\
&I^p(\Lambda_1)\circ I^{p'+l'+\ep,-l'-\ep}(\Lambda_1,\Lambda_0)\subset
I^{p+p'+l'+\ep,-l'-\ep}(\Lambda_1,\Lambda_0),\\
&I^{p+l+\ep,-l-\ep}(\Lambda_1,\Lambda_0)\circ I^{p'}(\Lambda_1)\subset
I^{p+p'+l+\ep,-l-\ep}(\Lambda_1,\Lambda_0)\\
&I^{p+l+\ep,-l-\ep}(\Lambda_1,\Lambda_0)\circ
I^{p'+l'+\ep,-l'-\ep}(\Lambda_1,\Lambda_0)\\
&\qquad\qquad\qquad\subset I^{p+p'+l+l'+2\ep+k/2,-l-l'-2\ep-k/2}(\Lambda_1,\Lambda_0).
\end{aligned}\end{equation*} 
Thus,
\begin{equation*}\begin{aligned}
&I^{p,l}(\Lambda_0,\Lambda_1)\circ
I^{p',l'}(\Lambda_0,\Lambda_1)\\
&\qquad\qquad\subset I^{p+p'}(\Lambda_1)
+I^{p+p'+l'+\ep,-l'-\ep}(\Lambda_1,\Lambda_0)+I^{p+p'+l+\ep,-l-\ep}(\Lambda_1,\Lambda_0),
\end{aligned}\end{equation*}
which suffices for our purposes. (Note that the order of the two
Lagrangians is reversed on the two sides!)

We now recall a result of Greenleaf and Uhlmann:

\begin{prop}\label{prop:Greenleaf-Uhlmann} (See 
\cite[Theorem~3.3]{Greenleaf-Uhlmann:Estimates})
An operator
$A\in I^{p,l}(\Lambda_1,\Lambda_0)$ (with, say, compactly supported
Schwartz kernel) is continuous $H^{m'}\to H^{m}$ if
$$
p+\frac{k}{2}\leq m'-m\ \text{and}\ p+l\leq m'-m.
$$
\end{prop}

Note that the first condition is exactly the boundedness condition for
elements of $I^{p}(\Lambda_0)$, while the second one is that of elements
of $I^{p+l}(\Lambda_1)$.

There is actually an error in the proof of
\cite[Theorem~3.3]{Greenleaf-Uhlmann:Estimates}. Recall that the
proposition is reduced to the case of $m=m'=0$ and equality holding in
one of the two inequalities. The $p+l=0$ (and then
$l\geq k/2$, so $p\leq -k/2$) case is the
problematic one in the proof; note that this means that the order on
the flow-out, $\Lambda_0$, which is regarded as the main Lagrangian, is
small compared to that on $\Lambda_1$, the conormal bundle of the
diagonal.
This is a problem since $\Id\in I^0(\Lambda_1)$
is assumed to be to be $I^{p,l}(\Lambda_1,\Lambda_0)$, but as we
remarked after this only holds for $p=-k/2$, $l=k/2$, and {\em not for
  smaller values of $p$}. However, this can be fixed: by Lemma~\ref{lemma:filter}, if
$p+l=0$, $p<-k/2$, then
$$
I^{p,l}(\Lambda_1,\Lambda_0)\subset I^{-k/2,k/2}(\Lambda_1,\Lambda_0),
$$
so one may assume that $p=-k/2$, $l=k/2$, in which case the rest of
the argument goes through.

In view of Corollary~\ref{cor:reverse-contain}, we deduce:

\begin{prop}\label{prop:diag-main-bded}
With $\Lambda_1=N^*\diag$, $\Lambda_0$ its flow out, $\tilde
I^{*,*}(\Lambda_0,\Lambda_1)$ as in \eqref{eq:tilde-I-def},
$K\in \tilde I^{p,l}(\Lambda_0,\Lambda_1)$ is bounded from $H^{m'}$ to $H^m$ if
$$
p\leq m'-m,\  p+l<m'-m-\frac{k}{2}.
$$
\end{prop}

Note that the assumptions are the criterion (except that equality is also allowed in the criterion)
for elements of $I^p(\Lambda_1)$, resp.\ $I^{p+l}(\Lambda_0)$,
to be bounded in the stated manner.

\begin{proof}
If $l<-k/2$ then Corollary~\ref{cor:reverse-contain} gives that for
all $\ep>0$,
$$
I^{p,l}(\Lambda_0,\Lambda_1)\subset I^p(\Lambda_1)+I^{p+l+\ep,-l-\ep}(\Lambda_1,\Lambda_0).
$$
Now, elements of $I^p(\Lambda_1)$ are bounded from $H^{m'}$ to $H^m$
when $p\leq m'-m$, while those of
$I^{p+l+\ep,-l-\ep}(\Lambda_1,\Lambda_0)$ are bounded from $H^{m'}$ to $H^m$ when
$$
p+l+\ep+\frac{k}{2}\leq m'-m\ \text{and}\ p\leq m'-m,
$$
taking $\ep>0$ sufficiently small (so that $\ep\leq
m'-m-\frac{k}{2}-p-l$, note that the right hand side is positive), the
proposition follows.

If $l\geq -k/2$ then for $\ep>0$, using $l+k/2+\ep>0$, and thus in
view of Lemma~\ref{lemma:filter}:
$$
I^{p,l}(\Lambda_0,\Lambda_1)\subset I^{p+l+k/2+\ep,-k/2-\ep}(\Lambda_0,\Lambda_1),
$$
so by the first part of the proof $I^{p,l}(\Lambda_0,\Lambda_1)$ is
bounded from $H^{m'}$ to $H^m$ when
$$
p+l+k/2+\ep\leq m'-m,\  p+l<m'-m-\frac{k}{2}.
$$
Taking $0<\ep<m'-m-\frac{k}{2}-(p+l)$, the inequalities are satisfied,
and the proposition follows.
\end{proof}

However, while boundedness is important for our purposes, we also need
to show that the classes $I^{p,l}(\Lambda_0,\Lambda_1)$ satisfy
a composition law. For this, as well as other, purposes, we consider another model of cleanly
intersecting Lagrangians, related to the $\Gamma=T^*_YX$ case
considered above.

This other model of a cleanly intersecting Lagrangian pair
is, in $T^*\RR^n\setminus o$, where $\RR^n=\RR^k_{x'}\times\RR^{n-k-d}_{x''}\times\RR^d_{x'''}$,
\begin{equation}\label{eq:model-Lag-2}
\Lambda_0=N^*\{x'=0,\ x''=0\},\ \Lambda_1=N^*\{x''=0\}.
\end{equation}
One may assume (via localization in the double primed dual variables, and
using that one is near the intersection $\Lambda_0\cap\Lambda_1$) that one is working
in the region where $|\xi''_q|>C\langle\xi\rangle$, and then
this pair is reduced to the standard Lagrangian pair
$(\tilde\Lambda_0,\tilde\Lambda_1)$ considered above via the homogeneous
symplectomorphism
$$
(x',x'',x''',\xi',\xi'',\xi''')\mapsto(x',x''+\frac{x'''\cdot\xi'''}{\xi''_q}e_q,-\frac{\xi'''}{\xi''_q},\xi',\xi'',\xi''_q x'''),
$$
which is quantized by the elliptic $0$th order FIO
$$
Fu(y)=\int e^{i[(y'-x')\cdot\xi'+(y''-x'')\cdot\xi''+(x'''\cdot y''')
e_q\cdot \xi'']}|\xi''|^{d/2}u(x)\,dx.
$$
The characterization of $I^{p,l}(\tilde\Lambda_0,\tilde\Lambda_1)$ as
inverse Fourier transforms modulo $I^p(\tilde\Lambda_1)$ of elements of
$S^{p,l}$ gives that they can also be described, modulo
$I^p(\Lambda_1)$, by oscillatory integrals
\begin{equation}\label{eq:osc-int-mod-2}
\int e^{i(y'\cdot\xi'+y''\cdot\xi'')}b(y''',\xi',\xi'') \,d\xi'\,d\xi'',
\end{equation}
where $b\in
S^{p-\frac{n}{4}+\frac{k}{2}+\frac{d}{2},l-\frac{k}{2}}(\RR^d_{x'''};\RR^{n-k-d}_{\xi''};\RR^k_{\xi'})$. Thus,
one has the inverse Fourier transform in the primed and double primed
variables, with the triple primed parameters serving as parameters,
i.e.\ one can add parametric variables to the above parameterization
using $N^*\{x'=0,x''=0\}$ and $N^*\{x''=0\}$ at the cost of shifting
the orders appropriately. The principal symbol of
\eqref{eq:osc-int-mod-2} on $\Lambda_1$ is then, with $\cF'$ the
inverse Fourier transform in the primed variables,
\begin{equation}\label{eq:Lambda_1-princ-symbol-mod}
(2\pi)^{\frac{(3n+2k-2d)}{4}}(\cF')^{-1}b\, |dy'|^{1/2}\,|d\xi''|^{1/2}\,|dy'''|^{1/2}
\end{equation}
in
\begin{equation}\begin{aligned}\label{eq:Lambda_1-symb-as-IFT-mod}
&S^{p-n/4+k/2+d/2}(\RR^d_{y'''};\RR^{n-k-d}_{\xi''};I^{l-\frac{k}{4}}(\RR^k_{y'};N^*\{0\}))\\
&=S^{p-n/4+k/2+d/2}(\RR^{n-k-d}_{\xi''};I^{l-\frac{k+d}{4}}(\RR^{k+d}_{y',y'''};N^*\{y'=0\}))
\end{aligned}\end{equation}
modulo
\begin{equation*}\begin{aligned}
&S^{p-n/4+k/2+d/2-1}(\RR^d_{y'''};\RR^{n-k-d}_{\xi''};I^{l+1-\frac{k}{4}}(\RR^k_{y'};N^*\{0\}))\\
&=S^{p-n/4+k/2+d/2-1}(\RR^{n-k-d}_{\xi''};I^{l+1-\frac{k+d}{4}}(\RR^{k+d}_{y',y'''};N^*\{y'=0\})).
\end{aligned}\end{equation*}

With this parameterization it is straightforward to see, as was shown
by Greenleaf and Uhlmann in
\cite[Lemma~1.1]{Greenleaf-Uhlmann:Recovering}, that if $Y$ and $Z$ are
transversal manifolds of codimension $d_1$, resp.\ $d_2$, in $\RR^n$,
then the product of distributions conormal to
$Y$ and $Z$, respectively, is a sum of paired Lagrangian distributions
associated to the pairs $(N^*(Y\cap Z),N^*Y)$ and $(N^*(Y\cap
Z),N^*Z)$. More precisely,
$$
I^{[\mu]}(Y) I^{[\mu']}(Z)\subset I^{[\mu,\mu']}(Y\cap
Z,Y)+I^{[\mu',\mu]}(Y\cap Z,Z),
$$
where
\begin{equation}\label{eq:pair-conorm-subset}
I^{[\mu,\mu']}(Y\cap
Z,Y)=I^{\mu+\frac{d_1}{2}-\frac{n}{4},\mu'+\frac{d_2}{2}}(N^*(Y\cap Z),N^*Y).
\end{equation}
(Here the left hand side is denoted by $I^{\mu,\mu'}(Y,Y\cap
Z)$, $Y\cap Z=Y_2\subset Y_1=Y$ in \cite{Greenleaf-Uhlmann:Recovering}
just after Equation~(1.4). Then the equality in
\eqref{eq:pair-conorm-subset} is the extreme left hand side of the first
displayed equation after Equation~(1.4) being equal to the extreme
right hand side. The middle expression in this equation is {\em not}
equal to the extreme right hand side.)

Note that here the codimension of the intersection of the two
Lagrangians $N^*Y$ and $N^*(Y\cap Z)$ is $d_2$, and thus using
$$
I^{[\mu]}(Y)=I^{\mu+\frac{d_1}{2}-\frac{n}{4}}(N^*Y),\ 
I^{[\mu']}(Z)=I^{\mu'+\frac{d_2}{2}-\frac{n}{4}}(N^*Z),
$$
one has
\begin{equation}\label{eq:product-conormal}
I^{\mu}(N^*Y) I^{\mu'}(N^*Z)\subset I^{\mu,\mu'+\frac{n}{4}}(N^*(Y\cap Z),N^*Y)+I^{\mu',\mu+\frac{n}{4}}(N^*(Y\cap Z),N^*Z),
\end{equation}
We remark here that one must be careful in ordering the Lagrangians,
as mentioned above; this {\em is} the correct ordering. Thus, the
`main' Lagrangians are the original ones, $N^*Y$ and $N^*Z$;
$N^*(Y\cap Z)$ carries a relative singularity only.

A special case of the model of \eqref{eq:model-Lag-2}
in $\RR^{2n}=\RR^n_x\times\RR^n_y$ (note the change of
dimension!), with $\RR^n=\RR^k_{x'}\times\RR^{n-k}_{x''}$ is, with
$(\xi,\eta)$ the dual variables of $(x,y)$,
\begin{equation}\begin{aligned}\label{eq:near-diag-model}
&\Lambda_1=\{x'=y',\ x''=y'',\ \xi'=-\eta',\ \xi''=-\eta''\}=N^*\diag,\\
&\Lambda_0=\{x'=0=y',\ x''=y'',\ \xi''=-\eta''\}=N^*\{x'=0=y',\ x''=y''\},
\end{aligned}\end{equation}
with codimension $k$ intersection; this corresponds to the flowout
with $\Gamma=T^*_Y X$, $Y=\{x'=0\}$, discussed above.
Then the parameterization of $I^{p,l}(\Lambda_0,\Lambda_1)$, modulo
$I^p(\Lambda_1)$, is
$$
\int e^{i[(x'-y')\cdot\xi'+(x''-y'')\cdot\xi''+x'\cdot\eta']} a(x'',
\xi',\xi'',\eta')\,d\xi\,d\eta',\ a\in S^{p,l-\frac{k}{2}}(\RR^{n-k}_{x''};\RR^{n}_\xi;\RR^k_{\eta'}),
$$
with a conic neighborhood of $\eta'=0$ in $(\RR^n_{\xi}\times
\RR^k_{\eta'})\setminus 0$ corresponding to a neighborhood of the
intersection $\Lambda_0\cap\Lambda_1$ (so $\xi$ is the `large'
variable on the parameter space,
note that it is indeed the variable in the parameterization of the
conormal bundle of the diagonal),
and the $x''$ dependence can be replaced by $y''$ dependence. (To see
this form of parameterization, write $z''=x-y$, $z'=x'$, $z'''=x''$
then $\Lambda_1=N^*\{z''=0\}$, $\Lambda_0=N^*\{z'=0,\
z''=0\}$. Replacing $x'$ and $x''$ by $y'$ and $y''$ in the definition
of $z'$ and $z'''$ gives the other parameterization.)
Here the principal symbol is, with $\cF'$ the Fourier transform
in the last variable, $\eta'$,
\begin{equation}\label{eq:Lambda_1-princ-symbol-op}
(2\pi)^{n+k}(\cF')^{-1}a\, |dy'|^{1/2}\,|d\xi|^{1/2}\,|dy'''|^{1/2}
\end{equation}
in
\begin{equation}\label{eq:Lambda_1-symb-as-IFT-op}
S^{p}(\RR^{n-k}_{x''};\RR^{n}_{\xi};I^{l-\frac{k}{4}}(\RR^k_{x'};N^*\{0\}))
=S^{p}(\RR^{n}_{\xi''};I^{l-\frac{n}{4}}(\RR^{n}_{x};N^*\{x'=0\}))
\end{equation}
modulo
$$
S^{p-1}(\RR^{n-k}_{x''};\RR^{n}_{\xi};I^{l+1-\frac{k}{4}}(\RR^k_{x'};N^*\{0\}))
=S^{p-1}(\RR^{n}_{\xi''};I^{l+1-\frac{n}{4}}(\RR^{n}_{x};N^*\{x'=0\})).
$$

Writing out the composition we have:

\begin{prop}\label{prop:diag-main-compose}
With $\Lambda_1=N^*\diag$, $\Lambda_0$ its flow out, the subset
$\tilde I^{*,*}(\Lambda_0,\Lambda_1)$ of
$I^{*,*}(\Lambda_0,\Lambda_1)$ defined in \eqref{eq:tilde-I-def},
satisfies that if $l+l'<0$ and $L=\max(l,l',l+l'+k/2)$, then
\begin{equation}\begin{aligned}\label{eq:composition-prop}
&\tilde I^{p,l}(\Lambda_0,\Lambda_1)\circ
\tilde I^{p',l'}(\Lambda_0,\Lambda_1)\subset \tilde
I^{p+p',L}(\Lambda_0,\Lambda_1).
\end{aligned}\end{equation}
Furthermore,  with $-(l+l')>\delta>0$, modulo
\begin{equation*}\begin{aligned}
&S^{p+p'-\min(1,\delta)}(\RR^{n-k}_{x''};\RR^n_{\xi};I^{L+\delta-\frac{k}{4}}(\RR^k_{x'};N^*\{0\}))\\
&=S^{p+p'-\min(1,\delta)}(\RR^{n}_{\xi''};I^{L+\delta-\frac{n}{4}}(\RR^{n}_{x};N^*\{x'=0\}))
\end{aligned}\end{equation*}
the principal symbol on $\Lambda_1=N^*\diag$ in
$$
S^{p+p'}(\RR^{n-k}_{x''};\RR^n_{\xi};I^{L-\frac{k}{4}}(\RR^k_{x'};N^*\{0\}))
=S^{p+p'}(\RR^{n}_{\xi''};I^{L-\frac{n}{4}}(\RR^{n}_{x};N^*\{x'=0\}))
$$
of the composition of two operators
is the product of their principal symbols.
\end{prop}

\begin{rem}\label{rem:big-off-diag}
As one can always decrease the second order $l$ at the cost of
increasing the first order $p$, see Lemma~\ref{lemma:filter}, this
result also gives that if $l+l'\geq 0$ then for any $\ell>l+l'$,
with $L=\max(l-\ell,l',l-\ell+l'+k/2)$,
\begin{equation}\begin{aligned}\label{eq:composition-prop-mod}
&\tilde I^{p,l}(\Lambda_0,\Lambda_1)\circ
\tilde I^{p',l'}(\Lambda_0,\Lambda_1)\subset \tilde
I^{p+p'+\ell,L}(\Lambda_0,\Lambda_1).
\end{aligned}\end{equation}
However, the increase of the order on $\Lambda_1$ relative to the
$l+l'<0$ case makes this a much
less useful result.
\end{rem}

\begin{rem}\label{rem:symbol-constraint}
The constraint $l+l'<0$ is exactly the constraint under which elements
of
$S^{p}(\RR^{n}_{\xi''};I^{l-\frac{n}{4}}(\RR^{n}_{x};N^*\{x'=0\}))$
and
$S^{p'}(\RR^{n}_{\xi''};I^{l'-\frac{n}{4}}(\RR^{n}_{x};N^*\{x'=0\}))$
can be multiplied in view of the lack of smoothness of these symbols
in $x'$. Namely, the issue is multiplication for elements of
$I^{l-\frac{n}{4}}(\RR^{n}_{x};N^*\{x'=0\})$ and
$I^{l'-\frac{n}{4}}(\RR^{n}_{x};N^*\{x'=0\})$ which come from partial inverse
Fourier transforms in $x'$ of symbols of order $l-k/2$, resp.\
$l'-k/2$. Typical members of these classes are asymptotically
homogeneous of degree $l-k/2$, resp.\ $l'-k/2$, so their
partial inverse
Fourier transforms in $x'$ are, modulo smooth functions, homogeneous
of degree $-k/2-l$, resp.\ $-k/2-l'$. The restriction $l+l'<0$ means
that the total homogeneity is $>-k$, i.e.\ is strictly greater than
that of a delta distribution on $x'=0$. Marginally disallowed products
are thus, in the case $k=1$, a delta distribution and a step function
at a hypersurface; any more smoothness than that of the step function
(in terms of conormal order) means that the functions is continuous
and may be multiplied by the $\delta$ distribution. Thus, in this
sense, this proposition is {\em sharp}.
\end{rem}

\begin{proof}
Let
$$
A\in\tilde I^{p,l}(\Lambda_0,\Lambda_1),
\ B\in \tilde I^{p',l'}(\Lambda_0,\Lambda_1);
$$
we may assume that $A$ and $B$ both have wave front set near the
intersection of the two Lagrangians.
Write $A$ resp.\ $B$ as an oscillatory integral with the amplitude
independent of the right, resp.\ left, base variable, i.e.
\begin{equation*}\begin{aligned}
&(Av)(x)=\int e^{i[(x'-y')\cdot\xi'+(x''-y'')\cdot\xi''+x'\cdot\eta']} a(x'',
\xi',\xi'',\eta')\,d\xi\,d\eta'\,v(y)\,dy,\\
&\qquad a\in S^{p,l-\frac{k}{2}}(\RR^{n-k}_{x''};\RR^{n}_\xi;\RR^k_{\eta'}),
\end{aligned}\end{equation*}
resp.\ 
\begin{equation*}\begin{aligned}
&(Bu)(y)=\int e^{i[(y'-z')\cdot\zeta'+(y''-z'')\cdot\zeta''+z'\cdot\mu']} b(z'',
\zeta',\zeta'',\mu')\,d\zeta\,d\mu'\,u(z)\,dz,\\
&\qquad b\in S^{p',l'-\frac{k}{2}}(\RR^{n-k}_{z''};\RR^{n}_\zeta;\RR^k_{\mu'}),
\end{aligned}\end{equation*}
with
$$
|\xi|\geq 1, |\eta'|\leq \ep|\xi|\ \text{on}\ \supp a,\ \text{and}
\ |\zeta|\geq 1, |\mu'|\leq \ep |\zeta|\ \text{on}\ \supp b,
$$
for $\ep<1/2$. Note that the wave front set of the Schwartz kernel of
$A$ (over $x'=y'=0$, $x''=y''$) is contained in the set of covectors
of the form
$(\xi'+\eta',\xi'',-\xi',-\xi'')$ such that $a$ is not Schwartz in the
direction $(\xi',\xi'',\eta')$, i.e.
$(\xi',\xi'',\eta')$ is not in the microsupport of $a$. Since we do not want covectors of the
kind $o\times T^*\RR^n$ in the wave front set, we need $(\xi'+\eta',\xi'')$
bounded away from $0$ on the microsupport of $a$ when
$(\xi',\xi'')\neq 0$, which is
accomplished by our requirement that $\ep<1/2$.

Thus, with $\cF$ denoting the Fourier transform on $\RR^n$,
\begin{equation*}\begin{aligned}
(Av)(x)=&\int e^{i[x'\cdot\xi'+x''\cdot\xi''+x'\cdot\eta']} a(x'',
\xi',\xi'',\eta')(\cF v)(\xi)\,d\xi\,d\eta',
\end{aligned}\end{equation*}
while $Bu$ is the inverse Fourier transform in $\zeta$ of
$$
\int e^{i[-z'\cdot\zeta'-z''\cdot\zeta''+z'\cdot\mu']}
(2\pi)^n b(z'',
\zeta',\zeta'',\mu')\,d\mu'\,u(z)\,dz.
$$
Therefore,
\begin{equation*}\begin{aligned}
(ABu)(x)=\int
&e^{i[(x'-z')\cdot\xi'+(x''-z'')\cdot\xi''+x'\cdot\eta'+z'\cdot\mu']}\\
& (2\pi)^n a(x'',
\xi',\xi'',\eta') b(z'',
\xi',\xi'',\mu') \,d\xi\,d\eta'\,d\mu' \,u(z)\,dz,
\end{aligned}\end{equation*}
i.e.\ the Schwartz kernel of $AB$ is given by the oscillatory integral
\begin{equation*}\begin{aligned}
\int
&e^{i[(x'-z')\cdot\xi'+(x''-z'')\cdot\xi''+x'\cdot\eta'+z'\cdot\mu']}\\
& (2\pi)^n a(x'',
\xi',\xi'',\eta') b(z'',
\xi',\xi'',\mu') \,d\xi\,d\eta'\,d\mu'.
\end{aligned}\end{equation*}
We rewrite the phase as
$$
(x'-z')\cdot(\xi'-\mu')+(x''-z'')\cdot\xi''+x'\cdot(\eta'+\mu').
$$
Letting $\nu'=\eta'+\mu'$, $\zeta'=\xi'-\mu'$, we deduce that the Schwartz
kernel of $AB$ is
\begin{equation*}\begin{aligned}
\int
&e^{i[(x'-z')\cdot\zeta'+(x''-z'')\cdot\xi''+x'\cdot\nu']} c(x'',z'',\zeta',\xi'',\nu') \,d\zeta'\,d\xi''\,d\nu',\\
& c(x'',z'',\zeta',\xi'',\nu')=(2\pi)^n \int a(x'',
\zeta'+\mu',\xi'',\nu'-\mu') b(z'',
\zeta'+\mu',\xi'',\mu')\,d\mu'.
\end{aligned}\end{equation*}
Thus, to show \eqref{eq:composition-prop}, we merely need to show that
\begin{equation}\label{eq:c-symbol-bd}
c\in S^{p+p',L-k/2}(\RR^{n-k}_{x''}\times \RR^{n-k}_{z''};\RR^{n}_\xi;\RR^k_{\nu'}),
\end{equation}
and then the composition result follows. Note that in view of the
support conditions on $a$ and $b$, on the support of the integrand of $c$,
$|\nu'-\mu'|,|\mu'|\leq\ep|\zeta+\mu'|$ (here $\zeta\in\RR^n$), thus $|\mu'|\leq
\frac{\ep}{1-\ep}|\zeta|$, $|\nu'|\leq 2 \frac{\ep}{1-\ep}|\zeta|$,
and thus the integral is certainly convergent, {\em without
  restrictions on $l,l'$}, with $c$ supported in
$|\nu'|\leq 2|\zeta|$, and moreover $|\zeta+\mu'|$ is bounded from above
and below by positive multiples of $|\zeta|$. For
$l+l'<0$, one gets, for an absolute constant $C>0$, and with
$\|a\|_{S^{p,l-\frac{k}{2}},0}$, etc., denoting $0$th symbol norms
($\sup$ norms),
\begin{equation}\label{eq:twisted-conv-bd}
|c|\leq C\|a\|_{S^{p,l-\frac{k}{2}},0}\|b\|_{S^{p',l'-\frac{k}{2}},0}\langle\zeta\rangle^{p+p'}
\int_{\RR^k} \langle\nu'-\mu'\rangle^{l-k/2}\langle\mu'\rangle^{l'-k/2}\,d\mu';
\end{equation}
here for $\nu'$ in a compact set, one gets uniform bounds for the
integral as the integrand is then bounded by $\tilde
C\langle\mu'\rangle^{l+l'-k}$; $l+l'<0$ is used here strongly. (If one
does not assume $l+l'<0$, one needs to use that $|\mu'|\lesssim
|\zeta|$ on the support of the integrand, so $\RR^k$ can be replaced
by the ball $B_{|\zeta|}(0)$, and one obtains a positive
power of $|\zeta|$ as a result when integrating, which allows one to
obtain a paired Lagrangian symbolic estimate but with the rather
undesirable increase of the order $p+p'$ on $\Lambda_1$. See also
Remark~\ref{rem:big-off-diag}.)
Further, for $l+l'<0$, the integral on the right hand side can be
estimated, uniformly as $|\nu'|\to\infty$, by
\begin{equation}\label{eq:basic-conv-bd}
C'(\langle\nu'\rangle^{l+l'}+\langle\nu'\rangle^{l-k/2}+\langle\nu'\rangle^{l'-k/2})\leq
C''\langle\nu'\rangle^{L-k/2}.
\end{equation}
Indeed, for $|\nu'|\leq 1$, say, we already explained this estimate.
Otherwise we
break up the region of integration into $|\mu'|\leq |\nu'|/2$, resp.\
$|\nu'-\mu'|\leq |\nu'|/2$, resp.\ $|\nu'|/2\leq
|\mu'|,|\nu'-\mu'|\leq 2|\nu'|$, resp.\ $2|\nu'|\leq
|\mu'|$, resp.\ $2|\nu'|\leq|\nu'-\mu'|$. Note that the last two regions are not disjoint,
but the union of the five regions is $\RR^k$.
On the first, resp.\ second of these,
$\langle\nu'-\mu'\rangle$, resp.\ $\langle\mu'\rangle$ is bounded from
above and below by a positive multiple of $\langle\nu'\rangle$, so the
corresponding weight can be pulled outside the integral, so in the
first case one is
reduced to the estimate
$$
\int_{B_{|\nu'|/2}(0)}\langle\mu'\rangle^{l-k/2}\,d\mu'\lesssim
  (1+|\nu'|^{l+k/2}),
$$
resulting in an overall bound $|\nu'|^{l'-k/2}(1+|\nu'|^{l+k/2})$,
yielding that \eqref{eq:basic-conv-bd} is satisfied in this case,
with a similar estimate in the second case. In the third case, both
$\langle\nu'-\mu'\rangle$
and  $\langle\mu'\rangle$ is bounded from
above and below by a positive multiple of $\langle\nu'\rangle$, and
one obtains a bound $\lesssim |\nu'|^{l+l'}$. In the fourth, resp.\ fifth case,
$\langle\nu'-\mu'\rangle$, resp.\ $\langle\mu'\rangle$ is bounded from
above and below by a positive multiple of $\langle\mu'\rangle$,
resp.\ $\langle\nu'-\mu'\rangle$, so in the fourth case one is reduced to the estimate
$$
\int_{|\mu'|\geq 2|\nu'|}\langle\mu'\rangle^{l+'l-k}\,d\mu'\lesssim
  \langle \nu'\rangle^{l+l'},
$$
with a similar bound in the fifth case; these use $l+l'<0$.
This proves \eqref{eq:basic-conv-bd}, and thus gives the $0$th seminorm estimate of the claimed
$S^{p+p',L}(\RR^{n-k}_{x''}\times
\RR^{n-k}_{z''};\RR^{n}_\xi;\RR^k_{\nu'})$ statement, \eqref{eq:c-symbol-bd}, for $c$.

The
derivatives can be handled easily, with this being immediate for
$\zeta$, $x''$ and $z''$ derivatives, while for $\nu'_j\pa_{\nu'_k}$
derivatives one writes
$\nu'_j\pa_{\nu'_k}=(\nu'_j-\mu'_j)\pa_{\nu'_k}+\mu'_j\pa_{\nu'_k}$
under the integral, then the first term is handled by the symbol
bounds for $a$, while for the second one rewrites $\mu'_j\pa_{\nu'_k}a$ as
$-\mu'_j\pa_{\mu'_k}a+\mu'_j\pa_{\zeta'_k} a$, integrates by parts for
the first term to use the symbol estimates of $b$, while the symbol
estimates for $a$ plus the bounds for $\mu'$ in terms of $\zeta+\mu'$
handle the second term. Proceeding inductively, one deduces that
\eqref{eq:c-symbol-bd} holds.

To prove the principal symbol property, take
$N\geq 1$ integer. (Here $N=1$ suffices; taking $N$ larger one can
obtain further terms in the $\Lambda_1$-symbolic expansion of the composition.) We expand $a,b$ in Taylor series in their second
argument, $\zeta'+\mu'$, around $\zeta'$ with the integral remainder
formula involving
$N$th derivatives. In case of $a$, this gives terms
$$
\frac{1}{\alpha!}(\mu')^\alpha(\pa_{\zeta'}^\alpha
a)(x'',
\zeta',\xi'',\nu'-\mu')
$$
with $|\alpha|<N$ in the expansion, and the remainder is a sum of
integrals with $|\alpha|=N$:
$$
\int_0^1 \frac{N}{\alpha!}(1-t)^N(\mu')^\alpha(\pa_{\zeta'}^\alpha
a)(x'',
\zeta'+t\mu',\xi'',\nu'-\mu')\,dt;
$$
similar expressions hold for $b$, with
$(\mu')^\beta\pa_{\zeta'}^\beta$ being the relevant derivatives. The
$(\alpha\beta)$th term (with $|\alpha|\leq N$, $|\beta|\leq N$)
in $c$ inside the integral has bounds
$$
\lesssim\langle\zeta'+\mu'\rangle^{p+p'-|\alpha|-|\beta|}\langle\nu'-\mu'\rangle^{l-k/2}
\langle\mu'\rangle^{l'+|\alpha|+|\beta|-k/2},
$$
and thus if $l+l'+|\alpha|+|\beta|<0$, the contribution to $c$ is in
\begin{equation*}\begin{aligned}
&S^{p+p'-|\alpha|-|\beta|,L_{\alpha\beta}},\ \text{with}\\
&L_{\alpha\beta}=\max(l,l'+|\alpha|+|\beta|,l+l'+|\alpha|+|\beta|+k/2)\leq
L+|\alpha|+|\beta|.
\end{aligned}\end{equation*} 
If $l+l'+|\alpha|+|\beta|\geq 0$, then, letting
$$
M=-\delta+|\alpha|+|\beta|>l+l'+|\alpha|+|\beta|\geq 0,
$$
so $M<|\alpha|+|\beta|$, and using
$$
\langle\mu'\rangle^{l'+|\alpha|+|\beta|-k/2}
\lesssim
\langle\zeta\rangle^M \langle\mu'\rangle^{l'+|\alpha|+|\beta|-k/2-M}
$$
(by the support conditions), we obtain that the contribution of the
$(\alpha\beta)$th term to
$c$ is in
\begin{equation*}\begin{aligned}
&S^{p+p'-|\alpha|-|\beta|+M,\tilde L_{\alpha\beta}},\ \text{with}\\
&\tilde
L_{\alpha\beta}=\max(l,l'+|\alpha|+|\beta|-M,l+l'+|\alpha|+|\beta|+k/2-M)\\
&\qquad\qquad\leq
L+|\alpha|+|\beta|-M.
\end{aligned}\end{equation*}
This gives that modulo
$S^{p+p'-\min(1,\delta),L+\min(1,\delta)}$,
$c$ is given by the convolution
\begin{equation*}
(2\pi)^n \int a(x'',
\zeta',\xi'',\nu'-\mu') b(z'',
\zeta',\xi'',\mu')\,d\mu'.
\end{equation*}
Taylor expanding $b$ in $z''$ around $x''$ and integrating by parts in
$\xi''$ gives that further this can be replaced by
\begin{equation*}
\tilde c(x'',\zeta',\xi'',\nu')=(2\pi)^n \int a(x'',
\zeta',\xi'',\nu'-\mu') b(x'',
\zeta',\xi'',\mu')\,d\mu'
\end{equation*}
modulo
$S^{p+p'-1,L+1}$.
The $\Lambda_1$-principal symbol of the distribution corresponding to
$\tilde c$ is $(2\pi)^{n-k}$ times
the partial inverse Fourier transform in $\nu'$ of $\tilde c$. Since
the inverse Fourier transform of a convolution in $\RR^k$ is
$(2\pi)^k$ times the product of the inverse Fourier transforms of the
factors, we deduce that this principal symbol is
\begin{equation*}\begin{aligned}
&(2\pi)^{n+k}((\cF')^{-1}\tilde c)
\,|d\zeta'|^{1/2}\,|d\xi''|^{1/2}\\
&\qquad=(2\pi)^{n+k} ((\cF')^{-1}a)
(2\pi)^{n+k} ((\cF')^{-1}b) \,|d\zeta'|^{1/2}\,|d\xi''|^{1/2},
\end{aligned}\end{equation*}
i.e.\ it is the product of the principal symbols of $a$ and $b$, as claimed.
\end{proof}

\begin{rem}\label{rem:conormal-product}
Note that the proof we just gave also shows that if
$$
\tilde a\in
S^{p}(\RR^{n}_{\xi};I^{l-\frac{n}{4}}(\RR^{n}_{x};N^*\{x'=0\})),\ 
\tilde b\in
S^{p'}(\RR^{n}_{\xi};I^{l'-\frac{n}{4}}(\RR^{n}_{x};N^*\{x'=0\})),
$$
with $l+l'<0$, then with $L=\max(l,l',l+l'+k/2)$,
$$
\tilde a\tilde b\in S^{p+p'}(\RR^{n}_{\xi};I^{L-\frac{n}{4}}(\RR^{n}_{x};N^*\{x'=0\})).
$$
This does not require a conic support condition on the partial ($x'$-)Fourier transforms
$a$, resp.\ $b$, of $\tilde a$, resp.\ $\tilde b$ like one we did above; one is
estimating a partial convolution $c$ of $a$ and $b$ in the dual variable $\mu'$
of $x'$, and the estimates boil down to \eqref{eq:basic-conv-bd} being
satisfied for the integral on the right hand side of
\eqref{eq:twisted-conv-bd}. Further, this shows that
\begin{equation}\label{eq:tilde-a-tilde-b-0th}
\|\tilde a\tilde
b\|_{S^{p+p'}(I^{L-\frac{n}{4}});0}
\leq C\|\tilde a\|_{S^p(I^{l-\frac{n}{4}});0}\|\tilde b\|_{S^{p'}(I^{l'-\frac{n}{4}});0},
\end{equation}
where we used a short hand notation for the symbol spaces discussed
above to simplify the notation. Now, the higher order product-type
symbol norms for the partial Fourier transform, of the
partial convolution $c$ are
equivalent to a product of $\pa_{\xi_j}$, $\xi_k\pa_{\xi_j}$,
$\pa_{x_j}$, $\mu'_j\pa_{\mu'_k}$, $\pa_{\mu'_k}$ being applied
iteratively to $c$ and the zeroth $S^{p+p',L-\frac{k}{2}}$ norm being
evaluated. As $c$ is the partial Fourier transform of $\tilde a\tilde
b$ in $x'$, this means $\pa_{\xi_j}$, $\xi_k\pa_{\xi_j}$,
$\pa_{x_j}$, $\pa_{x'_j}x'_k$, $x'_k$ being applied
iteratively to $\tilde a\tilde b$, and the zeroth $S^{p+p',L-\frac{k}{2}}$ 
norm of the partial Fourier transform of the result being evaluated. (Here $x'_k$
can be dropped if one assumes compact support for $\tilde a$ or
$\tilde b$; one can also replace $\pa_{x'_j}x'_k$ by
$x'_k\pa_{x'_j}$.)
Using Leibniz' rule, which is valid by the density of
order $-\infty$ symbols in $\mu'$, resp.\ order $-\infty$ conormal
distributions in $x'$, and using
\eqref{eq:tilde-a-tilde-b-0th}, the seminorms of $\tilde a\tilde
b$ in $S^{p+p'}(I^{L-\frac{n}{4}})$ are bounded by
\begin{equation}\label{eq:tilde-a-tilde-b-kth}
\|\tilde a\tilde
b\|_{S^{p+p'}(I^{L-\frac{n}{4}});k}
\leq C_k\|\tilde a\|_{S^p(I^{l-\frac{n}{4}});k}\|\tilde b\|_{S^{p'}(I^{l'-\frac{n}{4}});k}.
\end{equation}
\end{rem}

We record here a statement regarding square roots of conormal
distributions that will be useful later; it allows us to construct
square root of the principal symbols of paired Lagrangian
distributions.

\begin{lemma}\label{lemma:sqrt}
Suppose that
$a\in
S^{p}(\RR^{n}_{\xi};I^{l-\frac{n}{4}}(\RR^{n}_{x};N^*\{x'=0\}))$ with
$l<-k/2$,
and with $a\geq c|\xi|^p$, $c>0$, for $|\xi|\geq R$, on a conic open set
$\Gamma\subset\RR^n_{\xi}$. Let $l'\in (l,-k/2)$. Then
$$
b=\sqrt{a}\in S^{p/2}(\Gamma_{\xi};I^{l'-\frac{n}{4}}(\RR^{n}_{x};N^*\{x'=0\})).
$$
\end{lemma}

Note that under the assumptions, $a$ is the inverse Fourier
transform in $\mu'$, the dual variable of $x'$, of a symbol of order
$l-k/2<-k$, so $a$ is actually continuous, and indeed H\"older
$\alpha$ for $0<\alpha<-(l+k/2)$. Thus, the pointwise statement $a\geq
c|\xi|^p$ actually makes sense.

\begin{proof}
Note that the statement is a consequence of the positive ellipticity
of $a$ away from $x'=0$, so we may work in an arbitrarily small
neighborhood of $x'=0$ as is convenient.
Given $\ep>0$, we first decompose $a=a_1+a_2$ with
\begin{equation*}\begin{aligned}
&a_1\in S^p(\RR^n_{\xi};\CI(\RR^n_x))=S^p(\RR^n_x,\RR^n_\xi),\\
&a_2\in
S^{p}(\RR^{n}_{\xi};I^{l'-\frac{n}{4}}(\RR^{n}_{x};N^*\{x'=0\})),
\end{aligned}\end{equation*}
and with $a_1\geq (c/2)|\xi|^p$, for $|\xi|\geq R$, on
$\Gamma\subset\RR^n_{\xi}$, while
$\|a_2\|_{S^{p}(\RR^{n}_{\xi};I^{l'-\frac{n}{4}});0}<\ep$ (here
we use shorthand notation as in the above remark). To do so, we
note that
$$
a=a_0+\cF^{-1}_{\mu'} b,\ b\in
S^{p,l-k/2}(\RR^{n-k}_{x''};\RR^n_\xi;\RR^k_{\mu'}),\ a_0\in
S^p(\RR^n_\xi;\CI(\RR^n_x)).
$$
Now, given $\ep'>0$, the standard approximation argument, using $b_R=b\phi(\mu'/R)$,
where $\phi\equiv 1$ near $0$, has compact support, letting
$R\to\infty$ gives $b'_1\in S^{p,-\infty}$ such that
$\|b-b'_1\|_{S^{p,l'-k/2};0}<\ep'$. Then, as $l'-k/2<-k$, with $b_2=b-b'_1$,
$$
\sup|\langle\xi\rangle^{-p}\cF^{-1}_{\mu'} b_2|\leq
C_0\|b_2\|_{S^{p,l'-k/2};0}<C_0\ep'.
$$
Thus, with
$a_2=\cF^{-1}_{\mu'}b_2$, $a_1=a-a_2=a_0+\cF^{-1}_{\mu'}b'_1$, $a>(c-C_0\ep')|\xi|^p$.
Now let $\ep'=\min(\ep,c/(2C_0)$; then $a_1$ and $a_2$ satisfy all conditions.

We note that as $a_1$ is elliptic on $\Gamma$, with a positive
elliptic lower bound,
\begin{equation*}\begin{aligned}
&\sqrt{a_1}\in
S^{p/2}(\Gamma_{\xi};\CI(\RR^n)),\\
&\tilde a=a_1^{-1}a_2\in S^0
(\Gamma_{\xi};I^{l'-\frac{n}{4}}(\RR^{n}_{x};N^*\{x'=0\})),
\end{aligned}\end{equation*}
and $\tilde a$ vanishes at $x'=0$.
We write
$$
b=\sqrt{a_1}\sqrt{1+(a_1^{-1}a_2)},
$$
and we are reduced to showing that
\begin{equation}\label{eq:square-root-conormal}
\sqrt{1+\tilde a}\in S^0
(\Gamma_{\xi};I^{l'-\frac{n}{4}}(\RR^{n}_{x};N^*\{x'=0\})).
\end{equation}
We
expand $f=\sqrt{1+.}$ in Taylor series, whose radius of convergece
$1$. By Remark~\ref{rem:conormal-product},
$$
\tilde a^N\in S^0
(\Gamma_{\xi};I^{l'-\frac{n}{4}}(\RR^{n}_{x};N^*\{x'=0\})),
$$
with
$$
\|\tilde a^N\|_{S^0(I^{l'-\frac{n}{4}});0}\leq C^{N-1}\|\tilde a\|^N_{S^0(I^{l'-\frac{n}{4}});0},
$$
which follows from \eqref{eq:tilde-a-tilde-b-0th} by induction.
This shows that, provided
$\|\tilde a\|_{S^0(I^{l'-\frac{n}{4}});0}<C^{-1}$ (which holds if
$\ep<C^{-1}$),
the Taylor series
converges in the $0$th $S^0(I^{l'-\frac{n}{4}})$-norm. Then
differentiating the Taylor series with respect to operators giving
rise to the symbol topology, as discussed in
Remark~\ref{rem:conormal-product}, preserves the
$S^0(I^{l'-\frac{n}{4}})$-estimates in view of the chain rule for
derivatives, which gives $(f'\circ \tilde a)(V\tilde a)$,
where $V$ is one of $\pa_{\xi_j}$, $\xi_k\pa_{\xi_j}$,
$\pa_{x_j}$, $x'_k\pa_{x'_j}$, $x'_k$, and the fact that $Va$ satisfies
$S^0(I^{l-\frac{n}{4}})$-estimates as well, plus the fact that $f'$ also has
Taylor series with radius of convergence
$1$. Iterating this argument proves \eqref{eq:square-root-conormal},
and thus the lemma.
\end{proof}

As we already mentioned, a different model for the Lagrangians in $\RR^{2n}$, used by
Greenleaf and Uhlmann \cite{Greenleaf-Uhlmann:Estimates},
is the pair $(N^*\diag,\Lambda_{\tilde\Gamma})$ when 
$\tilde\Gamma$ given by $\xi'=0$, so
$$
\tilde\Lambda
=\Lambda_{\tilde\Gamma}=\{((x,\xi),(y,\eta)):\ \xi'=0=\eta',\
\xi''=-\eta'',\ x''=y''\}.
$$
With this
model, paired Lagrangian distributions in $I^{p,l}(N^*\diag,\tilde\Lambda)$
are given by oscillatory integrals
$$
\int e^{i[(x'-y'-s)\cdot\zeta'+(x''-y'')\cdot\zeta+s\sigma]}a(x,y,s,\zeta,\sigma)\,ds\,d\zeta\,d\sigma,
$$
with $a\in S^{M,M'}(\RR^{2n+k},\RR^n,\RR^k)$, $M=p+k/2$, $M'=l-k/2$ (there is a typo in
\cite{Greenleaf-Uhlmann:Estimates} in their definition of the first
order after (1.31)). Note here the flow-out is the second Lagrangian,
reversed as compared to Proposition~\ref{prop:diag-main-compose},
which is convenient to apply the results of Antoniano and Uhlmann, but
is not convenient in our case.

Since the structure of the projection maps of the left and the right
factors matters for composition purposes (i.e.\ just because all
Lagrangian pairs can be put to a model form via a symplectomorphism on
$\RR^{2n}$, it does not follow that they all have the same composition
properties!), we also need another
special case of the model of \eqref{eq:model-Lag-2}
in $\RR^{2n}=\RR^n_x\times\RR^n_y$, with
$\RR^n=\RR^k_{x'}\times\RR^{n-k}_{x''}$ and with
$(\xi,\eta)$ the dual variables of $(x,y)$, as before. This is
\begin{equation}\begin{aligned}\label{eq:one-sided-mapping}
&\Lambda_1=\{x'=0,\ \xi''=0,\ \eta'=0,\ \eta''=0\}=N^*\{x'=0\},\\
&\Lambda_0=\{x'=0=y',\ x''=y'',\ \xi''=-\eta''\}=N^*\{x'=0=y',\ x''=y''\},
\end{aligned}\end{equation}
this time with codimension $n$ intersection. Note that here
$\Lambda_0$ is the same `flow-out' Lagrangian as in
\eqref{eq:near-diag-model}, but $\Lambda_1$ a Lagrangian of the form
$\Lambda_1^\sharp\times o_{\RR^n}$, with $\Lambda_1^\sharp$ Lagrangian
in $T^*\RR^n\setminus o_{\RR^n}$, which means that if an operator
with Schwartz kernel in $I^p(\Lambda_1)$ is applied to even a
$\CI_c(\RR^n_y)$ function, the result is not $\CI$, merely Lagrangian
on $\Lambda_1^\sharp$. (There is a dual phenomenon if one reverses the
$x$ and the $y$ factors, namely then the operator cannot be applied to
all distributions.)
For this pair, the parameterization, modulo
$I^p(\Lambda_1)$, is
\begin{equation}\label{eq:one-sided-param}
\int e^{i[x'\cdot\xi'-y'\cdot\eta'+(x''-y'')\cdot\eta'']} a(x'',
\xi',\eta',\eta'')\,d\xi'\,d\eta,\ a\in S^{p+\frac{n-k}{2},l-\frac{n}{2}}(\RR^{n-k}_{x''};\RR^k_{\xi'};\RR^n_{\eta}),
\end{equation}
where a conic neighborhood of $\Lambda_0\cap\Lambda_1$ corresponds to
a conic neighborhood of $\eta=0$ in $\RR^k_{\xi'}\times\RR^n_{\eta}$
(so now $\xi'$ is the `large variable' on the parameter space),
and the $x''$ dependence can again be replaced by $y''$ dependence. (To see
this form of parameterization, write $z''=x'$, $z'=(-y',x''-y'')$, $z'''=x''$
then $\Lambda_1=N^*\{z''=0\}$, $\Lambda_0=N^*\{z'=0,\
z''=0\}$. Replacing $x''$ by $y''$ in the definition
of $z'''$ gives the other parameterization.)

We first note the action of pseudodifferential operators applied from
either factor to this pair:

\begin{lemma}\label{lemma:one-sided-psdo-comp}
Let $\Lambda_0,\Lambda_1$ as in \eqref{eq:one-sided-mapping}, with
$x$'s being the left variables. Then for $Q\in\Psi^s(\RR^n)$ (of proper
support), and for $K\in I^{p,l}(\Lambda_0,\Lambda_1)$, $QK\in
I^{p+s,l}(\Lambda_0,\Lambda_1)$ while $KQ\in I^{p,l+s}(\Lambda_0,\Lambda_1)$.
\end{lemma}

\begin{proof}
As before, it suffies to consider kernels $K$ of the form
\eqref{eq:one-sided-param}, or its $y''$-dependent analogue, for
kernels in $I^p(\Lambda_1)$, as well as those in $I^{p+l}(\Lambda_0)$
with wave front set disjoint from $\Lambda_0\cap\Lambda_1$ (thus away
from covectors with vanishing dual-to-$y$ components), can easily
be treated by standard results.

In order to find $QK$, write $K$ in the form
\eqref{eq:one-sided-param}, but with $x''$ dependence replaced by
$y''$ dependence. Writing $Q$ as left quantization,
$$
Qv(z)=(2\pi)^{-n}\int
e^{i(z'\cdot\zeta'+z''\cdot\zeta'')}q(z',z'',\zeta',\zeta'')(\cF v)(\zeta',\zeta'')\,d\zeta'\,d\zeta'',
$$
and using that \eqref{eq:one-sided-param} with $x''$-dependence
replaced by $y''$-dependence gives, when applied to a $\CI_c$ function
$u$,
$$
\left(\cF_{\xi',\eta''}\left(\int e^{i[-y'\cdot\eta'-y''\cdot\eta'']} a(y'',
\xi',\eta',\eta'')\,u(y',y'')\,d\eta'\,dy'\,dy''\right)\right)(x',x''),
$$
we conclude that
\begin{equation*}\begin{aligned}
QKu(z)=&(2\pi)^{-n}\int
e^{i(z'\cdot\zeta'+z''\cdot\zeta'' -y'\cdot\eta'-y''\cdot\zeta'')}\\
&\qquad q(z',z'',\zeta',\zeta'') a(y'',
\zeta',\eta',\zeta'')\,u(y',y'')\,d\eta'\,dy'\,dy''\,d\zeta'\,d\zeta'',
\end{aligned}\end{equation*}
so the Schwartz kernel of $QK$ is given by the oscillator integral
$$
QK=(2\pi)^{-n}\int
e^{i(z'\cdot\zeta'-y'\cdot\eta'+(z''-y'')\cdot\zeta'')}q(z',z'',\zeta',\zeta'') a(y'',
\zeta',\eta',\zeta'')\,d\eta'\,d\zeta'\,d\zeta'',
$$
which is of the desired form.

Composition from the right can be checked similarly, using
\eqref{eq:one-sided-param} as stated, with $x''$-dependence.
\end{proof}

Most crucially we need mapping properties of these operators on
Sobolev spaces.

\begin{prop}\label{prop:one-sided-bded}
Let $\Lambda_0,\Lambda_1$ as in \eqref{eq:one-sided-mapping}, with
$x$'s being the left variables. Then for $K\in
I^{p,l}(\Lambda_0,\Lambda_1)$
with wave front set disjoint from $o_{\RR^n}\times
T^*\RR^n$,
and for $m,m'\in\RR$,
\begin{equation}\label{eq:one-sided-bded-conditions}
p+l<m+m'-\frac{k}{2}\ \text{and}\ p<m-\frac{n}{2}\Rightarrow K\in\cL(H^{m'},H^{-m}).
\end{equation}
\end{prop}

\begin{rem}
Note that the first condition of \eqref{eq:one-sided-bded-conditions}
is almost exactly the statement that a distribution in
$I^{p+l}(\Lambda_0)$ with wave front set away from
$\Lambda_0\cap\Lambda_1$ is bounded from $H^{m'}$ to $H^{-m}$, with
`almost' referring to the loss of the normally allowed equality, cf.\
Proposition~\ref{prop:diag-main-bded} and the remarks afterwards. On the
other hand, the second condition in
\eqref{eq:one-sided-bded-conditions} is exactly
the condition that an element of $I^p(\Lambda_1)$ maps distributions
(or even just $\CI$, for that matter)
into $H^{-m}$.
\end{rem}

\begin{proof}
We first remark that \eqref{eq:one-sided-bded-conditions} is
equivalent to the combination of the two conditions: either
\begin{equation}\label{eq:one-sided-0-stronger}
l\geq m'-\frac{k}{2}+\frac{n}{2}\ \text{and}\ p+l<m+m'-\frac{k}{2},
\end{equation}
or
\begin{equation}\label{eq:one-sided-1-stronger}
l<m'-\frac{k}{2}+\frac{n}{2}\ \text{and}\ p<m-\frac{n}{2}.
\end{equation}
Indeed, \eqref{eq:one-sided-bded-conditions} automatically implies
these two, and conversely, if $l\geq m'-\frac{k}{2}+\frac{n}{2}$ then
subtracting the first inequality from the second yields
$p<m-\frac{n}{2}$ while if $l<m'-\frac{k}{2}+\frac{n}{2}$ then adding
the inequalities yields $p+l<m+m'-\frac{k}{2}$.

Now, in view of Lemma~\ref{lemma:one-sided-psdo-comp}, at the cost of
replacing $p$ by $p-m$ and $l$ by $l-m'$, as we now do,
it suffices to consider $L^2$-boundedness. Further, one may assume
that $K$ is of the form \eqref{eq:one-sided-param}, with $a$ supported
in the region $\langle\eta\rangle\leq\langle\xi'\rangle$. We claim that if
we let $A(x'',\eta'')$ be the operator on $\RR^k_{x'}$ given by
$$
(A(x'',\eta'')u)(x')=\int e^{i[x'\cdot\xi'-y'\cdot\eta']} a(x'',
\xi',\eta',\eta'')u(y')\,d\xi'\,d\eta'\,dy',\ u\in\CI_c(\RR^k),
$$
and if either set of conditions \eqref{eq:one-sided-0-stronger}, resp.\ \eqref{eq:one-sided-1-stronger}, is satisfied then
\begin{equation}\label{eq:op-valued-symbol}
A\in S^0(\RR^{n-k}_{x''};\RR^{n-k}_{\eta''};\cL(L^2(\RR^k),L^2(\RR^k))),
\end{equation}
i.e.\ it is an operator-valued symbol of order $0$, which thus by the
operator-valued version of the standard calculus, see
\cite[Section~18.1, Remark~2]{Hor}, gives a bounded
operator
$$
\cL(L^2(\RR^{n-k};L^2(\RR^k));L^2(\RR^{n-k};L^2(\RR^k)))=\cL(L^2(\RR^n),L^2(\RR^n)),
$$
proving the proposition.

But with $\cF'$ denoting the Fourier transform in the primed
variables,
$$
(A(x'',\eta'') u)(x')=(2\pi)^k\left((\cF')^{-1} (\int
a(x'',.,\eta',\eta'')(\cF u)(\eta')\,d\eta')\right)(x'),
$$
i.e.\ $\cF'A(\cF')^{-1}$ has Schwartz kernel $(2\pi)^k
a(x'',\xi',\eta',\eta'')$, with the action in the primed
variables.
First we check that $A(.,.)$ is a uniformly bounded family
of bounded operators (and indeed, a uniformly bounded family
of Hilbert-Schmidt operators). This follows if we show that
$$
a(x'',\xi',\eta',\eta'')\in
L^\infty(\RR^{2(n-k)}_{x'',\eta''};L^2(\RR^{2k}_{\xi',\eta'})),
$$
which in turn follows if for some $\delta>0$,
$$
\langle\xi'\rangle^{\frac{k}{2}+\delta}\langle\eta'\rangle^{\frac{k}{2}+\delta}a\in L^\infty(\RR^{2n}).
$$
But
\begin{equation}\label{eq:weighted-sup-estimate-1}
\langle\xi'\rangle^{\frac{k}{2}+\delta}\langle\eta'\rangle^{\frac{k}{2}+\delta}|a|
\leq C\langle\xi'\rangle^{p-m+\frac{n}{2}+\delta}\langle\eta\rangle^{l-m'-\frac{n-k}{2}+\delta}
\end{equation}
Now, if $l-m'-\frac{n-k}{2}\geq 0$ then
$\langle\eta\rangle^{l-m'-\frac{n-k}{2}+\delta}\leq
\langle\xi'\rangle^{l-m'-\frac{n-k}{2}+\delta}$, and thus
\begin{equation}\label{eq:weighted-sup-estimate-2}
\langle\xi'\rangle^{\frac{k}{2}+\delta}\langle\eta'\rangle^{\frac{k}{2}+\delta}|a|
\leq C\langle\xi'\rangle^{p+l-m-m'+\frac{k}{2}+2\delta},
\end{equation}
and thus is bounded since $p+l-m-m'+\frac{k}{2}<0$ means that one can
take sufficiently small $\delta>0$ to still have
$p+l-m-m'+\frac{k}{2}+2\delta\leq 0$. On the other hand, if
$l-m'-\frac{n-k}{2}< 0$ then $p-m+\frac{n}{2}<0$ as well, so one may
choose $\delta>0$ sufficiently small so that the right hand side of
\eqref{eq:weighted-sup-estimate-1} is bounded.

Since $D_{\eta''_j}$, $\eta''_jD_{\eta''_i}$, $D_{x''_j}$ preserve the
symbolic order of $a$, analogous properties follow when these
differential
operators are applied to $A(.,.)$ iteratively, implying that
\eqref{eq:op-valued-symbol} holds, which in turn completes the proof
of the proposition.
\end{proof}

If the role of the $x$ and $y$ variables is reversed one has
\begin{equation}\begin{aligned}\label{eq:one-sided-mapping-mod}
&\hat\Lambda_1=\{y'=0,\ \eta''=0,\ \xi'=0,\ \xi''=0\}=N^*\{y'=0\},\\
&\Lambda_0=\{x'=0=y',\ x''=y'',\ \xi''=-\eta''\}=N^*\{x'=0=y',\ x''=y''\},
\end{aligned}\end{equation}
as the modified model. Either essentially repeating the arguments
given above,
or noting that if $K\in
I^{p,l}(\Lambda_0,\hat\Lambda_1)$ then its adjoint is in
$I^{p,l}(\Lambda_0,\Lambda_1)$, and thus via dualization one obtains
mapping properties of $K$ from Proposition~\ref{prop:one-sided-bded},
one has

\begin{prop}\label{prop:one-sided-bded-mod}
Let $\Lambda_0,\hat\Lambda_1$ as in \eqref{eq:one-sided-mapping-mod}, with
$x$'s being the left variables. Then for $K\in
I^{p,l}(\Lambda_0,\Lambda_1)$
with wave front set disjoint from $o_{\RR^n}\times
T^*\RR^n$,
and for $m,m'\in\RR$,
\begin{equation}\label{eq:one-sided-bded-conditions-mod}
p+l<m+m'-\frac{k}{2}\ \text{and}\ p<m'-\frac{n}{2}\Rightarrow K\in\cL(H^{m'},H^{-m}).
\end{equation}
\end{prop}

Even if a distribution is Lagrangian associated to $\Lambda_0$ (i.e.\
has no singularity at $\Lambda_1$), the fact that $\Lambda_0$
intersects $T^*\RR^n\times o_{\RR^n}$ means that the standard results
on mapping properties do not apply. However, one {\em can} regard this
distribution as a paired Lagrangian associated to
$(\Lambda_0,\Lambda_1)$ and apply the previous propositions:

\begin{cor}\label{cor:one-sided-bded}
Let $\Lambda_0$ be as in \eqref{eq:one-sided-mapping}, with
$x$'s being the left variables. Then for $K\in
I^p(\Lambda_0)$ with wave front set disjoint from $o_{\RR^n}\times
T^*\RR^n$
and for $m,m'\in\RR$,
\begin{equation}\label{eq:Lambda_0-one-sided-bded-conditions}
p<m+m'-\frac{k}{2}\ \text{and}\ p<m\Rightarrow K\in\cL(H^{m'},H^{-m}).
\end{equation}
In case $K\in
I^p(\Lambda_0)$ with wave front set disjoint from $T^*\RR^n\times
o_{\RR^n}$, then the conditions become
\begin{equation}\label{eq:Lambda_0-one-sided-bded-conditions-mod}
p<m+m'-\frac{k}{2}\ \text{and}\ p<m'\Rightarrow K\in\cL(H^{m'},H^{-m}).
\end{equation}
\end{cor}

\begin{proof}
With $\Lambda_1$ as in \eqref{eq:one-sided-mapping},
$\Lambda_0\cap\Lambda_1$ has codimension $n$ in either of these two
Lagrangians, and thus by Lemma~\ref{lemma:Lambda_0-Lag},
$I^p(\Lambda_0)\subset
I^{p-\frac{n}{2},\frac{n}{2}}(\Lambda_0,\Lambda_1)$. Thus by
Proposition~\ref{prop:one-sided-bded}, $K$ is bounded as claimed
provided
$p<m+m'-\frac{k}{2}$ and $p<m$, which completes the proof.
\end{proof}

As an example, with $\codim Y=\codimY$, $\dim X=n$,
consider
$$
f\in I^{[-s_0]}(Y)=I^{-s_0-(\dim X-2\codimY)/4}(N^*Y).
$$
Then the pullback $\pi_L^* f$ of $f$ to $X\times X$, via the left projection to
$X$, is in $I^{[-s_0]}=I^{-s_0-\dim Y/2}(N^*(Y\times X))$. Since for
$A\in\Psi^s(X)$ one has $K_A\in I^s(N^*\diag)$ (with $K_A$ denoting the Schwartz kernel of
$A$), the Schwartz kernel $K_{fA}$ of $fA$ is $(\pi_L^*
f)K_A$, and by \eqref{eq:product-conormal} one has
\begin{equation}\begin{aligned}\label{eq:psdo-with-sing-coeff}
K_{fA}\in &I^{s,-s_0+\codimY/2}(N^*(\diag\cap(Y\times
X)),N^*\diag)\\
&\qquad\qquad+I^{-s_0-\dim Y/2,s+n/2}(N^*(\diag\cap(Y\times
X)),N^*(Y\times X)).
\end{aligned}\end{equation}
Similar results apply to $Af$, with the left and the right factors
interchanged.

In the special case $A=\Id$ we get
\begin{equation}\begin{aligned}\label{eq:mult-with-sing-coeff}
K_{f\,\Id}\in &I^{0,-s_0+\codimY/2}(N^*(\diag\cap(Y\times
X)),N^*\diag)\\
&\qquad\qquad+I^{-s_0-\dim Y/2,n/2}(N^*(\diag\cap(Y\times
X)),N^*(Y\times X)).
\end{aligned}\end{equation}
In this case one could write the multiplication also as a
multiplication from the right factor, and thus deduce that the second
summand can be dropped. (This also follows directly
from \cite{Greenleaf-Uhlmann:Estimates}.) However, this has no impact
on the following consequence:

\begin{prop}\label{prop:sing-mult-bded}
Multiplication by $f\in I^{[-s_0]}(Y)=I^{-s_0-(\dim
  X-2\codimY)/4}(N^*Y)$ is bounded $H^s\to H^s$ provided $s_0>\codim Y$
and $-s_0+\codimY/2<s<s_0-\codimY/2$.
\end{prop}

\begin{proof}
In view of \eqref{eq:mult-with-sing-coeff} and Propositions~\ref{prop:diag-main-bded}
and \ref{prop:one-sided-bded} and
Corollary~\ref{cor:one-sided-bded}, multiplication by $f$ is bounded
$H^s\to H^s$ provided
\begin{equation*}\begin{aligned}
&-s_0+\codimY/2<-\codimY/2,\\
&-s_0-\dim Y/2<s-n/2,\\
&-s_0+\codimY/2<-s,
\end{aligned}\end{equation*}
which gives exactly the constraints in the proposition.
\end{proof}

\section{Elliptic estimates}\label{sec:elliptic}
In this section we discuss microlocal elliptic estimates, which help
take care of the regions of phase space one would like to think of as
`irrelevant' for wave propagation purposes. Here, and in the next
section, we denote the position (base) variable by $x$, the dual
variable by $\xi$, and use local coordinates in which $Y$ is given by $\{x'=0\}$.

So suppose that $g\in
I^{[-s_0]}(Y)$, $\codim Y=\codimY$, with $G$ the dual metric. For simplicity, we reduce the problem from $\Box$ to
$$
P=(\det g)^{1/2}\Box=\sum_{ij} D_i(\det g)^{1/2} G_{ij} D_j.
$$
If $\Box u=f\in H^s$, then by Proposition~\ref{prop:sing-mult-bded}
multiplication by $(\det g)^{1/2}\in
I^{[-s_0]}(Y)$
preserves $H^s$ if
\begin{equation}\begin{aligned}\label{eq:ell-reduction-constraint}
&s_0>\codimY,\\
&-s_0+\codimY/2<s<s_0-\codimY/2.
\end{aligned}\end{equation}
Thus,
$$
Pu=(\det g)^{1/2}f\in H^s;
$$
so under these constraints, we may instead study the equation
$Pu=\tilde f$. We write
\begin{equation}\label{eq:reduced-op}
g_{ij}=(\det g)^{1/2} G_{ij},\ P=\sum_{ij}D_i g_{ij} D_j,
\end{equation}
and note that $P$ is formally self-adjoint with respect to the
Euclidean inner product.

For $A\in\Psi^{2s-2}(X)$; we need to compute the Schwartz
kernel of $ P A$ (or $A P$) as a (sum of) paired Lagrangian distribution(s).
The Schwarz kernel $K_{D_iA}$ of $D_i A$ is $D_{i,L}K_A$
(where the subscript $L$ denotes the derivative acting on the left
factor of $X\times X$), while the
Schwartz kernel $K_{AD_i}$ of $A D_i$ is $-D_{i,R}K_A$, we deduce that
the Schwartz kernel of $ P A$, resp.\ $A P$, is
\begin{equation*}\begin{split}
K_{ P A}=\sum D_{i,L}g_{ij,L}D_{j,L}K_A,\qquad  K_{A P }=\sum D_{j,R} g_{ij,R}D_{i,R}K_A.
\end{split}\end{equation*}
Here $g_{ij,L}$, resp.\ $g_{ij,R}$, is the pullback of $g_{ij}$ from
the left, resp.\ right, factor.
Now, $K_A\in I^{2s-2}(N^*\diag)$, so
$D_{i,L}K_A, D_{i,R}K_A\in I^{2s-1}(N^*\diag)$.
Now as $g_{ij}\in I^{[-s_0]}(Y)$, by \eqref{eq:psdo-with-sing-coeff}
(with the left and right factors interchanged in the first case),
\begin{equation*}\begin{aligned} 
g_{ij,R}D_{i,R}K_A\in  &I^{2s-1,-s_0+\codimY/2}(N^*(\diag\cap(X\times 
Y)),N^*\diag)\\
&+I^{-s_0-\dim Y/2,2s-1+n/2}(N^*(\diag\cap(X\times 
Y)),N^*(X\times Y)),
\end{aligned}\end{equation*} 
and
\begin{equation*}\begin{aligned}
g_{ij,L}D_{i,L}K_A\in  &I^{2s-1,-s_0+\codimY/2}(N^*(\diag\cap(Y\times
X)),N^*\diag)\\
&+I^{-s_0-\dim Y/2,2s-1+n/2}(N^*(\diag\cap(Y\times
X)),N^*(Y\times X)),
\end{aligned}\end{equation*}
so in particular
away from the intersections, these are Lagrangian associated to
the conormal bundles of
$X\times Y$ (or $Y\times X$), $\diag$, as well as their intersection, $(Y\times
Y)\cap\diag$, with orders $I^{[-s_0]}=I^{-s_0-2\dim Y/4}$,
$I^{[2s-1]}=I^{2s-1}$ and $I^{[-s_0+2s-1]}=I^{-s_0+2s-1+\codimY/2}$.
Applying $D_{j,R}$, resp.\ $D_{j,L}$ increases the orders on all
Lagrangians, i.e.\ in terms of paired Lagrangians it increases the
first order (corresponding to the main Lagrangian, i.e.\ the second in
the pair,
dictating the singular behavior) by $1$,
see in particular Lemma~\ref{lemma:one-sided-psdo-comp}. Thus, we
conclude:

\begin{lemma}
For $g\in
I^{[-s_0]}(Y)$, $A\in\Psi^{2s-2}(X)$ with compactly supported Schwartz
kernel,
\begin{equation}\begin{split}\label{eq:right-comp-orders}
K_{A P}\in  
&I^{2s,-s_0+\codimY/2}(N^*(\diag\cap(X\times  
Y)),N^*\diag)\\
&\qquad+I^{-s_0+1-\dim Y/2,2s-1+n/2}(N^*(\diag\cap(X\times  
Y)),N^*(X\times Y)),
\end{split}\end{equation}  
and
\begin{equation}\begin{split}\label{eq:left-comp-orders}
K_{ P A}\in 
&I^{2s,-s_0+\codimY/2}(N^*(\diag\cap(Y\times 
X)),N^*\diag)\\
&\qquad+I^{-s_0+1-\dim Y/2,2s-1+n/2}(N^*(\diag\cap(Y\times 
X)),N^*(Y\times X)).
\end{split}\end{equation} 
\end{lemma}

Now consider $K_{A P}$. Note that microlocally away from the intersection of the two
Lagrangians, microlocally near $N^*(\diag\cap(X\times
Y))$,
$$
I^{2s,-s_0+\codimY/2}(N^*(\diag\cap(X\times
Y)),N^*\diag)
$$
is just
$$
I^{2s-s_0+\codimY/2}(N^*(\diag\cap(X\times
Y))),
$$
and $N^*(\diag\cap(X\times Y))=N^*(\diag\cap (Y\times
Y))$ intersects $T^*\RR^n\times o_{\RR^n}$ at $N^*(Y\times X)$. Thus,
we need to use Corollary~\ref{cor:one-sided-bded} as well when
discussing boundedness between Sobolev spaces.
In view of Propositions~\ref{prop:diag-main-bded}
and \ref{prop:one-sided-bded-mod} and
Corollary~\ref{cor:one-sided-bded}, microlocally away from $N^*\diag$,
$A P$ is bounded from $H^{s-\ep_0}$ to $H^{-s+\ep_0}$ provided
\begin{equation}\begin{aligned}\label{eq:good-interact-Sob-prelim}
&-s_0+2s+\codimY/2< 2s-2\ep_0-\codimY/2,\\
&-s_0+1-\dim 
Y/2<s-\ep_0-\frac{n}{2}\ \text{and}\\
&-s_0+2s+\codimY/2<s-\ep_0,
\end{aligned}\end{equation} 
i.e.
\begin{equation}\begin{aligned}\label{eq:good-interact-Sob}
&\codimY+2\ep_0< s_0,\\
&s>-s_0+\ep_0+1+\codimY/2\ \text{and}\\
&s<s_0-\ep_0-\codimY/2.
\end{aligned}\end{equation}
Notice that these inequalities imply
\eqref{eq:ell-reduction-constraint}.
Note that if the first inequality holds then
$$
-s_0+\ep_0+1+\codimY/2<-\codimY/2-\ep_0+1<1-\codimY/2,
$$
so when $s\geq 1-\codimY/2$, the second inequality in
\eqref{eq:good-interact-Sob} is automatic when the first
holds. Moreover, if the stronger inequality $1+\codimY+2\ep_0< s_0$ is
assumed in place of the first in \eqref{eq:good-interact-Sob} (we need
the stronger inequality below in the hyperbolic setting), then
for $s\geq -\codimY/2$ it assures that the second one holds.
An
analogous (in some sense, dual) computation applies to $ P A$, using
Proposition~\ref{prop:one-sided-bded}  in place of
Proposition~\ref{prop:one-sided-bded-mod}, and
yielding the same constraints, \eqref{eq:good-interact-Sob}. We state
these results as a lemma:

\begin{lemma}\label{lemma:off-diag-elliptic}
For $g\in
I^{[-s_0]}(Y)$, $A\in\Psi^{2s-2}(X)$ with compactly supported Schwartz
kernel, $ P A,A P$ are, microlocally
away from $N^*\diag$, bounded from $H^{s-\ep_0}$ to $H^{-s+\ep_0}$
provided \eqref{eq:good-interact-Sob} is satisfied.
\end{lemma}

Microlocal elliptic regularity is now a straightforward
consequence. Consider $q_0\notin\Sigma$.
We shall assume that $q_0\notin\WF^{s-1/2}(u)$, thus there
is a conic neighborhood $O$ of $q_0$ on which $u$ is microlocally
in $H^{s-1/2}$; we may take $O$ disjoint from $\Sigma$. 
With $p$ the principal symbol of
$ P$, $p(q_0)\neq 0$, and we may assume that $\sign p$ is constant
on $O$.
We take $A\in\Psi^{2s-2}(X)$ with
principal symbol $a_0^2$ elliptic at $q_0$, supported close to $q_0$, in the region
where $\sign p$ is constant, with
$\WF'(A)\subset O$ and $A=A^*$.
Then the principal symbol of $A P$ on $N^*\diag$ is
$$
a_0^2p\in S^{2s}(\RR^n_\xi;I^{-s_0-n/4+k/2}(\RR^n_x;N^*\{x'=0\}))=S^{2s}(\RR^n_\xi;I^{[-s_0]}(\RR^n_x;N^*\{x'=0\})).
$$
By
assumption,
$$
p\in S^{2}(\RR^n_\xi;I^{-s_0-n/4+k/2}(\RR^n_x;N^*\{x'=0\})).
$$
has a fixed (non-zero) sign, $\sign p(q_0)$, on $\supp a_0$, so by
Lemma~\ref{lemma:sqrt}, for $\ep_1>0$ (which we take as small as convenient),
$$
a_0^2 p=(\sign p(q_0))b^2,\ b=a_0\sqrt{|p|}\in S^{s}(\RR^n_\xi;I^{-s_0-n/4+k/2+\ep_1}(\RR^n_x;N^*\{x'=0\})).
$$
Let
$$
B\in I^{s,-s_0+\codimY/2+\ep_1}(N^*(\diag\cap(X\times  
Y)),N^*\diag)
$$
with principal symbol $b$; then by
Proposition~\ref{prop:diag-main-compose}, taking into account that
$2(-s_0+\codimY/2)<-\codimY-4\ep_0<-1$ so there is a full order gain
in the symbolic calculation,
$$
A P=(\sign p(q_0))B^*B+F
$$
with
\begin{equation}\begin{aligned}\label{eq:elliptic-error-space-one}
F\in  
&I^{2s-1,1-s_0+\codimY/2+\ep_1}(N^*(\diag\cap(X\times 
Y)),N^*\diag)\\
&\qquad+I^{-s_0+1-\dim Y/2,2s-1+n/2}(N^*(\diag\cap(X\times 
Y)),N^*(X\times Y)),
\end{aligned}\end{equation}
so $F$ has order corresponding to $\Psi^{2s-1}(X)$ on the conormal
bundle of the diagonal, and elsewhere it has the same orders as
$A P$ had, apart from the $\ep_1>0$ loss from the symbolic
construction of $B$.
In view of Propositions~\ref{prop:diag-main-bded} and
\ref{prop:one-sided-bded-mod} and Corollary~\ref{cor:one-sided-bded},
for $\ep_0'=\min(1/2,\ep_0)$ and $\ep_1>0$ sufficiently small (since
we have strict inequalities in \eqref{eq:good-interact-Sob}), $F$
is bounded from
$H^{s-\ep'_0}$ to
$H^{-s+\ep'_0}$ (here we possibly reduced $\ep_0$ to $\ep_0'$ in order
to deal with the diagonal
singularity, which we thus far ignored),
if \eqref{eq:good-interact-Sob} holds.
Thus, subject to these
limitations on $\ep_0'$, $s_0$ and $s$,
$\langle Fu,u\rangle$ is bounded by the a priori assumptions.
Since the constraint on $\ep_0'$ is purely due to the diagonal
singularity, it is convenient to write
\begin{equation}\begin{aligned}\label{eq:elliptic-error-space}
&F=F'+F'',\\
&F'\in  
I^{2s-1,1-s_0+\codimY/2+\ep_1}(N^*(\diag\cap(X\times 
Y)),N^*\diag),\\
&F''\in I^{-s_0+1-\dim Y/2,2s-1+n/2+\ep_1}(N^*(\diag\cap(X\times 
Y)),N^*(X\times Y)),
\end{aligned}\end{equation}
with the wave front set of $F'$ in a prescribed arbitrary conic neighborhood
of $N^*\diag$ -- note that away from $N^*\diag$, elements of
$$
I^{2s-1,1-s_0+\codimY/2+\ep_1}(N^*(\diag\cap(X\times 
Y)),N^*\diag)
$$
are in $I^{-s_0+1-\dim Y/2,2s-1+n/2+\ep_1}(N^*(\diag\cap(X\times 
Y)),N^*(X\times Y))$, so can always be regarded as part of $F''$.
Concretely, as $\WF'(A)\subset O$, so the wave front set of
$K_{A P}$ intersects $N^*\diag$ only in $O\times O'$, we demand, as
we may, that
$$
\WF(K_{B}),\ \WF(K_{F'})\subset O\times O',
$$
where the prime on $O$ denotes the usual twisting, i.e.\ the switch of
the sign of the second covector. (Note that this means in particular
that $\WF(K_{F'})$ does not contain covectors in $(T^*X\setminus
o)\times o$ and $o\times (T^*X\setminus
o)$.)
With such a decomposition, for $\ep_1>0$ sufficiently small, $F''$
is bounded from $H^{s-\ep_0}$ to $H^{-s+\ep_0}$ so $\langle
F''u,u\rangle$ is bounded,
while $u$ being in
$H^{s-1/2}$ on $O$,
$\langle F'u,u\rangle$ is bounded by the a priori assumptions.
Further,
with
$Q\in\Psi^{s-2}(X)$ elliptic with positive principal symbol $q$, with
parametrix $G\in\Psi^{2-s}(X)$ with positive principal symbol $g$, such
that $GQ=\Id+R$, $R\in\Psi^{-\infty}(X)$, and for $\delta>0$,
\begin{equation}\begin{aligned}\label{eq:inhomog-elliptic-term}
|\langle Au, P
u\rangle|&\leq |\langle G^*Au,Q P u\rangle|+|\langle Au,R P
u\rangle|\\
&\leq \delta\|G^*Au\|^2+\delta^{-1}\|Q P u\|^2+|\langle Au,R P
u\rangle|,
\end{aligned}\end{equation}
where the last two terms are bounded by the a
priori assumptions. In order to absorb the $G^*A\in\Psi^{s}(X)$
term, to deal with the regularizer, as well as to facilitate the direct translation to a wave front
set statement, it is convenient to replace $B^*B$ by
\begin{equation}\label{eq:elliptic-mod-B}
B_1^*B_1+B_2^* B_2+c^2(G^*A)^*(G^*A)
\end{equation}
where $c>0$ is a small constant,
\begin{equation*}\begin{split}
&B_1\in I^s(N^*\diag),\\
&B_2\in I^{s,-s_0+\codimY/2+\ep_1}(N^*(\diag\cap(Y\times 
Y)),N^*\diag).
\end{split}\end{equation*}
This is achieved as follows. Let $\rho$ be a positive elliptic homogeneous
degree $1$ function on $T^*X\setminus o$. Since $\supp a_0$ is compact, disjoint from
$\Sigma$, $|p|\geq c_0^2\rho^2$ on it for some $c_0>0$. Further, the principal
symbol $\tilde g$ of $G^*$ satisfies $|\tilde g|\leq C'\rho^{2-s}$, and
that of $a_0$ satisfies $|a_0|\leq C''\rho^{s-1}$, so
the principal symbol $a_0^4 g^2$ of $(G^*A)^*(G^*A)$ is then bounded
by $C^2\rho^2a_0^2$. Then let $c=\frac{c_0}{2C}$, so the symbol of
$c^2(G^*A)^*(G^*A)$ is bounded by $\frac{c_0^2}{4} \rho^2 a_0^2$.
Now let
$$
b_1=\frac{c_0}{2}a_0\rho ,\ b_2=\left(|p|\rho^{-2}-\frac{c_0^2}{4}-c^2a_0^4 g^2\rho^{-2}\right)^{1/2}a_0\rho.
$$
Then on $\supp a_0$, the factor inside the parentheses is a
homogeneous degree zero $\CI$
function bounded below by a positive constant, thus the square root is
$\CI$. Taking $B_j$ with principal symbols $b_j$,
\eqref{eq:elliptic-mod-B} has principal symbol $|p|a_0^2$, hence
$$
A P=(\sign p(q_0))\left( B_1^*B_1+B_2^* B_2+c^2(G^*A)^*(G^*A)\right)+F,
$$
with $F$ satisfying \eqref{eq:elliptic-error-space} (but possibly
different from the $F$ given by $B^*B$).
Then
\begin{equation*}\begin{aligned}
&\langle P u,Au\rangle
=\langle A P u,u\rangle=(\sign p(q_0))\Big(\|B_1u\|^2+\|B_2 u\|^2+c\|G^*Au\|^2\Big)+\langle Fu,u\rangle,
\end{aligned}\end{equation*}
which we justify via a standard regularization argument, recalled below,
so using \eqref{eq:inhomog-elliptic-term} to estimate the left hand side
from above, and taking $\delta>0$ sufficiently small,
$\delta\|G^*Au\|^2$
can be absorbed in the right hand side.
This
gives the conclusion that $B_j u\in L^2$ for $j=1,2$,
which allows us to conclude that $\WF^s(u)$ is
disjoint from the elliptic set of $B_1$.

Finally, the regularization argument is to replace $A$ by $A_r=\Lambda_r
A\Lambda_r$, $r\in [0,1]$, where $\Lambda_r\in\Psi^{-1}$ for
$r>0$, $\Lambda_r$ is uniformly bounded in $\Psi^0$, and $\Lambda_r\to
\Id$ in $\Psi^\ep$ for $\ep>0$, and thus strongly in $L^2$; one may
take $\Lambda_r$ formally self-adjoint for convenience. (One can
for instance take $\Lambda_r$ to be a quantization of
$(1+r\rho)^{-1}$; and then replace it by its self-adjoint part which
does not affect the principal symbol or the boundedness and
convergence properties, as in Section~\ref{sec:structure}.)
Then $\Lambda_rA\Lambda_rP$ has the same
principal symbol, uniformly in $\Psi^{2s}$, as
$$
\Lambda_r(\sign
p(q_0))\Big(B_1^*B_1+B_2^*B_2+c(G^*A)^*(G^*A)\Big)\Lambda_r,
$$
and correspondingly
\begin{equation*}\begin{aligned}
&\langle P u,A_ru\rangle
=\langle A_r P u,u\rangle\\
&\qquad=(\sign p(q_0))\Big(\|B_1\Lambda_r
u\|^2+\|B_2 \Lambda_r u\|^2+c\|G^*A\Lambda_r u\|^2\Big)+\langle F_ru,u\rangle,
\end{aligned}\end{equation*}
where $F_r$ is uniformly bounded in $\Psi^{2s-1}$, and is in
$\Psi^{2s-3}$ for $r>0$. Here the calculations such as the first
equality and $\|\Lambda_r B_1 u\|^2=\langle \Lambda_r
B_1^*B_1\Lambda_r u,u\rangle$ follow since for $r>0$ on $O$, which
contains the (conic or essential)
support of $a_0$,
$u$ is in $H^{s-1/2}$ by the a priori assumptions, and the sum of the
diagonal orders
of the operators involved is $\leq 2s-1$. Now letting $r\to 0$ gives
uniform bounds for $\|B_j\Lambda_r u\|_{L^2}$, and thus proves $B_j
u\in L^2$ in view of the weak compactness of the unit ball in $L^2$
and since $B_j\Lambda_r u\to B_j u$ in distributions.
As $q_0\in\Sigma$ was
arbitrary, we conclude that

\begin{lemma}\label{lemma:mic-elliptic}
Suppose that \eqref{eq:good-interact-Sob} holds.
If $u\in H^{s-\ep_0}_{\loc}$, $ P u\in H^{s-2}_{\loc}$,
then $\WF^s(u)\subset\Sigma\cup \WF^{s-1/2}(u)$.
\end{lemma}

Now one can iterate this, gradually increasing $s$ by $\leq 1/2$; here
we also return to $\Box$ instead of $P$:

\begin{prop}\label{prop:mic-elliptic}
Suppose that
$\codimY+1+2\ep_0< s_0$ and $-\codimY/2<s<s_0-\ep_0-\codimY/2$.
If $u\in H^{s-\ep_0}_{\loc}$, $\Box u\in H^{s-2}_{\loc}$,
then $\WF^s(u)\subset\Sigma$.
\end{prop}

\begin{proof}
First apply Lemma~\ref{lemma:mic-elliptic} with
$s'=\min(s-\ep_0+1/2,s)\leq s$ in place of $s$ (and $\ep_0$ unchanged); then
$$
s'\geq s-\ep_0+1/2>-\codimY/2-\ep_0+1/2-(s_0-\codimY-1-2\ep_0)>-s_0+\codimY/2+\ep_0+1,
$$
so the second inequality in \eqref{eq:good-interact-Sob} holds, and all
others hold because $s'\leq s$.
Since $u\in
H^{s-\ep_0}_{\loc}$ implies that $\WF^{s'-1/2}(u)=\emptyset$, the
conclusion of the lemma gives $\WF^{s'}(u)\subset \Sigma$. Now repeat
this argument with $s''=\min(s'+1/2,s)\in[s',s]$, so
\eqref{eq:good-interact-Sob} holds for $s''$ in place of $s$, to conclude
$\WF^{s''}(u)\subset \Sigma$. An inductive argument gives
$\WF^{s}(u)\subset \Sigma$ in a finite number of steps, as desired.
\end{proof}

\section{Propagation estimate}\label{sec:prop-est}

We now return to the positive commutator propagation estimates, but
unlike the smooth coefficients in Section~\ref{sec:structure}, we consider $g\in
I^{[-s_0]}(Y)$, $\codim Y=\codimY$. We again work with the reduced operator
$P=\sum_{ij}D_i g_{ij}D_j$ given by \eqref{eq:reduced-op}, replacing
$\Box u=f\in H^{s-1}$ by $Pu=(\det g)^{1/2}f\in H^{s-1}$ provided
(in view of Proposition~\ref{prop:sing-mult-bded})
\begin{equation}\begin{aligned}\label{eq:hyp-reduction-constraint}
&s_0>\codimY,\\
&-s_0+\codimY/2<s-1<s_0-\codimY/2.
\end{aligned}\end{equation}

So suppose that $A\in\Psi^{2s-1}(X)$; we need to compute the Schwartz
kernel of $[ P,A]$ as a (sum of) paired Lagrangian distribution(s).
By the remarks at the beginning of Section~\ref{sec:elliptic}, and
writing
$$
[ P,A]=\sum [D_i,A] g_{ij}D_j+\sum D_i [g_{ij},A]D_j+\sum D_i g_{ij}[D_j,A],
$$
the Schwartz kernel of $[ P,A]$ is
\begin{equation*}\begin{split}
K_{[ P,A]}=&-\sum D_{j,R} g_{ij,R}(D_{i,L}+D_{i,R})K_A+\sum D_{i,L}
g_{ij,L}(D_{j,L}+D_{j,R})K_A\\
&\qquad-\sum D_{i,L}D_{j,R} (g_{ij,L}-g_{ij,R})K_A.
\end{split}\end{equation*}
As before, $g_{ij,L}$, resp.\ $g_{ij,R}$, is the pullback of $g_{ij}$ from
the left, resp.\ right, factor.
Now, $K_A\in I^{2s-1}(N^*\diag)$, and
$D_{i,L}+D_{i,R}$ is tangent to the diagonal, so
$(D_{i,L}+D_{i,R})K_A\in I^{2s-1}(N^*\diag)$ still.
Now as $g_{ij}\in I^{[-s_0]}(Y)$, by \eqref{eq:psdo-with-sing-coeff}
(with the left and right factors interchanged),
\begin{equation*}\begin{aligned}
&g_{ij,R}(D_{i,L}+D_{i,R})K_A\\
&\in  I^{2s-1,-s_0+\codimY/2}(N^*(\diag\cap(X\times
Y)),N^*\diag)\\
&\qquad+I^{-s_0-\dim Y/2,2s-1+n/2}(N^*(\diag\cap(X\times
Y)),N^*(X\times Y)),
\end{aligned}\end{equation*}
and applying $D_{j,R}$ increases the orders on all
Lagrangians by $1$,
see in particular Lemma~\ref{lemma:one-sided-psdo-comp}. Thus,
\begin{equation}\begin{split}\label{eq:good-comm-orders}
&\sum D_{j,R} g_{ij,R}(D_{i,L}+D_{i,R})K_A\\
&\in
I^{2s,-s_0+\codimY/2}(N^*(\diag\cap(X\times
Y)),N^*\diag)\\
&\qquad+I^{-s_0+1-\dim Y/2,2s-1+n/2}(N^*(\diag\cap(X\times
Y)),N^*(X\times Y)).
\end{split}\end{equation}
The right hand side is exactly the same space as what we obtained in
\eqref{eq:right-comp-orders} and \eqref{eq:left-comp-orders}.
As in Lemma~\ref{lemma:off-diag-elliptic}, we deduce that
microlocally away from $N^*\diag$, \eqref{eq:good-comm-orders}
is bounded from $H^{s-\ep_0}$ to $H^{-s+\ep_0}$ provided
\eqref{eq:good-interact-Sob} holds.

An
analogous computation applies to $\sum D_{i,L}
g_{ij,L}(D_{j,L}+D_{j,R})K_A$, yielding the same constraints, \eqref{eq:good-interact-Sob}.

A similar computation applies to $D_{i,L}D_{j,R}
(g_{ij,L}-g_{ij,R})K_A$, i.e.\ when $g_{ij}$ is commuted through
$A$. However, while the order on $N^*\diag$ is the same as in the
above cases, the order on the {\em other} Lagrangians is just that of 
$D_{i,L}D_{j,R} g_{ij,L} K_A$ and $D_{i,L}D_{j,R} g_{ij,R} K_A$, i.e.\ the commutator does not proved
additional help as compared to the product. This means a loss of $1$
order on $N^*(\diag\cap(Y\times Y))$ as compared to
\eqref{eq:good-comm-orders},
but no extra loss on $N^*(Y\times
X)$ since $D_{i,R}$ is characteristic there. Concretely, as above,
\begin{equation*}\begin{aligned}
&g_{ij,R}K_A\\
&\in  I^{2s-1,-s_0+\codimY/2}(N^*(\diag\cap(X\times
Y)),N^*\diag)\\
&\qquad+I^{-s_0-\dim Y/2,2s-1+n/2}(N^*(\diag\cap(X\times
Y)),N^*(X\times Y)),
\end{aligned}\end{equation*}
and so, using Lemma~\ref{lemma:one-sided-psdo-comp} for the second
summand on the right hand side,
\begin{equation*}\begin{aligned}
&D_{i,L}D_{j,R} g_{ij,R} K_A\\
&\in  I^{2s+1,-s_0+\codimY/2}(N^*(\diag\cap(X\times 
Y)),N^*\diag)\\
&\qquad+I^{-s_0+1-\dim Y/2,2s+n/2}(N^*(\diag\cap(X\times 
Y)),N^*(X\times Y)),
\end{aligned}\end{equation*} 
Similarly,
\begin{equation*}\begin{aligned}
&D_{i,L}D_{j,R} g_{ij,L} K_A\\
&\in  I^{2s+1,-s_0+\codimY/2}(N^*(\diag\cap(Y\times
X)),N^*\diag)\\
&\qquad+I^{-s_0+1-\dim Y/2,2s+n/2}(N^*(\diag\cap(Y\times
X)),N^*(Y\times X)),
\end{aligned}\end{equation*}
and so, in principle,
\begin{equation*}\begin{aligned}
&D_{i,L}D_{j,R} (g_{ij,L}-g_{ij,R}) K_A\\
&\in  I^{2s+1,-s_0+\codimY/2}(N^*(\diag\cap(Y\times 
Y)),N^*\diag)\\
&\qquad+I^{-s_0+1-\dim Y/2,2s+n/2}(N^*(\diag\cap(X\times 
Y)),N^*(X\times Y))\\
&\qquad+I^{-s_0+1-\dim Y/2,2s+n/2}(N^*(\diag\cap(Y\times
X)),N^*(Y\times X)).
\end{aligned}\end{equation*} 
However, by the standard pseudodifferential calculus, the principal
symbol on $N^*\diag$ in
$$
S^{2s+1}(I^{-s_0+\codimY/2-n/4})/S^{2s}(I^{-s_0+\codimY/2+1-n/4}),
$$
where we used short hand notation so that e.g.
$$
S^{2s+1}(I^{-s_0+\codimY/2-n/4}) =S^{2s+1}(\RR^n_\xi;I^{-s_0+\codimY/2-n/4}(\RR^n_x;N^*\{x'=0\})),
$$
vanishes since it is
given by (the equivalence class of) $\xi_i\xi_j g_{ij}(x) a(x,\xi)$ for both $D_{i,L}D_{j,R}
g_{ij,L} K_A$ and $D_{i,L}D_{j,R}
g_{ij,R} K_A$, so by Lemma~\ref{lemma:basic-symbol},
\begin{equation}\begin{aligned}\label{eq:bad-comm-orders}
&D_{i,L}D_{j,R} (g_{ij,L}-g_{ij,R}) K_A\\
&\in  I^{2s,-s_0+1+\codimY/2}(N^*(\diag\cap(Y\times 
Y)),N^*\diag)\\
&\qquad+I^{-s_0+1-\dim Y/2,2s+n/2}(N^*(\diag\cap(X\times 
Y)),N^*(X\times Y))\\
&\qquad+I^{-s_0+1-\dim Y/2,2s+n/2}(N^*(\diag\cap(Y\times
X)),N^*(Y\times X)).
\end{aligned}\end{equation} 
Thus, the only change compared to the previous calculations for
boundedness $H^{s-\ep_0}\to H^{-s+\ep_0}$ away from $N^*\diag$
is that \eqref{eq:good-interact-Sob-prelim} is replaced by
\begin{equation}\begin{aligned}\label{eq:bad-interact-Sob-prelim}
&-s_0+2s+1+\codimY/2< 2s-2\ep_0-\codimY/2\\
&-s_0+1-\dim Y/2<s-\ep_0-\frac{n}{2}\ \text{and}\\
&-s_0+2s+1+\codimY/2<s-\ep_0,
\end{aligned}\end{equation}
thus \eqref{eq:good-interact-Sob}
is replaced by
\begin{equation}\begin{aligned}\label{eq:bad-interact-Sob}
&\codimY+1+2\ep_0< s_0\\
&s>-s_0+\ep_0+1+\codimY/2\ \text{and}\\
&s<s_0-\ep_0-1-\codimY/2.
\end{aligned}\end{equation}
Note that these inequalities imply \eqref{eq:hyp-reduction-constraint}.
The first of these inequalities implies
$$
-s_0+\ep_0+1+\codimY/2<-\codimY/2-\ep_0,
$$
so again,  when $s>-\codimY/2$, the second inequality automatically holds
if the first holds.

Now, the actual argument proceeds as follows. We want to take $A\in\Psi^{2s-1}$ satisfying
\eqref{eq:commutant-symbol}; this requires $\etat$ and $\sigma_j$ to be
fixed; as in Section~\ref{sec:structure} we also use a positive
elliptic order $1$ symbol $\rho$.
As in Section~\ref{sec:structure}, we then actually arrange that $A$ is of the form
$\check A^2$, with $\check A\in\Psi^{s-1/2}$ formally self-adjoint;
and in fact we take
$$
A_r=\Lambda_r A\Lambda_r,\ \check A_r=\check A\Lambda_r,
$$
with $\Lambda_r$ as before with symbol $\phi_r$.
The functions $\etat$ and $\sigma_j$ depend on $p$ only via
\eqref{eq:Ham-vf-flow-relation} and
\eqref{eq:Ham-vf-localizer-relation}, both of which are purely
conditions at $\bar q$. Now, if $g_{ij}$ are conormal to $Y$ and $\bar
q\in T^*_Y X$, then for $s_0>1+\codimY$, $p$ is still $C^1$, and thus
$\sH_p$ is continuous, and is indeed $C^{\alpha_0}$, $\alpha_0<s_0-1-\codimY$.
Correspondingly, $\sH_p(\bar q)$ is
well-defined, and one can use it in the definition of the $\CI$ functions $\etat$ and
$\sigma_j$ on $S^*X$. However,
as $\sH_p\sigma$ is now $C^{\alpha_0}$, instead of
\eqref{eq:H_p-sigma-est}
one has for $\alpha=\min(1,\alpha_0)>0$,
\begin{equation}\label{eq:H_p-sigma-est-mod}
|\sH_p\sigma_j|\leq C_0(\omega^{1/2}+|\etat|)^\alpha,
\end{equation}
so
$|\sH_p\omega|\leq C \omega^{1/2}(\omega^{1/2}+|\etat|)^\alpha$. Using
\eqref{eq:supp-a-est}, we now deduce that $|\sH_p\omega|\leq \frac{c_0}{2}\ep^2\delta$
provided that $\frac{c_0}{2}\ep^2\delta\geq C''(\ep\delta)\delta^\alpha$, i.e.\ that
$\ep\geq C'\delta^{\alpha}$ for some constant $C'$ independent of
$\ep$, $\delta$. Taking $\ep\sim\delta^{\alpha}$,
the size of the parabola at $\etat=-\delta$ is roughly
$\omega^{1/2}\sim\delta^{1+\alpha}$, which still suffices for the
proof of propagation of singularities in view of
$\alpha>0$,
as we have localized along a single
direction, namely the direction of $\sH_p$ at $\bar q$.

We also assume at first (an assumption that will be eliminated by an iterative
procedure) that
\begin{equation}\label{eq:iterative-assump}
\WF'(A)\subset O,\ O=\WF^{s-1/2}(u)^c;
\end{equation}
note that given $O$ the $\delta$-localization of $a$ makes this achievable.

By construction, the
principal symbol of the commutator along the conormal bundle of the
diagonal, which is in
$S^{2s}(\RR^n_\xi;I^{-s_0+\codimY/2+1-n/4}(\RR^n_x;N^*\{x'=0\}))$,
is still of the form \eqref{eq:pos-comm-symbol}, though now
the symbols have a conormal singularity at $Y$ as well.
More precisely, with
$$
e_0=\rho^{2s}e\in S^{2s}(\RR^n_\xi;I^{-s_0+\codimY/2+1-n/4}(\RR^n_x;N^*\{x'=0\})),
$$
as in \eqref{eq:e-form} times the weight $\rho^{2s}$, cf.\ \eqref{eq:commutant-symbol},  and
$$
b_0=\rho^s b\in S^{s}(\RR^n_\xi;I^{-s_0+\codimY/2+1-n/4+\ep_1}(\RR^n_x;N^*\{x'=0\})),
$$
as in
\eqref{eq:b-form}  times the weight $\rho^{s}$ (getting $b_0$ to lie
in the indicated space uses
Lemma~\ref{lemma:sqrt}, applied to $\sH_p\phi$, which is bounded away
from $0$; this gives the loss of $\ep_1>0$ which one can take as small
as convenient, as we did in the elliptic setting),
one takes
$B,E$ paired Lagrangian associated to $N^*\diag$ and
$N^*(\diag\cap(Y\times Y))$ with principal symbols given by $b_0,e_0$
on
$N^*\diag$, more precisely
$$
B\in  I^{s,-s_0+1+\codimY/2+\ep_1}(N^*(\diag\cap(Y\times 
Y)),N^*\diag),
$$
and
$$
E\in I^{2s,-s_0+1+\codimY/2}(N^*(\diag\cap(Y\times 
Y)),N^*\diag)
$$
so they are in $\Psi^s$ and $\Psi^{2s}$ on $N^*\diag\setminus
N^*(\diag\cap(Y\times Y))$, and the orders on
$N^*(\diag\cap(Y\times Y))$ given by
$I^{[-s_0+s+1]}=I^{-s_0+s+1+\codimY/2+\ep_1}$ for $B$ and
$I^{[-s_0+s+1]}=I^{-s_0+2s+1+\codimY/2}$ for $E$; one can also arrange
(by applying a pseudodifferential operator microlocally the identity
near $N^*\diag$ but with wave front set in $O\times O'$)
that the Schwartz kernels of $B,E$ satisfy
$$
\WF'(K_B),\WF'(K_E)\subset O\times O',
$$
where $O'$ is the usual twisted version of $O$ (sign of the covector switched).
Then by
Proposition~\ref{prop:diag-main-compose}, taking into account that
$2(-s_0+1+\codimY/2)<-\codimY-4\ep_0<-1$ so there is a full order gain
in the symbolic calculation (if we take $\ep_1>0$ sufficiently small),
$$
{i}[ P,A]=B^*B+E+F,
$$ 
where away from $N^*(Y\times X)\cup N^*(X\times Y)$, at which $F$ has
the same orders as the commutator, as given in
\eqref{eq:good-comm-orders} and \eqref{eq:bad-comm-orders} (i.e.\ is
dictated by the second of these, as these are greater),
$$
F\in  I^{2s-1,-s_0+2+\codimY/2+\ep_1}(N^*(\diag\cap(Y\times 
Y)),N^*\diag).
$$
As in the elliptic setting, we break up $F$:
\begin{equation}\begin{aligned}\label{eq:hyp-error-space}
&F=F'+F'',\\
&F'\in  
I^{2s-1,-s_0+2+\codimY/2+\ep_1}(N^*(\diag\cap(X\times 
Y)),N^*\diag),\\
&F''\in 
I^{-s_0+1-\dim Y/2,2s+n/2+\ep_1}(N^*(\diag\cap(X\times 
Y)),N^*(X\times Y))\\
&\qquad+I^{-s_0+1-\dim Y/2,2s+n/2+\ep_1}(N^*(\diag\cap(Y\times
X)),N^*(Y\times X)),
\end{aligned}\end{equation}
with the wave front set of the Schwartz kernel of $F'$
in
$$
\WF(K_{F'})\subset O\times O';
$$
note that away from $N^*\diag$, elements of
$$
I^{2s-1,-s_0+2+\codimY/2+\ep_1}(N^*(\diag\cap(X\times 
Y)),N^*\diag)
$$
are in $I^{-s_0+1-\dim Y/2,2s+n/2+\ep_1}(N^*(\diag\cap(X\times 
Y)),N^*(X\times Y))$, so can always be regarded as part of $F''$.
With such a decomposition, for $\ep_1>0$ sufficiently small, in view of
Propositions~\ref{prop:one-sided-bded}
and \ref{prop:one-sided-bded-mod} and
Corollary~\ref{cor:one-sided-bded}, $F''$
is bounded from $H^{s-\ep_0}$ to $H^{-s+\ep_0}$ so $\langle
F''u,u\rangle$ is bounded. On the other hand
$F'$ is bounded $H^{s-1/2}$ to $H^{1/2-s}$ by
Proposition~\ref{prop:diag-main-bded},
and has wave front set in $O\times O'$, so $u$ being in
$H^{s-1/2}$ on $O$,
$\langle F'u,u\rangle$ is bounded by the a priori assumptions as well.
Thus, subject to \eqref{eq:bad-interact-Sob}, $\langle Eu,u\rangle$
and $\langle Fu,u\rangle$ are bounded by the a priori assumptions.

Further, as in Section~\ref{sec:structure}, if $ P u\in H^{s-1}_{\loc}(X)$, then with
$Q\in\Psi^{1/2}(X)$ elliptic with positive principal symbol $\rho^{1/2}$, with
parametrix $G\in\Psi^{-1/2}(X)$, such
that $GQ=\Id+R$, $R\in\Psi^{-\infty}(X)$, we use
\eqref{eq:break-up-inhomog} to control $|\langle A_r u, P
u\rangle|$.
In order to absorb the $Q\check A_r\in\Psi^{s}(X)$
term in \eqref{eq:break-up-inhomog}, and to deal with the regularizer
and the weight as in Section~\ref{sec:structure}, as well as to facilitate the direct translation to a wave front
set statement, we replace $B^*B$ by
$B_{1,r}^*B_{1,r}+B_{2,r}^* B_{2,r}+M^2(Q\check A_r)^*(Q\check A_r)$ where $M>0$ is a large constant,
\begin{equation*}\begin{split}
&B_{1,r}=B_1\Lambda_r,\ B_1\in I^s(N^*\diag)=\Psi^s,\\
&B_{2,r}\in I^{s,-s_0+1+\codimY/2+\ep_1}(N^*(\diag\cap(Y\times 
Y)),N^*\diag),
\end{split}\end{equation*}
$B_{2,r}$ uniformly bounded in $I^{s,-s_0+1+\codimY/2+\ep_1}(N^*(\diag\cap(Y\times 
Y)),N^*\diag)$ and with the Schwartz kernel of $B_{j,r}$ having (uniform) wave front set in
$O\times O'$.
To achieve this, we proceed as in
\eqref{eq:absorb-a-into-b}-\eqref{eq:reg-b-def}, and we recall
that we arranged that $\sH_p\phi\geq c_0/2$, and thus writing
$$
\sH_p\phi=\psi_1+\psi_2,\ \psi_1\equiv c_0/4,\ \psi_2\geq c_0/4,
$$
we let
\begin{equation}\label{eq:b_1-form} 
b_{1,r}=\rho^s\phi_r\digamma^{-1/2}\delta^{-1/2}\sqrt{\psi_1}\sqrt{\chi_0'\left(\digamma^{-1}\Big(2\beta-\frac{\phi}{\delta}\Big)\right)}\sqrt{\chi_1\left(\frac{\etat+\delta}{\ep\delta}+1\right)},
\end{equation} 
and
\begin{equation}\begin{aligned}\label{eq:b_2-form}
b_{2,r}=&\rho^s\phi_r \digamma^{-1/2}\delta^{-1/2}c_{2,r}\sqrt{\chi_0'\left(\digamma^{-1}\Big(2\beta-\frac{\phi}{\delta}\Big)\right)}
\sqrt{\chi_1\left(\frac{\etat+\delta}{\ep\delta}+1\right)},\\
c_{2,r}=&\left(\psi_2-\Big(\big((2s-1)-r\rho\phi_r\big)  (\rho^{-1}\sH_p\rho)+M^2\Big)
 \digamma^{-1}\delta\Big(2\beta-\frac{\phi}{\delta}\Big)^2\right)^{1/2}
\end{aligned}\end{equation}
and let $B_{j,r}$ have principal symbol $b_{j,r}$, noting that $b_{1,r}$ is
$\CI$ (i.e.\ does not have a conormal singularity).
As in Section~\ref{sec:structure}, the expression in the large
parentheses defining $c_{2,r}$ is bounded below by a positive constant
(uniformly in $r$)
for $\digamma>0$ sufficiently large as $|2\beta-\frac{\phi}{\delta}|\leq 4$
on $\supp a$.
Then the analogue of \eqref{eq:basic-pairing-mod} is
\begin{equation*}\begin{aligned}
\|B_{1,r}u\|^2+\|B_{2,r} u\|^2+M^2\|Q\check A_ru\|^2\leq 2|\langle A_ru, P u\rangle|+|\langle
E_ru,u\rangle|+|\langle F_ru,u\rangle|.
\end{aligned}\end{equation*}
Using \eqref{eq:break-up-inhomog} to estimate the first term on the
right hand side
from above, $\|Q\check A_ru\|^2$ can be absorbed in the $M^2\|Q\check A_ru\|^2$ term
on the left hand side (for $M>1$).
This
gives the conclusion that $B_{j,0} u\in L^2$ for $j=1,2$,
which allows us to conclude that $\WF^s(u)$ is
disjoint from the elliptic set of $B_{1,0}$.

What we have proved is the following:

\begin{lemma}\label{lemma:rough-prop}
Suppose that \eqref{eq:bad-interact-Sob} holds.
Let $\alpha=\min(1,\alpha_0)\in (0,1]$, $\alpha_0<s_0-1-k$, and let $U\subset X$ be
coordinate chart (identified with a subset of $\RR^n$).
For any $K\subset\Sigma\cap T^*_UX$
compact there exists $\delta_0>0$ and $C_0>0$
such that the following holds. If $u\in H^{s-\ep_0}_{\loc}$, $ P u\in H^{s-1}_{\loc}$,
$\delta\in(0,\delta_0)$ and
$q_0\in K$ and if the Euclidean metric ball around $q_0+\delta \sH_p(q_0)$ of
radius $C_0\delta^{1+\alpha}$ is disjoint from $\WF^s(u)$, and the Euclidean
metric tube (union of metric balls) around the straight line segment connecting $q_0$ and
$q_0+\delta \sH_p(q_0)$ of radius $C_0\delta^{1+\alpha}$ is disjoint from $\WF^{s-1/2}(u)$ then
$q_0\notin\WF^s(u)$.

The analogous conclusion also holds with $q_0+\delta \sH_p(q_0)$
replaced by $q_0-\delta \sH_p(q_0)$.
\end{lemma}

As in the elliptic case, one can eliminate the
background regularity assumption on the metric tube; here one needs to
proceed more directly, shrink the supports of the cutoffs defining $a$
slightly in each step of the iteration, as is standard, see
\cite[Section~24.5]{Hor}, last paragraph of the proof of
Proposition~24.5.1, and the end of Section~\ref{sec:structure}. The key point in starting the iteration is that
with $s'=\min(s-\ep_0+1/2,s)\leq s$, if $\codimY+1+2\ep_0< s_0$ and
$-\codimY/2<s$ then
$$
s'\geq s-\ep_0+1/2>-\codimY/2-\ep_0+1/2-(s_0-\codimY-1-2\ep_0)
>-s_0+\codimY/2+\ep_0+1,
$$
so the second inequality in \eqref{eq:bad-interact-Sob} holds; all others
follow at once from those of $s$ using $s'\leq s$.

\begin{prop}\label{prop:rough-prop}
Suppose that $\codimY+1+2\ep_0< s_0$ and $-\codimY/2<s<s_0-\ep_0-1-\codimY/2$.
Let $\alpha=\min(1,\alpha_0)\in (0,1]$, $\alpha_0<s_0-1-k$, and let $U\subset X$ be
coordinate chart (identified with a subset of $\RR^n$).
For any $K\subset\Sigma\cap T^*_UX$
compact there exists $\delta_0>0$ and $C_0>0$
such that the following holds. If $u\in H^{s-\ep_0}_{\loc}$, $\Box u\in H^{s-1}_{\loc}$,
$\delta\in(0,\delta_0)$ and
$q_0\in K$ and if the metric ball around $q_0+\delta \sH_p(q_0)$ of
radius $C_0\delta^{1+\alpha}$ is disjoint from $\WF^s(u)$ then
$q_0\notin\WF^s(u)$.

The analogous conclusion also holds with $q_0+\delta \sH_p(q_0)$
replaced by $q_0-\delta \sH_p(q_0)$.
\end{prop}

\section{Propagation of singularities}\label{sec:prop-sing}

In order to convert Proposition~\ref{prop:rough-prop} into a
propagation of singularities along bicharacteristics statement, we
need a more precise analysis of the bicharacteristics.
One has the following lemma, which
is just a version of the argument of Melrose and Sj\"ostrand \cite{Melrose-Sjostrand:I,
Melrose-Sjostrand:II}, see also \cite[Chapter~XXIV]{Hor} and
\cite{Lebeau:Propagation}.

\begin{lemma}(Version of \cite[Lemma~24.3.15]{Hor}.)
Suppose that $\alpha\in(0,1]$ and $\sH_p$ is in $C^\alpha$.
Suppose that $F$ is a closed subset of $\Sigma$ with the property that
for every $U\subset X$ coordinate chart and for every $K\subset\Sigma\cap T^*_UX$
compact there exists $\delta_0>0$ and $C_0>0$ such that for all $t\in(-\delta_0,\delta_0)\setminus\{0\}$ and
$q_0\in K\cap F$ there exists $q=q(t,q_0)\in F$
in the metric ball $B(q_0+t\sH_p(q_0),C_0 |t|^{1+\alpha})$
around $q_0+t\sH_p(q_0)$ of radius
$C_0|t|^{1+\alpha}$. Then for every $q_0\in F$ there is a
bicharacteristic $\gamma:(t_-,t_+)\to F$ with $\gamma(0)=q_0$ and such
that $\gamma$ leaves every compact subset of $F$ when $t\to t_\pm$.
\end{lemma}

\begin{proof}
One can follow the proof of \cite[Lemma~24.3.15]{Hor} quite closely,
ignoring case (i). Here we present a slightly different version of the
argument, following \cite{Lebeau:Propagation}, see also
\cite[Proof of Theorem~8.1]{Vasy:Propagation-Wave}.

A standard argument based on Zorn's lemma shows that it suffices to
prove the local assertion that for every
$q_0\in F$
there exists a bicharacteristic
$\gamma:[-\ep,\ep]\to\Sigma$, $\ep>0$,
with $\gamma(0)=q_0$ and such that
$\gamma(t)\in F$ for $t\in[-\ep,\ep]$.
Indeed, it suffices to do a one-sided version, i.e.\ that if
$q_0\in F$
then
\begin{equation}\label{eq:prop-104}\begin{split}
&\text{there exists a bicharacteristic}
\ \gamma:[-\ep,0]\to\Sigma,\ \ep>0,\\
&\qquad\qquad \gamma(0)=q_0,
\ \gamma(t)\in F,\ t\in[-\ep,0],
\end{split}\end{equation}
for the existence of a bicharacteristic
on $[0,\ep]$ can be demonstrated similarly
by replacing the forward propagation
estimates by backward ones, and piecing together the two
bicharacteristics $\gamma_\pm$ gives one defined on $[-\ep,\ep]$ since at $0$
they both satisfy $\frac{d}{dt}\gamma_\pm(0)=\sH_p (q_0)$, so the
curve defined on $[-\ep,\ep]$ is $C^1$ with the correct derivative
everywhere.

Let $\cU$
be a neighborhood of $q_0$ with $\overline{\cU}\subset T^*_U X$ so
$\sH_p$ is H\"older-$\alpha$ in $\overline{\cU}$, and is in particular
bounded; $\sup\|\sH_p\|\leq C'$. Let $\cU_0$ be a smaller neighborhood with closure
in $\cU$ and (with $\delta_0$ as in Proposition~\ref{prop:rough-prop})
$\ep\in (0,\delta_0)$ such that for any $q\in\cU_0$, $\|q'-q\|\leq
(C'+C_0\ep^\alpha)\ep$ implies $q'\in\cU$.
Suppose that
$0<\delta<\ep$, $q\in \cU_0$.
For $q\in T^*X$, let
\begin{equation}
D(q,\delta)=B(q-\delta \sH_p(q),C_0\delta^{1+\alpha})\cap F.
\end{equation}
For each integer $N\geq 1$ now we define a sequence of $2^N+1$ points
$q_{j,N}$,
$0\leq j\leq 2^N$ integer,
which will be used to construct points $\gamma(-j 2^{-N}\ep)$
on the desired
bicharacteristic $\gamma:[-\ep,0]\to F$ through $q_0$.
Namely, let $\delta=2^{-N}\ep$, $q_{0,N}=q_0$, and choose $q_{j+1,N}
\in D(q_{j,N},\delta)$; such $q_{j+1,N}$ exists by assumption.
Here one needs to check that $q_{j,N}\in\cU$ inductively for $0\leq j\leq 2^N$, but this
follows as
\begin{equation}\begin{aligned}\label{eq:rough-bich-bd-est}
\|q_{j,N}-q_0\|&\leq \sum_{i=0}^{j-1}\|q_{i+1,N}-q_i\|\\
&\leq j(C'
2^{-N}\ep+C_0(2^{-N}\ep)^{1+\alpha})
\leq C'\ep+C_0 2^{-\alpha N}\ep^{1+\alpha}.
\end{aligned}\end{equation}
Let
$\gamma_N:[-\ep,0]$ be the curve defined by $\gamma_N(t)=q_{j,N}$
for $t=-j 2^{-N}\ep$, with $\gamma$ given by the straight line
between successive dyadic points. Thus, by an estimate similar to
\eqref{eq:rough-bich-bd-est},
$\gamma_N$ is a uniformly Lipschitz family with
$$
\|\gamma_N(t)-\gamma_N(t')\|\leq (C'+C_0\ep^\alpha)|t-t'|,
$$
and thus there is a subsequence $\gamma_{N_k}$
converging uniformly to some $\gamma$; as $F$ is closed, $\gamma$
takes values in $F$. It remains to check the differentiability of
$\gamma$, and that $\frac{d}{dt}\gamma(t)=\sH_p(\gamma(t))$. For this
it suffices to show that there is $\tilde C_0>0$ such that for all
relevant $t$ and $\delta$,
$$
\gamma(t+\delta)\in B(\gamma(t)+\delta\sH_p(\gamma(t)),\tilde C_0|\delta|^{1+\alpha}),
$$
which follows if we show the analogous statement for $\gamma_N$ (with
constant $\tilde C_0$ independent of $N$) when $t$ and $t+\delta$ are
both dyadic points (so $\delta=-k\ep 2^{-N}$ is such as well). This is straightforward to
check from the definition of $\gamma_N$ since, with $C_\alpha$ the
H\"older-$\alpha$ constant of $\sH_p$ on $\overline{\cU}$, so $\|\sH_p(q)-\sH_p(q')\|\leq C_\alpha\|q-q'\|^\alpha$,
\begin{equation*}\begin{aligned}
\|&\gamma_N(t-k\ep 2^{-N})-\gamma_N(t)+k\ep 2^{-N}\sH_p(\gamma_N(t))\|\\
&\leq\sum_{j=0}^{k-1}\|\gamma_N(t-(j+1)\ep 2^{-N})-\gamma_N(t-j\ep 2^{-N})+\ep
2^{-N}\sH_p(\gamma_N(t-j\ep
2^{-N}))\|\\
&\qquad+\sum_{j=0}^{k-1}\ep
2^{-N}\|\sH_p(\gamma_N(t-j\ep 2^{-N}))-\sH_p(\gamma_N(t))\|\\
&\leq \sum_{j=0}^{k-1}C_0(\ep
2^{-N})^{1+\alpha}+\sum_{j=0}^{k-1}C_\alpha \ep 2^{-N} (j\ep
2^{-N})^\alpha\\
&\leq (kC_0+\frac{C_\alpha}{1+\alpha}k^{1+\alpha})(\ep
2^{-N})^{1+\alpha}\leq (C_0+\frac{C_\alpha}{1+\alpha})(k\ep 2^{-N})^{1+\alpha},
\end{aligned}\end{equation*}
which gives the desired estimate with $\tilde C_0=C_0+\frac{C_\alpha}{1+\alpha}$.
\end{proof}

Applying the lemma with $F=\WF^s(u)$,
Proposition~\ref{prop:rough-prop} implies
Theorem~\ref{thm:global-background}, which we restate as a corollary:

\begin{cor}\label{cor:global-background}
Suppose that
$\codimY+1+2\ep_0< s_0$ and $-\codimY/2<s<s_0-\ep_0-1-\codimY/2$. Then
for $u\in H^{s-\ep_0}_{\loc}$, $\Box u\in H^{s-1}_{\loc}$, $\WF^s(u)$ is a union of maximally extended
bicharacteristics in $\Sigma$.
\end{cor}

A corollary of Theorem~\ref{thm:global-background} is the following
global regularity result:

\begin{cor}\label{cor:global-reg}
If $s_0>1+\codimY/2$, $-\codimY/2<s'<s<s_0-1-\codimY/2$, $u\in
H^{s'}_{\loc}$, $\Box u\in H^{s-1}_{\loc}$ and for each
$q\in\Sigma$ the bicharacteristic through $q$ has a point $q'$ on it
which is not in $\WF^{s}(u)$, then $u\in H^s_{\loc}$.
\end{cor}

\begin{proof}
By microlocal elliptic regularity which is valid with this $s$,
$\WF^s(u)\subset\Sigma$.
Now let $\ep_0=\min((s_0-\codimY-1)/2,s_0-1-\codimY/2-s)/2>0$. Then
for $s'\leq \tilde s\leq s$, the hypotheses of
Corollary~\ref{cor:global-background}, apart from possibly $u\in H^{\tilde
  s-\ep_0}$,
are satisfied with $s$ replaced
by $\tilde s$ and with this $\ep_0$. Thus, taking $\tilde
s=\min(s,s'+\ep_0)$, all hypotheses are satisfied, so as a point on
any bicharacteristic is not in $\WF^s(u)$ and thus not in $\WF^{\tilde
  s}(u)$, one concludes that $\WF^{\tilde s}(u)=\emptyset$, i.e.\
$u\in H^{\tilde s}_{\loc}$. If $\tilde s=s$, we are done, otherwise we
have $u\in H^{s'+\ep_0}_{\loc}$
repeat the argument, with $\tilde s=\min(s,s'+2\ep_0)$; in finite
number of steps we conclude that $u\in H^s_{\loc}$.
\end{proof}

A further consequence is:

\begin{cor}\label{cor:forward-solution-reg}
Suppose $s_0>1+\codimY/2$, $0\leq s<s_0-1-\codimY/2$.
Let $\Box_+^{-1}f\in H_{\bl,\loc}^{1,-\infty}(X)$ denote the
forward solution for $\Box u=f$, i.e.\ for $f\in H_{\bl,\loc}^{-1,-\infty}(X)$ supported in
$t>t_0$, $u=\Box_+^{-1}f$ is supported in $t>t_0$.

If $f\in
H^{s-1}_{\loc}$ is supported in $t>t_0$, then $u=\Box_+^{-1}f\in
H^{s}_{\loc}$.

An analogous result holds with $\Box_+^{-1}$ replaced by the backward
  solution operator $\Box_-^{-1}$ and $t>t_0$ replaced by $t<t_0$.
\end{cor}

\begin{proof}
First we note $f\in H^{s-1}_{\loc}(X)$ implies $f\in
H_{\bl,\loc}^{-1,s}(X)$, and thus $u=\Box_+^{-1}f\in
H_{\bl,\loc}^{1,s-1}(X)\subset L^2(X)$.
Then
we merely need to observe that every bicharacteristic reaches $t<t_0$,
where $u$ vanishes, thus is in $H^s_{\loc}$, so
Corollary~\ref{cor:global-reg} is applicable with $s'=0$ and yields the conclusion.
\end{proof}

\bibliographystyle{plain}
\bibliography{sm}

\def\cprime{$'$} \def\cprime{$'$}
\begin{thebibliography}{10}

\bibitem{Antoniano-Uhlmann:Functional}
Jos{\'e}~L. Antoniano and Gunther~A. Uhlmann.
\newblock A functional calculus for a class of pseudodifferential operators
  with singular symbols.
\newblock In {\em Pseudodifferential operators and applications ({N}otre
  {D}ame, {I}nd., 1984)}, volume~43 of {\em Proc. Sympos. Pure Math.}, pages
  5--16. Amer. Math. Soc., Providence, RI, 1985.

\bibitem{Geba-Tataru:Phase}
Dan-Andrei Geba and Daniel Tataru.
\newblock A phase space transform adapted to the wave equation.
\newblock {\em Comm. Partial Differential Equations}, 32(7-9):1065--1101, 2007.

\bibitem{Greenleaf-Uhlmann:Estimates}
Allan Greenleaf and Gunther Uhlmann.
\newblock Estimates for singular {R}adon transforms and pseudodifferential
  operators with singular symbols.
\newblock {\em J. Funct. Anal.}, 89(1):202--232, 1990.

\bibitem{Greenleaf-Uhlmann:Recovering}
Allan Greenleaf and Gunther Uhlmann.
\newblock Recovering singularities of a potential from singularities of
  scattering data.
\newblock {\em Comm. Math. Phys.}, 157(3):549--572, 1993.

\bibitem{Guillemin-Uhlmann:Oscillatory}
V.~Guillemin and G.~Uhlmann.
\newblock Oscillatory integrals with singular symbols.
\newblock {\em Duke Math. J.}, 48(1):251--267, 1981.

\bibitem{FIO1}
L.~H\"ormander.
\newblock Fourier integral operators, {I}.
\newblock {\em Acta Mathematica}, 127:79--183, 1971.

\bibitem{Hor}
L.~H\"ormander.
\newblock {\em The analysis of linear partial differential operators, {\rm vol.
  1-4}}.
\newblock Springer-Verlag, 1983.

\bibitem{Hormander:Existence}
Lars H{\"o}rmander.
\newblock On the existence and the regularity of solutions of linear
  pseudo-differential equations.
\newblock {\em Enseignement Math. (2)}, 17:99--163, 1971.

\bibitem{Lebeau:Propagation}
G.~Lebeau.
\newblock Propagation des ondes dans les vari\'et\'es \`a coins.
\newblock {\em Ann. Scient. \'Ec. Norm. Sup.}, 30:429--497, 1997.

\bibitem{Melrose-Sjostrand:I}
R.~B. Melrose and J.~Sj{\"o}strand.
\newblock Singularities of boundary value problems. {I}.
\newblock {\em Comm. Pure Appl. Math}, 31:593--617, 1978.

\bibitem{Melrose-Sjostrand:II}
R.~B. Melrose and J.~Sj{\"o}strand.
\newblock Singularities of boundary value problems. {II}.
\newblock {\em Comm. Pure Appl. Math}, 35:129--168, 1982.

\bibitem{Melrose-Uhlmann:Intersection}
R.~B. Melrose and G.~A. Uhlmann.
\newblock Lagrangian intersection and the {C}auchy problem.
\newblock {\em Comm. Pure and Appl. Math.}, 32:483--519, 1979.

\bibitem{Melrose:Transformation}
Richard~B. Melrose.
\newblock Transformation of boundary problems.
\newblock {\em Acta Math.}, 147(3-4):149--236, 1981.

\bibitem{Melrose:Atiyah}
Richard~B. Melrose.
\newblock {\em The {A}tiyah-{P}atodi-{S}inger index theorem}, volume~4 of {\em
  Research Notes in Mathematics}.
\newblock A K Peters Ltd., Wellesley, MA, 1993.

\bibitem{Melrose-Piazza:Analytic}
Richard~B. Melrose and Paolo Piazza.
\newblock Analytic {$K$}-theory on manifolds with corners.
\newblock {\em Adv. Math.}, 92(1):1--26, 1992.

\bibitem{Smith:Parametrix}
Hart~F. Smith.
\newblock A parametrix construction for wave equations with {$C^{1,1}$}
  coefficients.
\newblock {\em Ann. Inst. Fourier (Grenoble)}, 48(3):797--835, 1998.

\bibitem{Tataru:Strichartz}
Daniel Tataru.
\newblock Strichartz estimates for operators with nonsmooth coefficients and
  the nonlinear wave equation.
\newblock {\em Amer. J. Math.}, 122(2):349--376, 2000.

\bibitem{Vasy:Propagation-Wave}
A.~Vasy.
\newblock Propagation of singularities for the wave equation on manifolds with
  corners.
\newblock {\em Annals of Mathematics}, 168:749--812, 2008.

\bibitem{Vasy:Maxwell}
A.~Vasy.
\newblock Diffraction at corners for the wave equation on differential forms.
\newblock {\em Commun. in PDEs}, 35:1236--1275, 2010.

\bibitem{Vasy:Geometric-optics}
Andr{\'a}s Vasy.
\newblock Geometric optics and the wave equation on manifolds with corners.
\newblock In {\em Recent advances in differential equations and mathematical
  physics}, volume 412 of {\em Contemp. Math.}, pages 315--333. Amer. Math.
  Soc., Providence, RI, 2006.

\end{thebibliography}

\end{document}